\documentclass[10pt]{amsart}
\usepackage{amsthm}
\usepackage{amsmath}
\usepackage{amssymb}
\usepackage{stmaryrd}
\usepackage{verbatim}
\usepackage{mathrsfs}
\usepackage{eucal}
\usepackage{color}
\let\mathcal=\CMcal
\usepackage[dvipdfm,bookmarks=true,bookmarksnumbered=true]{hyperref}

\allowdisplaybreaks[4]

\begin{document}
\def\sbt{\raisebox{1.2pt}{$\scriptscriptstyle\,\bullet\,$}}

\def\alp{\alpha}
\def\bet{\beta}
\def\gam{\gamma}
\def\del{\delta}
\def\eps{\epsilon}
\def\zet{\zeta}
\def\tht{c}
\def\iot{\iota}
\def\kap{\kappa}
\def\lam{\lambda}
\def\sig{\sigma}
\def\ome{\omega}
\def\vep{\varepsilon}
\def\vth{\vartheta}
\def\vpi{\varpi}
\def\vrh{\varrho}
\def\vsi{\varsigma}
\def\vph{\varphi}
\def\Gam{\Gamma}
\def\Del{\Delta}
\def\Tht{\Theta}
\def\Lam{\Lambda}
\def\Sig{\Sigma}
\def\Ups{\Upsilon}
\def\Ome{\Omega}
\def\vka{\varkappa}
\def\vDe{\varDelta}
\def\vSi{\varSigma}
\def\vTh{\varTheta}
\def\vGm{\varGamma}
\def\vOm{\varOmega}
\def\vPi{\varPi}
\def\vPh{\varPhi}
\def\vPs{\varPsi}
\def\vUp{\varUpsilon}
\def\vXi{\varXi}

\def\bdse{\boldsymbol{e}}
\def\bdsf{\boldsymbol{f}}
\def\bdsg{\boldsymbol{g}}
\def\bdsh{\boldsymbol{h}}
\def\bdsF{\boldsymbol{F}}
\def\bdscF{\boldsymbol{\calf}}
\def\bdscU{\boldsymbol{\calu}}

\def\bdsig{\boldsymbol{\sig}}
\def\bdpi{\boldsymbol{\pi}}
\def\bdome{\boldsymbol{\ome}}
\def\bdkap{\boldsymbol{\kap}}
\def\ulQ{\underline{Q}}
\def\ulk{\underline{k}}
\def\uleps{\underline{\eps}}
\def\ulchi{\underline{\chi}}
\def\ulkap{\underline{\kap}}

\def\frka{{\mathfrak a}}    \def\frkA{{\mathfrak A}}
\def\frkb{{\mathfrak b}}    \def\frkB{{\mathfrak B}}
\def\frkc{{\mathfrak c}}    \def\frkC{{\mathfrak C}}
\def\frkd{{\mathfrak d}}    \def\frkD{{\mathfrak D}}
\def\frke{{\mathfrak e}}    \def\frkE{{\mathfrak E}}
\def\frkf{{\mathfrak f}}    \def\frkF{{\mathfrak F}}
\def\frkg{{\mathfrak g}}    \def\frkG{{\mathfrak G}}
\def\frkh{{\mathfrak h}}    \def\frkH{{\mathfrak H}}
\def\frki{{\mathfrak i}}    \def\frkI{{\mathfrak I}}
\def\frkj{{\mathfrak j}}    \def\frkJ{{\mathfrak J}}
\def\frkk{{\mathfrak k}}    \def\frkK{{\mathfrak K}}
\def\frkl{{\mathfrak l}}    \def\frkL{{\mathfrak L}}
\def\frkm{{\mathfrak m}}    \def\frkM{{\mathfrak M}}
\def\frkn{{\mathfrak n}}    \def\frkN{{\mathfrak N}}
\def\frko{{\mathfrak o}}    \def\frkO{{\mathfrak O}}
\def\frkp{{\mathfrak p}}    \def\frkP{{\mathfrak P}}
\def\frkq{{\mathfrak q}}    \def\frkQ{{\mathfrak Q}}
\def\frkr{{\mathfrak r}}    \def\frkR{{\mathfrak R}}
\def\frks{{\mathfrak s}}    \def\frkS{{\mathfrak S}}
\def\frkt{{\mathfrak t}}    \def\frkT{{\mathfrak T}}
\def\frku{{\mathfrak u}}    \def\frkU{{\mathfrak U}}
\def\frkv{{\mathfrak v}}    \def\frkV{{\mathfrak V}}
\def\frkw{{\mathfrak w}}    \def\frkW{{\mathfrak W}}
\def\frkx{{\mathfrak x}}    \def\frkX{{\mathfrak X}}
\def\frky{{\mathfrak y}}    \def\frkY{{\mathfrak Y}}
\def\frkz{{\mathfrak z}}    \def\frkZ{{\mathfrak Z}}

\def\cal{\fam2}
\def\cala{{\cal A}}
\def\calb{{\cal B}}
\def\calc{{\cal C}}
\def\cald{{\cal D}}
\def\cale{{\cal E}}
\def\calf{{\cal F}}
\def\calg{{\cal G}}
\def\calh{{\cal H}}
\def\cali{{\cal I}}
\def\calj{{\cal J}}
\def\calk{{\cal K}}
\def\call{{\cal L}}
\def\calm{{\cal M}}
\def\caln{{\cal N}}
\def\calo{{\cal O}}
\def\calp{{\cal P}}
\def\calq{{\cal Q}}
\def\calr{{\cal R}}
\def\cals{{\cal S}}
\def\calt{{\cal T}}
\def\calu{{\cal U}}
\def\calv{{\cal V}}
\def\calw{{\cal W}}
\def\calx{{\cal X}}
\def\caly{{\cal Y}}
\def\calz{{\cal Z}}

\def\AA{{\mathbb A}}
\def\BB{{\mathbb B}}
\def\CC{{\mathbb C}}
\def\DD{{\mathbb D}}
\def\EE{{\mathbb E}}
\def\FF{{\mathbb F}}
\def\GG{{\mathbb G}}
\def\HH{{\mathbb H}}
\def\II{{\mathbb I}}
\def\JJ{{\mathbb J}}
\def\KK{{\mathbb K}}
\def\LL{{\mathbb L}}
\def\MM{{\mathbb M}}
\def\NN{{\mathbb N}}
\def\OO{{\mathbb O}}
\def\PP{{\mathbb P}}
\def\QQ{{\mathbb Q}}
\def\RR{{\mathbb R}}
\def\SS{{\mathbb S}}
\def\TT{{\mathbb T}}
\def\UU{{\mathbb U}}
\def\VV{{\mathbb V}}
\def\WW{{\mathbb W}}
\def\XX{{\mathbb X}}
\def\YY{{\mathbb Y}}
\def\ZZ{{\mathbb Z}}

\def\bfa{{\mathbf a}}    \def\bfA{{\mathbf A}}
\def\bfb{{\mathbf b}}    \def\bfB{{\mathbf B}}
\def\bfc{{\mathbf c}}    \def\bfC{{\mathbf C}}
\def\bfd{{\mathbf d}}    \def\bfD{{\mathbf D}}
\def\bfe{{\mathbf e}}    \def\bfE{{\mathbf E}}
\def\bff{{\mathbf f}}    \def\bfF{{\mathbf F}}
\def\bfg{{\mathbf g}}    \def\bfG{{\mathbf G}}
\def\bfh{{\mathbf h}}    \def\bfH{{\mathbf H}}
\def\bfi{{\mathbf i}}    \def\bfI{{\mathbf I}}
\def\bfj{{\mathbf j}}    \def\bfJ{{\mathbf J}}
\def\bfk{{\mathbf k}}    \def\bfK{{\mathbf K}}
\def\bfl{{\mathbf l}}    \def\bfL{{\mathbf L}}
\def\bfm{{\mathbf m}}    \def\bfM{{\mathbf M}}
\def\bfn{{\mathbf n}}    \def\bfN{{\mathbf N}}
\def\bfo{{\mathbf o}}    \def\bfO{{\mathbf O}}
\def\bfp{{\mathbf p}}    \def\bfP{{\mathbf P}}
\def\bfq{{\mathbf q}}    \def\bfQ{{\mathbf Q}}
\def\bfr{{\mathbf r}}    \def\bfR{{\mathbf R}}
\def\bfs{{\mathbf s}}    \def\bfS{{\mathbf S}}
\def\bft{{\mathbf t}}    \def\bfT{{\mathbf T}}
\def\bfu{{\mathbf u}}    \def\bfU{{\mathbf U}}
\def\bfv{{\mathbf v}}    \def\bfV{{\mathbf V}}
\def\bfw{{\mathbf w}}    \def\bfW{{\mathbf W}}
\def\bfx{{\mathbf x}}    \def\bfX{{\mathbf X}}
\def\bfy{{\mathbf y}}    \def\bfY{{\mathbf Y}}
\def\bfz{{\mathbf z}}    \def\bfZ{{\mathbf Z}}

\def\scra{{\mathscr A}}
\def\scrb{{\mathscr B}}
\def\scrc{{\mathscr C}}
\def\scrd{{\mathscr D}}
\def\scre{{\mathscr E}}
\def\scrf{{\mathscr F}}
\def\scrg{{\mathscr G}}
\def\scrh{{\mathscr H}}
\def\scri{{\mathscr I}}
\def\scrj{{\mathscr J}}
\def\scrk{{\mathscr K}}
\def\scrl{{\mathscr L}}
\def\scrm{{\mathscr M}}
\def\scrn{{\mathscr N}}
\def\scro{{\mathscr O}}
\def\scrp{{\mathscr P}}
\def\scrq{{\mathscr Q}}
\def\scrr{{\mathscr R}}
\def\scrs{{\mathscr S}}
\def\scrt{{\mathscr T}}
\def\scru{{\mathscr U}}
\def\scrv{{\mathscr V}}
\def\scrw{{\mathscr W}}
\def\scrx{{\mathscr X}}
\def\scry{{\mathscr Y}}
\def\scrz{{\mathscr Z}}

\def\phm{\phantom}
\def\smallstrut{\vphantom{\vrule height 3pt }}
\def\bdm #1#2#3#4{\left(
\begin{array} {c|c}{\ds{#1}}
 & {\ds{#2}} \\ \hline
{\ds{#3}\vphantom{\ds{#3}^1}} &  {\ds{#4}}
\end{array}
\right)}
\newcommand{\powerseries}[1]{\llbracket{#1}\rrbracket}

\def\GL{\mathrm{GL}}
\def\BC{\mathrm{BC}}
\def\PGL{\mathrm{PGL}}
\def\PG{\mathrm{P}G}
\def\SL{\mathrm{SL}}
\def\Mp{\mathrm{Mp}}
\def\GSp{\mathrm{GSp}}
\def\PGSp{\mathrm{PGSp}}
\def\Sp{\mathrm{Sp}}
\def\St{\mathrm{St}}
\def\SU{\mathrm{SU}}
\def\SO{\mathrm{SO}}
\def\U{\mathrm{U}}
\def\GU{\mathrm{GU}}
\def\O{\mathrm{O}}
\def\Mat{\mathrm{M}}
\def\Tr{\mathrm{Tr}}
\def\Nr{\mathrm{N}}
\def\tr{\mathrm{tr}}
\def\Ad{\mathrm{Ad}}
\def\As{\mathrm{As}}
\def\Paf{\mathrm{Paf}}
\def\new{\mathrm{new}}
\def\Wh{\mathrm{Wh}}
\def\FJ{\mathrm{FJ}}
\def\Fj{\mathrm{Fj}}
\def\Sym{\mathrm{Sym}}
\def\Her{\mathrm{Her}}
\def\sym{\mathrm{sym}}
\def\supp{\mathrm{supp}}
\def\proj{\mathrm{proj}}
\def\Hom{\mathrm{Hom}}
\def\End{\mathrm{End}}
\def\Aut{\mathrm{Aut}}
\def\Ker{\mathrm{Ker}}
\def\Res{\mathrm{Res}}
\def\res{\mathrm{res}}
\def\cusp{\mathrm{cusp}}
\def\Irr{\mathrm{Irr}}
\def\rank{\mathrm{rank}}
\def\sgn{\mathrm{sgn}}
\def\diag{\mathrm{diag}}
\def\Wd{\mathrm{Wd}}
\def\nd{\mathrm{nd}}
\def\ord{\mathrm{ord}}
\def\hol{\mathrm{hol}}
\def\rec{\mathrm{rec}}
\def\ar{\mathrm{ar}}
\def\d{\mathrm{d}}
\def\Gal{\mathrm{Gal}}
\def\frkpbar{{\overline{\frkp}}}
\def\La{\langle}
\def\Ra{\rangle}
\newcommand{\one}{1\hspace{-0.25em}{\rm{l}}}
\def\addchar{{\boldsymbol\psi}}
\def\Abs{{\boldsymbol\alp}}
\def\bvep{{\boldsymbol \vep}}
\def\bnu{{\boldsymbol \nu}}
\def\bj{{\boldsymbol j}}
\def\cyc{\boldsymbol\varepsilon_{\rm cyc}}
\def\wh{\widehat}

\def\trs{\,^t\!}
\def\tri{\,^\iot\!}
\def\iu{\sqrt{-1}}
\def\oo{\hbox{\bf 0}}
\def\ono{\hbox{\bf 1}}
\def\smallcirc{\lower .3em \hbox{\rm\char'27}\!}
\def\AAf{\AA_\bff}
\def\thalf{\tfrac{1}{2}}
\def\bsl{\backslash}
\def\wtl{\widetilde}
\def\til{\tilde}
\def\Ind{\operatorname{Ind}}
\def\ind{\operatorname{ind}}
\def\cind{\operatorname{c-ind}}
\def\beq{\begin{equation}}
\def\eeq{\end{equation}}
\def\d{\mathrm{d}}
\def\lddots{\mathinner{\mskip1mu\raise1pt\vbox{\kern7pt\hbox{.}}\mskip2mu\raise4pt\hbox{.}\mskip2mu\raise7pt\hbox{.}\mskip1mu}}
\newcommand{\1}{1\hspace{-0.25em}{\rm{l}}}

\newcounter{one}
\setcounter{one}{1}
\newcounter{two}
\setcounter{two}{2}
\newcounter{thr}
\setcounter{thr}{3}
\newcounter{fou}
\setcounter{fou}{4}
\newcounter{fiv}
\setcounter{fiv}{5}
\newcounter{six}
\setcounter{six}{6}
\newcounter{sev}
\setcounter{sev}{7}

\newcommand{\shp}{\rm\char'43}

\def\lddots{\mathinner{\mskip1mu\raise1pt\vbox{\kern7pt\hbox{.}}\mskip2mu\raise4pt\hbox{.}\mskip2mu\raise7pt\hbox{.}\mskip1mu}}

\makeatletter
\def\varddots{\mathinner{\mkern1mu
    \raise\p@\hbox{.}\mkern2mu\raise4\p@\hbox{.}\mkern2mu
    \raise7\p@\vbox{\kern7\p@\hbox{.}}\mkern1mu}}
\makeatother

\def\today{\ifcase\month\or
 January\or February\or March\or April\or May\or June\or
 July\or August\or September\or October\or November\or December\fi
 \space\number\day, \number\year}

\makeatletter
\def\varddots{\mathinner{\mkern1mu
    \raise\p@\hbox{.}\mkern2mu\raise4\p@\hbox{.}\mkern2mu
    \raise7\p@\vbox{\kern7\p@\hbox{.}}\mkern1mu}}
\makeatother

\def\today{\ifcase\month\or
 January\or February\or March\or April\or May\or June\or
 July\or August\or September\or October\or November\or December\fi
 \space\number\day, \number\year}

\makeatletter\@addtoreset{equation}{section}\def\theequation{\thesection.\arabic{equation}}

\theoremstyle{plain}
\newtheorem{theorem}{Theorem}[section]
\newtheorem*{theorem_a}{Theorem A}
\newtheorem*{theorem_b}{Theorem B}
\newtheorem*{theorem_c}{Theorem C}
\newtheorem*{theorem_e}{Theorem}
\newtheorem*{corollary_a}{Corollary A}
\newtheorem{lemma}[theorem]{Lemma}
\newtheorem{proposition}[theorem]{Proposition}
\theoremstyle{definition}
\newtheorem{definition}[theorem]{Definition}
\newtheorem{conjecture}[theorem]{Conjecture}
\newtheorem{corollary}[theorem]{Corollary}
\newtheorem{Remark}[theorem]{Remark}
\theoremstyle{remark}
\newtheorem{remark}[theorem]{Remark}
\newtheorem*{main_remark}{Remark}

\renewcommand{\thepart}{\Roman{part}}
\setcounter{tocdepth}{1}
\setcounter{section}{0} 

\def\Spec{{\rm Spec}\,}

\title{Five-variable $p$-adic $L$-functions for $\U(3)\times\U(2)$}
\author{Ming-Lun Hsieh}
\author{Shunsuke Yamana}
\date{\today}
\address{Department of Mathematics and NCTS, National Taiwan University, No. 1, Sec. 4, Roosevelt Road, Taipei 10617, Taiwan}
\email{mlhsieh@math.ntu.edu.tw}
\address{Department of Mathematics, Graduate School of Science, Osaka Metropolitan University, 3-3-138 Sugimoto, Sumiyoshi-ku, Osaka 558-8585, Japan}
\email{yamana@omu.ac.jp}
\begin{abstract}
We construct a five-variable $p$-adic $L$-function attached to Hida families on the definite unitary groups $\U(3)$ and $\U(2)$ by using the Ichino-Ikeda formula. 
The interpolation formula fits into the conjectural shape of $p$-adic $L$-functions predicted by Coates and Perrin-Riou. 
\end{abstract}

\maketitle
\tableofcontents

\section{Introduction}\label{sec:1}

The aim of this paper is to construct a five-variable Iwasawa function interpolating a square root of the algebraic part of central values of the $L$-series attached to a pair of Hida families on the definite unitary groups $\U(3)$ and $\U(2)$. 
We establish the explicit interpolation formulae, which completely comply the conjectural framework described in \cite{CPR,Coe,Coe2}. 

Throughout this paper we fix a {\it split} prime number $p>3$ and an embedding $\iot_p:\overline{\QQ}\hookrightarrow\overline{\QQ}_p$, where $\overline{\QQ}_p$ is a fixed algebraic closure of $\QQ_p$. 


\subsection{Hida families on $\U(n)$ and the associated Galois representations}\label{ssec:15}

Fix a finite extension $F$ of $\QQ_p$ and denote its maximal compact subring by $\calo$. 
For each positive integer $n$, let $T_n\subset \GL_n$ be the diagonal torus. 
Let 
\[\Lam_n:=\calo\powerseries{T_n(\ZZ_p)}=\varprojlim_{m\geq 1}\calo\left[T_n(\ZZ/p^m\ZZ)\right] \]
be the completed group algebra, and $\bfI_n$ a local and normal $\Lam_n$-algebra finite and flat over $\Lam_n$. 
We say that an $\calo$-algebra homomorphism $\ulQ :\bfI_n\to\overline{\QQ}_p$ is locally algebraic if its restriction to $T_n(\ZZ_p)$ is of the form $\ulQ(z_1,\dots,z_n)=\prod_{i=1}^nz_i^{k_{Q_i}}\eps_{Q_i}(z_i^{})$ with $(k_{Q_1},\dots,k_{Q_n})\in\ZZ^n$ and characters $\eps_{Q_i}:\ZZ_p^\times\rightarrow \overline{\QQ}_p^\times$ of finite order. 
We call $k_{\ulQ}=(k_{Q_1},\dots,k_{Q_n})\in\ZZ^n$ the weight of $\ulQ$ and $\eps_{\ulQ}=(\eps_{Q_1},\dots,\eps_{Q_n})$ the finite part of $\ulQ$. 
Let $\frkX_{\bfI_n}$ be the set of locally algebraic points in $\Spec\bfI_n(\overline{\QQ}_p)$. 
We say that $\ulQ\in\frkX_{\bfI_n}$ is dominant if $k_{Q_1}\leq k_{Q_2}\leq\dots\leq k_{Q_n}$, and $\ulQ$ is sufficiently regular if $k_{Q_1}< k_{Q_2} <\dots < k_{Q_n}$. Let $\frkX_{\bfI_n}^+$ be the subset of locally algebraic points of dominant weights and $\frkX_{\bfI_n}^{++}$ the subset of points of sufficiently regular weights. 

Let $E$ be an imaginary quadratic field in which $p$ is split. 
We denote the rings of ad\`{e}les of $\QQ$ and $E$ by $\AA$ and $\EE$. 
Let $x\mapsto x^\tht$ be the non-trivial automorphism of $E$. 
We write $\frkp$ for the prime ideal induced by the restriction of $\iot_p$ to $E$. 
Fix a positive definite Hermitian matrix $T$ in $\Mat_n(E)$. 
For $g\in\Mat_n(E)$ we define $g^\ddagger:=T^{-1}{}\trs g^\tht T$. 
The definite unitary group $\U(n)$ associated with $T$ is the algebraic group defined over $\QQ$ by setting 
\[\U(n)(R)=\{g\in \Mat_n(E\otimes_{\QQ} R)\mid g^\ddagger g=1\}\]
for any $\QQ$-algebra $R$. 

We shall make use of Hida theory for definite unitary groups developed in \cite[\S 2]{Geraghty}. 
Let $N$ be a positive integer only divisible by primes $q\neq p$ split in $E$. 
Choose an ideal $\frkN$ of the ring $\frkr$ of integers of $E$ such that $\frkN\overline{\frkN}=N\frkr$. 
This ideal $\frkN$ shall be referred to the tame level. 
Hida theory produces a free $\Lam_n$-module $e_\ord\bfS^{\U(n)}(\frkN,\Lam_n)$ of finite rank equipped with a faithful action of the universal ordinary Hecke algebra $\bfT^{(n)}(\frkN)$ for the unitary group $\U(n)$ (See \cite[Definition 2.23]{Geraghty}). The $\Lam_n$-module $e_\ord\bfS^{\U(n)}(\frkN,\Lam_n)$ is referred to the space of ordinary $\Lam_n$-adic forms, which roughly speaking consists of $p$-adic families of $p$-ordinary modular forms on $\U(n)$ invariant by the mirabolic subgroup of level $\frkN$. 
An $\bfI_n$-adic Hida family $\bdsf$ on $\U(n)$ is a non-zero Hecke eiegnform in $e_\ord\bfS^{\U(n)}(\frkN,\bfI_n):=e_\ord\bfS^{\U(n)}(\frkN,\Lam_n)\otimes_{\Lam_n}\bfI_n$, which induces a $\Lam_n$-algebra homomorphism $\lam_{\bdsf}:\bfT^{(n)}(\frkN)\to \bfI_n$. 

Denote the absolute Galois group of a field $L$ by $\Gam_L$ and its cyclotomic character by $\cyc$. Let $\frkm$ be the maximal ideal of $\bfI_n$. 
To each $\bfI_n$-adic Hida family $\bdsf$, one can associate the residual semisimple Galois representation $\bar\rho_{\bdsf}: \Gamma_E\to \GL_n(\bfI_n/\frkm)$ (see \cite[Proposition 2.28]{Geraghty}). 
If $\bar\rho_{\bdsf}$ is absolutely irreducible, then we can further obtain the Galois representation $\rho_{\bdsf}:\Gamma_E\to \GL_n(\bfI_n)$ unramified outside primes dividing $Np$ and primes $l$ where $\U(n)(\QQ_l)$ is ramified (see \cite[Proposition 2.29]{Geraghty} for the more detailed descriptions). Denote by $V_{\bdsf}$ the free $\bfI_n$-module of rank $n$ on which $\Gamma_E$ acts via $\rho_{\bdsf}$. 
For each $Q\in \frkX^+_{\bfI_n}$, the specialization $V_{\ulQ(\bdsf)}:=V_{\bdsf}\otimes_{\bfI_n,\ulQ}\overline{\QQ}_p$ is the geometric $p$-adic Galois representation associated with some automorphic representation $\bdpi_{\ulQ}\simeq\otimes_v\bdpi_{\ulQ,v}$ of $\U(n)(\AA)$. 
Then the representation $V_{\bdsf}$ is \emph{conjugate self-dual} in the sense that
\[V_{\bdsf}^\vee\simeq V_{\bdsf}^\tht\otimes\cyc^{n-1}.\]
Moreover, by the local description of $p$-adic Galois representations \cite[Corollary 2.33]{Geraghty} at $p$ combined with \cite[Lemma 7.2]{TU99}, there exists a filtration $\{\mathrm{Fil}_i(V_{\bdsf})\}_{i=1}^n$ of $\Gam_{E_\frkp}$-stable lattices 
\[\{0\}=\mathrm{Fil}_0V_{\bdsf} \subset \mathrm{Fil}_1V_{\bdsf}\subset \dots\subset \mathrm{Fil}_{n-1} V_{\bdsf}\subset \mathrm{Fil}_nV_{\bdsf} =V_{\bdsf}|_{\Gam_{E_\frkp}}\]
such that for every $\ulQ\in\frkX^{++}_{\bfI_n}$, the specialization $V_{\ulQ(\bdsf)}|_{\Gam_{E_\frkp}}$ is Hodge-Tate and each graded piece \[\mathrm{gr}_i(V_{\ulQ(\bdsf)}|_{\Gam_{E_\frkp}}):=\mathrm{Fil}_iV_{\ulQ(\bdsf)}/\mathrm{Fil}_{i-1}V_{\ulQ(\bdsf)}\] has Hodge-Tate weights $-k_{Q_i}-i+1$ for $1\leq i\leq n$. 
Here the Hodge-Tate number of $\QQ_p(1)$ is one in our convention. 
Likewise there exists a filtration $\{\mathrm{Fil}_i(V_{\bdsf}^c)\}_{i=1}^n$ of $\Gam_{E_{\frkpbar}}$-stable lattices in $V_{\bdsf}^\tht$ such that such that for every $\ulQ\in\frkX^{++}_{\bfI_n}$, each graded piece $\mathrm{gr}_i(V_{\ulQ(\bdsf)}^\tht|_{\Gam_{E_{\frkpbar}}})=\mathrm{Fil}_iV^\tht_{\ulQ(\bdsf)}/\mathrm{Fil}_{i-1}V^\tht_{\ulQ(\bdsf)}$ is Hodge-Tate of weight $k_{Q_{n-i+1}}-i+1$. 


\subsection{The algebraicity of central values}

Let $\ulk=(k_1^{},k_2^{},\dots,k_n^{})$ and $\ulk'=(k_1',k_2',\dots,k_{n-1}')$ be tuples of integers satisfying the following interlacing relation
\beq
k_1^{}\leq -k_{n-1}'\leq k_2^{}\leq\cdots\leq-k_2'\leq k_{n-1}^{}\leq -k_1'\leq k_n^{}. \label{tag:11}
\eeq 
Let $\pi$ be an irreducible tempered automorphic representation of $\U(n)(\AA)$ such that $\pi_\infty$ has highest weight $-\ulk$, and $\sig$ an irreducible tempered automorphic representation of $\U(n-1)(\AA)$ such that $\sig_\infty$ has highest weight $-\ulk'$. 

The {\it complete} automorphic $L$-function for the product $\pi\times\sigma$ is defined by 
\[L(s,\pi\times\sig)=L^{\GL}(s,\BC(\pi)\times \BC(\sig)),\]
where $\BC(\pi)$ (resp. $\BC(\sig)$) is the functorial lift of $\pi$ (resp. $\sig$) to an automorphic representation of $\GL_n(\EE)$ (resp. $\GL_{n-1}(\EE)$). 
The $L$-function in the right hand side has been defined by the Rankin-Selberg convolution whose local and global analytic theories were established by Jacquet, Piatetski-Shapiro and Shalika in \cite{JPSS2}. 
Let $L(s,\sig,\Ad)$ and $L(s,\pi,\Ad)$ be the {\it complete} adjoint $L$-functions of $\sig$ and $\pi$, respectively. 
These are the Asai and twisted Asai $L$-functions of $\BC(\sig)$ and $\BC(\pi)$ (cf. Remark \ref{rem:21}). 

The ratio $ \frac{L\bigl(\frac{1}{2},\pi\times\sig\bigl)}{L(1,\pi,\Ad)L(1,\sig,\Ad)}$ is indeed an algebraic number thanks to \cite[Theorem C]{GL21AJM} and \cite[Corollary 7.9]{Chen23AIM}. 
Proposition \ref{prop:22} proves its refinement for central values by using Shimura's mass formula.

\subsection{The period}\label{ssec:17}

To make our interpolation formula meaningful, we will give the precise definition of periods for the motive $V_{\bdsf}$. 
We denote the conductor of $\bdpi_{\ulQ}$ by $\frkN_{\bdpi_{\ulQ}}$. 
In this introductory section we use a simplified period defined by 
\[\Ome^{(M)}(V_{\ulQ(\bdsf)})=[\calk:\calk_0(\frkN_{\bdpi_{\ulQ}})]2^{\vka_{\bdpi_{\ulQ}}}L^{(M)}(1,\bdpi_{\ulQ},\Ad)\cale(\bdpi_{\ulQ,p},\Ad)\calb_{\pi_{\ulQ,p}} \]
for a positive integer $M$, where 
\begin{itemize} 
\item $2^{\vka_{\bdpi_{\ulQ}}}$ is the order of the $S$-group associated to the $L$-parameter of $\bdpi_{\ulQ}$; 
\item $L^{(M)}(1,\bdpi_{\ulQ},\Ad)$ is the partial adjoint $L$-series of $\bdpi_{\ulQ}$ with the archimdean factor but without Euler factors at primes dividing $M$,  
\item $\calb_{\pi_{\ulQ,p}}$ is the normalized local norm of essential Whittaker function at $p$ (see Proposition \ref{prop:42}); 
\item $\cale(\bdpi_{\ulQ,p},\Ad)$ is the modified Euler factor for the adjoint motive of $\ulQ(\bdsf)$ defined in Definition \ref{D:adjoint}. 
\end{itemize} 
In a future work of Shih-Yu Chen, the adjoint $L$-value $L^{(M)}(1,\bdpi_{\ulQ},\Ad)$ is roughly the Petersson norm of the suitably normalized generic form on the quasi-split inner form of $\U(n)$ for $n\leq 3$. 
This period $\Ome^{(M)}(V_{\ulQ(\bdsf)})$ is canonical in the sense only depends on $M$ and the representation $\bdpi_{\ulQ}$ associated with the form $\ulQ(\bdsf)$.



\subsection{The product $L$-series for $\U(3)\times \U(2)$} 

Fix positive rational numbers $t_1$ and $t_2$.  
Let $\U(3)$ and $\U(2)$ be the definite unitary groups attached to $T$ and $T'$, where 
\begin{align*}
T&=\begin{bmatrix}t_1&&\\
&1&\\
&&t_2\end{bmatrix}, & 
T'&=\begin{bmatrix}t_1&\\
&t_2\end{bmatrix}. 
\end{align*} 
Let $\varSigma^-_T$ be the finite set consisting of primes $q$ such that $\U(2)(\QQ_q)$ is compact. 
Let $N$ and $N'$ be natural numbers that satisfy the following condition: 
\beq\label{H1}
\text{all the prime factors of $NN'$ are split in $E$.} \tag{splt} 
\eeq
 To simply the discussion of the introduction, we assume that 
\beq\label{odd}\text{$NN'$ is odd and $2\notin\varSigma^-_T$}.\tag{odd}\eeq
Fix a decomposition $N\frkr=\frkN\overline{\frkN}$ and $N'\frkr=\frkN'\overline{\frkN'}$. 
Let 
\begin{align*}
\bdsf&\in e_\ord\bfS^{\U(3)}(\frkN,\bfI_3), &
\bdsg&\in e_\ord\bfS^{\U(2)}(\frkN',\bfI_2)
\end{align*}
be Hida families. 
We further assume that the residual Galois representations $\bar\rho_{\bdsf}$ and $\bar\rho_{\bdsg}$ are both absolutely irreducible. Let $V_{\bdsf}$ and $V_{\bdsg}$ be the Galois representation of $\Gam_E$ associated with the Hida family $\bdsf$ and $\bdsg$ respectively. Consider the tensor product representation $V_{\bdsf\bdsg}:=V_{\bdsf}\otimes V_{\bdsg}$ of rank six over the five variable Iwasawa algebra $\bfI_3\wh\otimes_{\calo} \bfI_2$ and let $\bfV$ be the induced representation of $\Gam_\QQ$ defined by
\[\bfV:=\Ind_{\Gamma_E}^{\Gamma_\QQ}(V_{\bdsf\bdsg}^{}\otimes\cyc^2).
\]
For each prime number $q$ we denote the Weil-Deligne group of $\QQ_q$ by $W_{\QQ_q}$. 
For each $\calq=(Q,Q')\in\frkX_{\bfI_3}^+\times\frkX_{\bfI_2}^+$, let $\bfV_\calq$ be the specialization of $\bfV$ at $\calq$ and define the complex $L$-series of the $p$-adic Galois representation $\bfV_\calq$ by the Euler product 
\[L(\bfV_\calq,s)=\prod_q L_q(\bfV_\calq,s) \]
of the local $L$-factors attached to $\mathrm{WD}_q(\bfV_\calq)\otimes_{\overline{\QQ}_p^{},\iot_p^{-1}}\CC$, where $\mathrm{WD}_q(\bfV_\calq)$ is the Weil-Deligne representation of $W_{\QQ_q}$ over $\overline{\QQ}_p$ associated to $\bfV_\calq$. Putting
\begin{align*}
(\lam_{Q_1},\lam_{Q_2},\lam_{Q_3})&=(-k_{Q_1}+1,-k_{Q_2},-k_{Q_3}-1);\\ 
(\mu_{Q_1'},\mu_{Q_2'})&=\biggl(-k_{Q_1'}+\frac{1}{2},-k_{Q_2'}-\frac{1}{2}\biggl),
\end{align*}
we define the archimedean $L$-factor of $\bfV_\calq$ by 
\[\Gam(\bfV_\calq^{},s)=\prod_{i=1,2,3}\prod_{j=1,2}\Gamma_\CC\biggl(s+\frac{1}{2}+|\lam_{Q_i^{}}+\mu_{Q_j'}|\biggl),\]
where $\Gamma_{\CC}(s)=2(2\pi)^{-s}\Gamma(s)$. We are interested in the algebraic part of the value of $L(\bfV_\calq^{},s)$ at $s=0$. 
Note that $L(\bfV_\calq,0)$ are central values as $\bfV^\vee\otimes\cyc\simeq\bfV$ is self-dual. 
With the assumption \eqref{H1},  the specializations of the Hecke eigensystems $\lam_{\ulQ(\bdsf)}:=\ulQ\circ\lam_{\bdsf}$ and $\lam_{\ulQ'(\bdsg)}:=\ulQ'\circ\lam_{\bdsg}$ determine unique unitary automorphic representations $\bdpi_{\ulQ}$ and $\bdsig_{\ulQ'}$ of $\U(3)(\AA)$ and $\U(2)(\AA)$, and we have
\[\Gam(\bfV_\calq^{},s)L(\bfV_\calq^{},s)=L\biggl(s+\frac{1}{2},\bdpi_{\ulQ^{}}^{}\times\bdsig_{\ulQ'}\biggl).\]

Consider the set of critical points defined by \[\frkX^{\rm crit}=\{(\ulQ,\ulQ')\in \frkX_{\bfI_3}^{++}\times \frkX_{\bfI_2}^{++}\mid k_{Q_1^{}}\leq -k_{Q_2'}\leq k_{Q_2^{}}\leq -k_{Q_1'}\leq k_{Q_3^{}}\}.\]
For $\calq=(\ulQ,\ulQ')\in\frkX^{\rm crit}$ we view the algebraic numbers \[\frac{\Gam(\bfV_\calq^{},0)L(\bfV_\calq^{},0)}{\Ome^{(M)}(V_{\ulQ(\bdsf)})\Ome^{(M)}(V_{\ulQ'(\bdsg)})}\sim \frac{L\bigl(\frac{1}{2},\bdpi_{\ulQ}\times\bdsig_{\ulQ'}\bigl)}{L(1,\bdpi_{\ulQ},\Ad)L(1,\bdsig_{\ulQ'},\Ad)}\pmod{\overline{\QQ}^\times},\]through the embedding $\iota_p$, as a $p$-adic number. The purpose of this paper is to understand the $p$-adic behavior of this ratio when $\calq\in\frkX^{\rm crit}$ varies. 


\subsection{The modified Euler factor at $p$}\label{ssec:16}

To introduce the modified Euler factor at $p$, we prepare some notation. 
We consider the  rank three $\Gam_{E_\frkp}$-invariant and $\Gam_{E_{\frkpbar}}$-invariant subspaces of $V_{\bdsf\bdsg}$ by 
\begin{align*}
\mathrm{Fil}^+_\frkp\bfV_{\bdsf\bdsg}&=\mathrm{Fil}_1V_{\bdsf}^{}\otimes V_{\bdsg}^{}+\mathrm{Fil}_2V_{\bdsf}^{}\otimes \mathrm{Fil}_1V_{\bdsg}^{}; \\
\mathrm{Fil}^+_{\frkpbar}\bfV_{\bdsf\bdsg}&=\mathrm{Fil}_1V_{\bdsf}^\tht\otimes V_{\bdsg}^\tht+\mathrm{Fil}_2V_{\bdsf}^\tht\otimes \mathrm{Fil}_1V_{\bdsg}^\tht,
\end{align*}
and define the six dimensional $\Gam_{\QQ_p}$-invariant subspace of $\bfV$ by 
\begin{align*}
\mathrm{Fil}^+\bfV&=\left(\mathrm{Fil}^+_\frkp\bfV_{\bdsf\bdsg}\oplus\mathrm{Fil}^+_{\frkpbar}\bfV_{\bdsf\bdsg}\right)\otimes\cyc^2.
\end{align*}
The pair $(\mathrm{Fil}^+\bfV,\frkX^{\rm crit})$ satisfies the Panchishkin condition in \cite[p.~217]{Gre} in the sense that all the Hodge-Tate numbers of $\mathrm{Fil}^+\bfV_\calq$ are positive but none of the Hodge-Tate numbers of $\bfV_\calq/\mathrm{Fil}^+\bfV_\calq$ is positive for $\calq\in\frkX^{\rm crit}$. 

Let $\addchar:\AA/\QQ\to\CC^\times$ be the additive character with the archimedean component $\addchar_\infty(z)=e^{2\pi\sqrt{-1}x}$ and $\addchar_p:\QQ_p\to\CC^\times$ the local component of $\addchar$ at the prime number $p$. Let $dx$ be the self-dual Haar measure on $\QQ_p$ with respect to $\addchar_p$. For each $p$-adic representation $V$ of $\Gamma_{\QQ_p}$, recall that the $\gamma$-factor $\gamma(V,s)$ is defined by 
\[\gamma(V,s)=\frac{\varepsilon({\rm WD}_p(V),\addchar_p,dx)L_p(V^\vee,1-s)}{L_p(V,s)},\]
where $\varepsilon({\rm WD}_p(V),\addchar_p, dx)$ is the local constant defined in \cite[\S 4]{Deligne73}. 
The modified Euler factor at $p$ is defined by 
\[\cale_p(\mathrm{Fil}^+\bfV_\calq)
=\frac{1}{\gam\Big(\mathrm{Fil}^+_\frkp\bfV_\calq,0\Big)\gam\Big(\mathrm{Fil}^+_{\frkpbar}\bfV_\calq,0\Big)L_p(\bfV_\calq,0)}. \]


\subsection{Interpolation formulae}\label{ssec:18}

Let $\mathrm{Frac}(\bfI_3\widehat{\otimes}_\calo\bfI_2)$ stand for the total ring of fractions of $\bfI_3\widehat{\otimes}_\calo\bfI_2$. 

\begin{theorem}\label{thm:12}
We assume \eqref{odd} and \eqref{H1}. 
Then there exists a unique element $L_p(\bfV)\in\mathrm{Frac}(\bfI_3\widehat{\otimes}_\calo\bfI_2)$ such that for $\calq=(\ulQ,\ulQ')\in\frkX^{\rm crit}$
\[\calq(L_p(\bfV))=\frac{\Gam(\bfV_\calq^{},0)L(\bfV_\calq^{},0)}{\Ome^{(NN')}(V_{\ulQ(\bdsf)})\Ome^{(NN')}(V_{\ulQ'(\bdsg)})}\cale(\mathrm{Fil}^+\bfV_\calq). \]
\end{theorem}

\begin{Remark}\label{rem:12}
\begin{enumerate}
\item\label{rem:131} We note that the global root number \[\vep\biggl(\frac{1}{2},\bdpi_{\ulQ}\times\bdsig_{\ulQ'}\biggl)=+1\text{ for }(\ulQ,\ulQ')\in\frkX^{\rm crit}\] for $(\ulQ,\ulQ')\in\frkX^{\rm crit}$ by (splt), (\ref{tag:b2}) and Remark \ref{rem:e3}. 
\item\label{rem:134} For $\ulQ\in\frkX_{\bfI_3}^{++}$, the weights of $\bdpi_{\ulQ,\infty}$ is sufficiently regular, so $\bdpi_{\ulQ}$ is tempered according to the endoscopic classification of cuspidal automorphic representations on $\U(3)$ (c.f. \cite[3.2.3, p.~618]{BC04}).
\item\label{rem:133} Yifeng Liu \cite{YL} has recently constructed anticyclotomic $p$-adic $L$-functions for automorphic representations of $\U(n)\times\U(n-1)$ whose archimedean component is the trivial representation. 
Despite our result is restricted to $\U(3)\times\U(2)$, this paper works with automorphic representations of general weights at archimedean components and constructs several variable $p$-adic square roots of $L$-functions attached to Hida families of modular forms on unitary groups. We provide more precise interpolation formulae including the interpolation at infinite order anticyclotomic characters. We expect to generalize our construction to $\U(n)\times\U(n-1)$ in the future. 
\item This paper mainly concerns the $p$-adic $L$-functions for the Rankin-Selberg convolution $\BC(\bdpi_{\ulQ})\times \BC(\bdsig_{\ulQ'})$ in the \emph{balanced} case in the sense that the weights satisfy the interlacing relation (\ref{tag:11}). 
This is an analogue of theta elements in \cite{BD} and $p$-adic triple product $L$-functions in the balanced case (\cite{GS20} and \cite{MH}). In a forthcoming joint work with Michael Harris, we construct $p$-adic $L$-functions in the \emph{unbalanced} case: the weights of $\BC(\bdpi_{\ulQ})\times \BC(\bdsig_{\ulQ'})$ satisfy different interlacing relation \[k_{Q_1}\leq -k_{Q_2'}\leq k_{Q_2}< k_{Q_3}\leq -k_{Q_1'}.\]
The method uses Hida families of modular forms on non-compact unitary groups $\U(2,1)\times\U(1,1)$. 
\end{enumerate}
\end{Remark}


In Definition \ref{def:510}, we shall construct the Hecke-equivariant perfect pairing 
\[\bfB_\frkN: e_{\ord}\bfS^{\U(n)}(\frkN,\chi,\bfI_n)\times e_{\ord}\bfS^{\U(n)}(\frkN,\chi,\bfI_n)\to\bfI_n,\]
which interpolates the canonical bilinear pairing between automorphic forms on definite unitary groups. 
Put \[\eta_{\bdsf}=\bfB_\frkN(\bdsf,\bdsf)\in\bfI_n.\] 
The element $\eta_{\bdsf}$ is expected to be related to the congruence number of $\bdsf$. 
More precisely, we will construct a \emph{theta element} $\scrl_{\bdsf,\bdsg}\in \bfI_3\widehat{\otimes}_\calo\bfI_2$ associated to a pair of Hida families $(\bdsf,\bdsg)$ on $\U(3)\times\U(2)$, and define \[L_p(\bfV)=\frac{\scrl_{\bdsf,\bdsg}^2}{\eta_{\bdsf}\eta_{\bdsg}}. \] 
Therefore, we actually construct the square root for the $p$-adic $L$-function for the Galois representation $\bfV$.

The proof of Theorem \ref{thm:12} is divided into two steps: 
\begin{itemize}
\item[(i)] we construct $\scrl_{\bdsf,\bdsg}$ via the $p$-adic interpolation of the period integrals of modular forms on $\U(3)\times \U(2)$ along $\U(2)$ in Section \ref{sec:6}; 
\item[(ii)] we evaluate those period integrals via the Ikeda-Ichino conjecture explicitly in Section \ref{sec:2}.
\end{itemize}
The second step (ii) may have independent interest in some analytic aspects of the Rankin-Selberg $L$-values for unitary groups. We begin with a brief outline of the first step (i). 

\subsection{Construction of $\scrl_{\bdsf,\bdsg}$}\label{ssec:19}

We denote the finite ad\`{e}le ring of $\QQ$ by $\widehat{\QQ}$. 
Let $G=\U(3)$ and $H=\U(2)$. 
For simplicity we suppose that 
\begin{align*}
N'&=1, &
\chi'&=1.
\end{align*}
We will specify suitable maximal compact subgroups $\calk$ of $G(\widehat{\QQ})$ and $\calk'$ of $H(\widehat{\QQ})$ (cf. Appendix \ref{sec:d}). 
Define open subgroups of $\calk'$ by
\begin{align*}
\calk''(1)&=\calk\cap\calk', & 
\calk''(p^\ell)&=\calk''(1)\cap\imath_\frkp^{\prime-1}\left(\left\{\begin{bmatrix} 1+\frkp^\ell & * \\ \frkp^\ell & 1+\frkp^\ell \end{bmatrix}\right\}\right)
\end{align*}
for each positive integer $\ell$. 
We consider the finite sets 
\begin{align*}
X_\ell^{}&=G(\QQ)\bsl G(\widehat{\QQ})/\calk_1(p^\ell\frkN), &
X_\ell'&=H(\QQ)\bsl H(\widehat{\QQ})/\calk''(p^\ell), & 
\bfX_\ell^{}&=X_\ell^{}\times X_\ell'. 
\end{align*}
Define $\vsi^{(p)}_{}=(\vsi^{(p)}_l)\in G(\widehat{\QQ})$ by 
\beq
\vsi^{(p)}_l=\begin{cases}
\imath_\frkl^{-1}(\vsi) &\text{if $l$ splits in $E$ and differs from $p$, }\\
\ono_3 &\text{otherwise. } 
\end{cases} \label{tag:12}
\eeq
Here the matrix $\vsi$ is defined in (\ref{tag:23}). 
Since 
\beq
\iot(\calk''(1))\vsi^{(p)}\subset\vsi^{(p)}\calk_1(\frkN), \label{tag:13}
\eeq
we define the map $\jmath(x)=\iot(x)\vsi^{(p)}$ induces a map $X_0'\to X_0^{}$. In \S \ref{ssec:61} we construct a collection of regularized diagonal cycles $\Del_\ell^\dagger$ of $\bfX_\ell$ that are compatible with respect to the projection maps $\bfX_\ell\to\bfX_{\ell-1}$. 
We therefore obtain the big diagonal cycle 
\[\Del_\infty^\dagger=\lim_{\longleftarrow \ell}\Del_\ell^\dagger\in\lim_{\longleftarrow \ell}\mathrm{Pic}\bfX_\ell\otimes_\ZZ\ZZ_p. \]
We define an $\bfI_3\widehat{\otimes}\bfI_2$-adic modular form on $G\times H$ by 
\[\bdsF=\bdsf\boxtimes\bdsg:\lim_{\longleftarrow \ell}\bfX_\ell\to\bfI_3\widehat{\otimes}_\calo\bfI_2\] 
by $\bdsF(x,x')=\bdsf(x)\bdsg(x')$. 
Then $\bdsF$ naturally induces a map 
\[\bdsF_*:\lim_{\longleftarrow \ell}\mathrm{Pic}\bfX_\ell\otimes_\ZZ\ZZ_p\to\bfI_3\widehat{\otimes}_\calo\bfI_2. \]  
The element $\scrl_{\bdsf,\bdsg}$ equals the theta element $\Tht_{\bdsF}$ attached to $\bdsF$ defined by 
\[\Tht_{\bdsF}=\bdsF_*(\Del_\infty^\dagger)\in\bfI_3\widehat{\otimes}_\calo\bfI_2 \]
up to some fudge factor (see (\ref{tag:63})). 
This theta element is an analogue for $\U(2)\times\U(3)$ of theta elements constructed by Bertolini and Darmon \cite{BD} for $\SO(2)\times\SO(3)$ (cf. \cite{CH}) and by the first author \cite{MH} for $\SO(3)\times\SO(4)$.

Let $\d_{\calk'}h$ be the Haar measure of $H(\AA)$ giving $H(\RR)\calk'$ volume $1$. 
We then proceed to show in Proposition \ref{prop:61} that the evaluation of $\Tht_{\bdsF}$ at $\calq\in\frkX^{\rm crit}$ is given by the normalized period integral 
\[\scrp_{\calk'}^*(\vPh^\dagger):=\frac{p^{\frac{5}{2}\ell}}{(\rho_p(p)\mu_p(p)^2\nu'_p(p))^\ell}\int\limits_{H(\QQ)\bsl H(\AA)}\vPh^\dagger\left(\jmath(h)\imath_\frkp^{-1}\left(\begin{bmatrix}
0 & 0 & -p^{-\ell} \\
0 & p^\ell & 1 \\
p^{2\ell} & p^\ell & 0
\end{bmatrix}\right),h\right)\d_{\calk'}h, \]
where $\vPh^\dagger\in\bdpi_{\ulQ}\otimes\bdsig_{\ulQ'}$ is a suitable $H(\RR)$-invariant modular form on $G\times H$.

\subsection{The explicit Ichino-Ikeda formula}\label{ssec:13}
A key ingredient to evaluate the period integral $\scrp_{\calk'}^*(\vPh^\dagger)$ is the Ichino-Ikeda conjecture which has been first formulated for orthogonal groups in \cite{II}. 
Its unitary analogue was formulated for unitary groups in \cite{NHarris} and proved in \cite{Z,BLZZ,BCZ} (cf. Theorem \ref{thm:21}). 
This Ichino-Ikeda formula relates a square of this period integral to the product of the central value $L\bigl(\frac{1}{2},\pi\times\sig\bigl)$ and local integrals.

To make this formula precise, we need to compute various local integrals for suitable test vectors. 
Thanks to (\ref{H1}), we can apply the local theory of newforms for representations of general linear groups, which was developed in \cite{JPSS,NM}. 
We will compute these integrals at $p$ in Section \ref{sec:3} and at archimedean and ramified places in Appendices \ref{sec:b}, \ref{sec:d} and \ref{sec:e}. 
To remove Hypothesis (\ref{H1}), one needs to compute the local integral for suitably chosen test vectors at inert primes. 

A local key ingredient is the splitting lemma which has been proved in \cite{LM,Z} (see Lemma \ref{lem:a1}). 
This lemma relates the Ichino-Ikeda local integral to a square of the JPSS integral at split primes. 
It is well-known that the JPSS integral $Z(s,W_{\pi_l},W_{\sig_l})$ of essential vectors coincides with the $L$-factor $L^\GL(s,\pi_l\times\sig_l)$ when $\sig_l$ is unramified. 

When $N'>1$, we need to compute the local integral when $\sig_l$ is ramified. 
When $\sig_l$ satisfies the condition ($H_3$) in \S \ref{ssec:213}, following the construction \cite{Schmidt} of $p$-adic $L$-functions for $\GL_3\times\GL_2$, Section \ref{sec:4} will construct an operator $\Tht^{\chi'_l}_{\frkN'}$ having the following properties: 
\begin{itemize}
\item the restriction of $\Tht^{\chi'_l}_{\frkN'}W_{\pi_l}$ to $H(\QQ_l)$ has an appropriate $\calk'_l$-type; 
\item $Z(s,\Tht^{\chi'_l}_{\frkN'}W_{\pi_l},W_{\sig_l})$ has a simple formula. 
\end{itemize}
Furthermore, we replace the pair $(\pi,\sig)$ by a suitable twist $(\pi\otimes\vrh_\AA^{-1},\sig\otimes\vrh_\AA)$ and replace the pair $(\bdsf,\bdsg)$ by another pair of Hida families $(\bdsf_\vrh,\bdsg_\vrh)$. 
We apply the step (i) to $\Tht^{\chi'}_{\frkN'}\bdsf_\vrh\boxtimes\bdsg_\vrh$. 
With these local calculations we conclude that 
\[\frac{\scrp_{\calk'}^*(\vPh^\dagger)^2}{(\vph_\pi^{},\vph_{\pi^\vee})_\calk(\vph_\sig^{},\vph_{\sig^\vee})_{\calk'}}\approx\frac{L\bigl(\frac{1}{2},\bdpi_{\ulQ}\times\bdsig_{\ulQ'}\bigl)}{L(1,\bdpi_{\ulQ},\Ad)L(1,\bdsig_{\ulQ'},\Ad)}, \]
where $\vph_\pi\in\bdpi_{\ulQ}$ and $\vph_\sig\in\bdsig_{\ulQ'}$ are highest weight essential vectors. 


\subsection*{Acknowledgement}
Hsieh is supported by the NSTC grant 112-2628-M-002-011. Yamana is partially supported by JSPS Grant-in-Aid for Scientific Research (C) 23K03055. 
This work was partly supported by MEXT Promotion of Distinctive Joint Research Center Program JPMXP0723833165. 
We thank Michael Harris for constant help and encouragements, and Shih-Yu Chen for useful discussions on the algebraicity of $L$-values.


\section*{Notation}

Besides the standard symbols $\ZZ$, $\QQ$, $\RR$, $\CC$, $\ZZ_q$ and $\QQ_q$ we denote by $\RR_+^\times$ and $\QQ^\times_+$ the groups of strictly positive real and rational numbers, and by $\CC^1$ the group of complex numbers of absolute value $1$. 
Let $\AA$ be the ring of ad\`{e}les of $\QQ$. 
We write $\widehat{\QQ}$ for the finite part of $\AA$. 
Put $\widehat{\ZZ}=\prod_q\ZZ_q\subset\widehat{\QQ}$.  
Given a place $v$ of $\QQ$, we write $\QQ_v$ for the completion of $\QQ$ with respect to $v$. 
We shall regard $\QQ_v$ and $\QQ_v^\times$ as subgroups of $\AA$ and $\AA^\times$ in a natural way. 
We denote by the formal symbol $\infty$ the real place of $\QQ$ and do not use $q,l$ for the infinite place. 

Define $\addchar_v:\QQ_v\to\CC^1$ by $\addchar_\infty(z)=e^{2\pi\sqrt{-1}z}$ for $z\in\RR$ and by $\addchar_p(x)=\addchar_\infty(-y)$ for $x\in\QQ_p$ with $y\in\ZZ[p^{-1}]$ such that $y-x\in\ZZ_p$. 
Then $\addchar=\prod_v\addchar_v$ defines a character of $\AA/\QQ$. 
We associate to $m\in\QQ$ the global additive character $\addchar^m$ defined by $\addchar^m(z)=\addchar(mz)$ for $z\in\AA$. 
For $a\in\QQ_q^\times$ we define an additive character $\addchar_q^a$ of $\QQ_q^{}$ by $\addchar_q^a(z)=\addchar_q(az)$ for $z\in\QQ_q$. 

Let $\d z_v$ be the Haar measure on $\QQ_v$ self-dual with respect to the pairing $(z_v^{},z_v')\mapsto\addchar_v(z_v^{}z_v')$. 
Note that $\int_{\ZZ_q}\d z_q=1$ for each rational prime $q$ and that $\d z_\infty$ is the usual Lebesgue measure on $\RR$. 
Let $\d z=\prod_v\d z_v$.  
Then $\d z$ is the Haar measures on $\AA$ such that $\AA/\QQ$ has volume 1. 
Let $\d^\times t=\prod_v\d^\times t_v$ be the Haar measure on $\AA^\times$, where $\d^\times t_v$ is the Haar measure on $\QQ_v^\times$ normalized by $\int_{\ZZ_q^\times}\d^\times t_q=1$ if $v = q <\infty$, and $\d^\times t_\infty=\frac{\d t_\infty}{|t_\infty|}$. 


Let $E$ be an imaginary quadratic field of discriminant $-D_E$ with the integer ring $\frkr$. 
We write $\eps_{E/\QQ}=\prod_v\eps_{E_v/\QQ_v}$ for the quadratic character of $\QQ^\times\bsl\AA^\times$ associated to $E$.  
Set 
\begin{align*}
\EE&=E\otimes_\QQ\AA, & 
E_v&=E\otimes_\QQ\QQ_v, & 
\frkr_q&=\frkr\otimes_\ZZ\ZZ_q. 
\end{align*}
We denote by $x\mapsto x^\tht$ the non-trivial automorphism of $E$.  
We write $\trs x\in\Mat_{n,m}(E)$ for the transpose of a matrix $x=(x_{ij})\in\Mat_{m,n}(E)$ and put 
\[x_{}^\tht=(x_{ij}^\tht)\in\Mat_{m,n}(E). \]
Let $\vSi_E^r$ be the set of prime numbers which are ramified in $E$.
 
Once and for all we fix an odd rational prime $p$ that is split in $\frkr$. 
Fix an algebraic closure $\overline{\QQ}$ of $\QQ$.
We fix an embedding $\iot_\infty:\overline{\QQ}\hookrightarrow\CC$ and an isomorphism $\iot_p:\CC\simeq\CC_p$, where $\CC_p$ is the completion of an algebraic closure of $\QQ_p$. 
Given an algebraic number field $L$, we regard $L$ as a subfield in $\CC$ (resp. $\CC_p$) via $\iot_\infty$ (resp. $\iot_p\circ\iot_\infty$). 
Let $\ord_p$ be the $p$-adic valuation on $\CC_p$ normalized so that $\ord_pp=1$. 
We write $\frkp$ for the prime ideal of $\frkr$ above $p$ that corresponds to the restriction of $\iot_p\circ\iot_\infty$ to $E$. 


\section{Hida families on definite unitary groups}\label{sec:5}

\subsection{Unitary groups}\label{ssec:21}

We let the base field be the rational field $\QQ$. 
This assumption simplifies the notation and reduce technicality. 
We write $\Res_{E/\QQ}\GL_n$ for the general linear group over an imaginary quadratic field $E$, regarded as an algebraic group over $\QQ$ by restricting scalars. 
Fix a non-degenerate {\it rational diagonal} matrix $T$ of size $n$. 
Define a Hermitian form $(\;,\;)_T$ on $W=E^n$ by $(u,v)_T=\trs u^\tht Tv$ for $u,v\in W$.  
Let $G=\U(T)$ be the unitary group associated with the Hermitian form $T$. Namely, $G$ is the algebraic group defined by 
\[G=\{g\in\Res_{E/\QQ}\GL_n\;|\;\trs g^\tht Tg=T\}. \]  

We let $\frkr$ be the ring of integers of $E$. Let $l$ be a split prime, i.e., $l\frkr=\frkl\bar\frkl$, $E_l\simeq E_\frkl\oplus E_{\bar\frkl}$ and $\frkr_l\simeq \frkr_\frkl\oplus \frkr_{\bar\frkl}$. 
The projection $(g_1,g_2)\mapsto g_1$ gives an isomorphism $\imath_\frkl:G(\QQ_l)\stackrel{\sim}{\to}\GL_n(E_\frkl)$. 
Fix a maximal compact subgroup $\calk=\prod_q\calk_q^{}$ of $G(\widehat{\QQ})$ such that $\imath_\frkl(\calk_l)=\GL_n(\frkr_\frkl)$ for every split prime $l$. 
Define open subgroups of $\GL_n(\frkr_\frkl)$ by 
\begin{align*}
\calk_0^{(n)}(\frkl^k)&=\{(g_{ij})\in\GL_n(\frkr_\frkl)\;|\;g_{nj}\in \frkl^k\text{ for }j=1,2,\dots,n-1\}, \\ 
\calk_1^{(n)}(\frkl^k)&=\{(g_{ij})\in\calk_0^{(n)}(\frkl^k)\;|\;g_{nn}-1\in \frkl^k\}. 
\end{align*}

Fix a positive integer $N$ whose prime factors are split in $E$. 
We take an ideal $\frkN$ of $\frkr$ such that $\frkr/\frkN\simeq\ZZ/N\ZZ$. 
Let  
\begin{align*}
\calk_0(\frkN)&=\{(g_q)\in\calk\;|\; \imath_\frkl(g_l^{})\in\calk_0^{(n)}(N\frkr_\frkl)\text{ for }l|N\}, \\ \calk_1(\frkN)&=\{(g_q)\in\calk\;|\; \imath_\frkl(g_l^{})\in\calk_1^{(n)}(N\frkr_\frkl)\text{ for }l|N\}
\end{align*} 
be open compact subgroups of $G(\widehat{\QQ})$. 



\subsection{Modular forms on $\U(n)$}\label{ssec:51}
For each positive integer $\ell$ we define open subgroups of $\GL_n(\frkr_\frkp)$ by 
\begin{align*}
\cali_0^{(n)}(\frkp^\ell)&=\{(g_{ij})\in\GL_n(\frkr_\frkp)\;|\;g_{ij}\in \frkp^{\ell(i-j)}\text{ for }i>j\}, \\ 
\cali_1^{(n)}(\frkp^\ell)&=\{(g_{ij})\in\cali_0^{(n)}(\frkp^\ell)\;|\;g_{ii}-1\in \frkp^\ell\text{ for }i=1,2,\dots,n\} 
\end{align*}
and open compact subgroups of $G(\widehat{\QQ})$ by 
\begin{align*}
\calk_0(p^\ell\frkN)&=\{(g_q)\in\calk_0(\frkN)\;|\; \imath_\frkp(g_p^{})\in\cali_0^{(n)}(\frkp^\ell)\}, \\ 
\calk_1(p^\ell\frkN)&=\{(g_q)\in\calk_1(\frkN)\;|\; \imath_\frkp(g_p^{})\in\cali_1^{(n)}(\frkp^\ell)\},
\end{align*} 
where the open compact subgroups $\calk_0(\frkN)$ and $\calk_1(\frkN)$ of $G(\widehat{\QQ})$ are defined in \S \ref{ssec:21} (see \S \ref{ssec:26} for the definition of $\calk=\calk_0(\frkr)$). 

\begin{definition}\label{def:51}
Let $\ulk=(k_1,k_2,\dots,k_n)\in\ZZ^n$ be an $n$-tuple of integers such that $k_1\leq k_2\leq\dots\leq k_n$. 
For a commutative ring $A$ of characteristic $0$ we write $\call_{\ulk}$ for an $A$-module on which $\GL_n(A)$ acts as the irreducible representation $\rho_{\ulk}$ with highest weight $\Bbbk=(k_n,\dots,k_2,k_1)\in\ZZ^n$. 
\end{definition}

We define an embedding $\iot_\infty^T:G(\RR)\hookrightarrow\GL_n(\CC)$ by 
\begin{align*}
\iot_\infty^T(g)&=\sqrt{T}g\sqrt{T}^{-1}, &
\sqrt{T}&=\begin{bmatrix} \sqrt{t_1} & & \\ & \ddots & \\ & & \sqrt{t_n}\end{bmatrix}. 
\end{align*}
It is important to note that 
\[\iot_\infty^T(g^\tht)=\trs\iot_\infty^T(g)^{-1}. \]
We define a representation of $\rho_{\ulk,\infty}$ of $G(\RR)$ on $\call_{\ulk}(\CC)$ by $\rho_{\ulk,\infty}(g)=\rho_{\ulk}(\iot_\infty^T(g))$. 

\begin{definition}\label{def:52}
We call an $\call_{\ulk}(\CC)$-valued function $f$ on $G(\AA)$ a vector-valued modular form on $G$ of weight $\ulk$ and level $p^\ell\frkN$ if it satisfies 
\[f(\gam gu_\infty u)=\rho_{\ulk,\infty}(u_\infty)^{-1}f(g)\]
for $\gam\in G(\QQ)$, $g\in G(\widehat{\QQ})$, $u_\infty\in G(\RR)$, $u\in\calk_1(p^\ell\frkN)$. 
The space $\cala_{\ulk}^G(p^\ell\frkN)$ consists of vector-valued modular forms on $G(\AA)$ of weight $\ulk$ and level $p^\ell\frkN$. 
We say that $f\in \cala_{\ulk}^G(p^\ell\frkN)$ is $\overline{\QQ}$-rational if the restriction of $f$ to $G(\widehat{\QQ})$ takes values in $\call_{\ulk}(\iota_\infty(\overline{\QQ}))$. 
\end{definition}
Given $g\in G(\AA)$ and a function $\calf$ on $G(\QQ)\bsl G(\AA)$, we define another function $r(g)\calf$ on $G(\QQ)\bsl G(\AA)$ by 
\[[r(g)\calf](h)=\calf(hg). \]
For a character $\chi$ of $\calk_0(p^\ell\frkN)$ we set
\[\cala_{\ulk}^G(p^\ell\frkN,\chi)=\{f\in\cala_{\ulk}^G(p^\ell\frkN)\;|\;r(u)f=\chi(u)^{-1}f\text{ for }u\in\calk_0(p^\ell\frkN)\}. \]

Put $\ulk^\vee=(-k_n,\dots,-k_2,-k_1)$. 
Fix a perfect pairing 
\[\ell_{\ulk}:\call_{\ulk^\vee}^{}(\QQ)\otimes\call_{\ulk}(\QQ)\to \QQ. \]
Let $\scra(G)$ denotes the space of scalar valued modular forms on $G$. 
We associate to $f\in \cala_{\ulk}^G(p^\ell\frkN)$ and $\bfv\in\call_{\ulk^\vee}(\QQ)$ a scalar valued modular form $f_\bfv\in\scra(G)$ defined by 
\begin{align*}
f_\bfv(g)&=\ell_{\ulk}(\bfv\otimes f(g)), &
g&\in G(\AA). 
\end{align*}
A $G(\RR)$-equivariant map $\call_{\ulk^\vee}(\CC)\to\scra(G)$ is given by $\bfv\mapsto f_\bfv$. 
Let $\calt$ be the diagonal torus of $G$. 
If $\bfv_{\ulk}\in\call_{\ulk^\vee}(\CC)$ is a highest weight vector, then 
\[f_{\bfv_{\ulk}}(gt_\infty)=t_{\infty,1}^{-k_1}t_{\infty,2}^{-k_2}\cdots t_{\infty,n}^{-k_n}f_{\bfv_{\ulk}}(g)\]
for $g\in G(\widehat{\QQ})$ and $t_\infty=\diag(t_{\infty,1},t_{\infty,2},\dots,t_{\infty,n})\in\calt(\RR)$. 

Let $\d_\calk x$ be the Haar measure on $G(\AA)$ that gives $G(\RR)\calk$ volume $1$. 
We define the Petersson pairing of $\vph,\vph'\in\scra(G)$ by 
\[(\vph,\vph')_\calk=\int_{G(\QQ)\bsl G(\AA)}\vph(x)\vph'(x)\,\d_\calk x. \]
For any function $\calf$ on $G(\QQ)\bsl G(\AA)/G(\RR)\calk_0(p^\ell\frkN)$ we have 
\[\int_{G(\QQ)\bsl G(\AA)}\calf(x)\,\d_\calk x=\frac{1}{[\calk:\calk_0(p^\ell\frkN)]}\sum_{[x]\in G(\QQ)\bsl G(\widehat{\QQ})/\calk_0(p^\ell\frkN)}\frac{\calf(x)}{\sharp\Gam_{p^\ell\frkN,x}}, \]
where $[x]$ means the double coset $G(\QQ)x\calk_0(p^\ell\frkN)$ and 
\[\Gam_{p^\ell\frkN,x}=G(\QQ)\cap x\calk_0(p^\ell\frkN)x^{-1}. \]
We define a perfect pairing
\[(\;,\;)_{p^\ell\frkN}:\cala_{\ulk^\vee}^G(p^\ell\frkN,\chi^{-1})\times\cala_{\ulk}^G(p^\ell\frkN,\chi)\to\CC\]
by 
\[(f',f)_{p^\ell\frkN}=\sum_{[x]\in G(\QQ)\bsl G(\widehat{\QQ})/\calk_0(p^\ell\frkN)}\frac{\ell_{\ulk}(f'(x)\otimes f(x))}{\sharp\Gam_{p^\ell\frkN,x}}. \]
The Schur orthogonality relation gives 
\beq
(f'_\bfu,f_\bfv^{})_\calk=\frac{\ell_{\ulk}(\bfu,\bfv)}{[\calk:\calk_0(p^\ell\frkN)]\dim\call_{\ulk}}(f',f)_{p^\ell\frkN}. \label{tag:51}
\eeq

\begin{definition}\label{def:53}
Let $\pi$ be an irreducible automorphic representation of $G(\AA)$ generated by $f_{\bfv_{\ulk}}$ with a Hecke eigenform $f\in\cala_{\ulk}^G(p^\ell\frkN)$. 
Let $(\lam_1,\lam_2,\dots,\lam_n)$ be the Harish-Chandra parameter of $\pi_\infty$. 
Note that $\lam_i=-k_i+\frac{n+1}{2}-i$. 
Assume that $\pi_p$ is an irreducible principal series $V=I(\mu_1,\mu_2,\dots,\mu_n)$ of $\GL_n(\QQ_p)$ with locally algebraic characters $\iot_p\circ\mu_i:\QQ_p^\times\to\CC_p^\times$. 
Put $\alp_i=\iot_p(\mu_i(p))$. 
We order $\mu_i$ so that $\ord_p\alp_1\geq\ord_p\alp_2\geq\cdots\geq\ord_p\alp_n$. 
We say that $\pi$ is $p$-ordinary with respect to $\iot_p$ if $\ord_p\alp_i=-\lam_{n-i+1}$ for $i=1,2,\dots,n$ (cf. Conjecture 4.1 of \cite{CPR}). 
\end{definition}

\begin{definition}\label{def:54}
The operator $U_p$ on $f\in\cala_{\ulk}^G(p^\ell\frkN,\chi)$ is defined by 
\[U_pf=\sum_{\begin{smallmatrix} u=(u_{i,j})\in \caln_n\\
u_{i,j}\in\ZZ_p/p^{j-i}\ZZ_p\text{ for }i<j \end{smallmatrix}}r(\imath_\frkp^{-1}(uD_{n,p}))f, \]
where 
\[D_{n,p}=\begin{bmatrix} p^{n-1} & 0 & \dots & 0 & 0 \\ 0 & p^{n-2} & \dots & 0 & 0 \\ \vdots & \vdots & \ddots & \vdots & \vdots \\ 0 & 0 & \dots & p & 0 \\ 0 & 0 & \dots & 0 & 1 \end{bmatrix}. \]
\end{definition}

\begin{remark}\label{rem:51}
Put 
\[\alp_{\pi_p}=\displaystyle\prod_{i=1}^n\Big(\alp_ip^{i-\frac{n+1}{2}}\Big)^{i-1}=p^\frac{n(n^2-1)}{12}\prod_{i=1}^n\alp_i^{i-1}. \]
If $\pi$ is $p$-ordinary with respect to $\iot_p$, then $\ord_p\alp_i\neq\ord_p\alp_j$ for $i\neq j$, and so by \cite[Theorem 5.3]{Hid04}  and Proposition \ref{prop:31} below, the subspace 
\[V^{\ord}=\{h\in V\;|\; U_p h=\alp_{\pi_p}h\}\]
is spanned by the vector $h_{\pi_p}^{\ord}$ defined in \S \ref{ssec:33} (cf. \cite[Remark 6.3]{MH3}, \cite[Lemma 5.4]{Geraghty}). 
\end{remark}

Define an automorphism of $\GL_n(A)$ by $g^\vth=\trs g^{-1}$. 
Let $^\vth:\call_{\ulk}(A)\simeq\call_{\ulk^\vee}(A)$ be the isomorphism such that for $\bfu_1,\bfu_2\in\call_{\ulk}(\QQ)$
\begin{align}
(\rho_{\ulk}(g^\vth)\bfu_1^{})^\vth&=\rho_{\ulk^\vee}(g)\bfu_1^\vth, &
\ell_{\ulk}(\bfu_1^\vth\otimes\bfu_2^{})&=\ell_{\ulk^\vee}(\bfu_2^\vth\otimes\bfu_1^{}), \label{tag:52}
\end{align}

\begin{Remark}\label{rem:52}
Let $\bfu_{\ulk}$ be the lowest weight vector in $\call_{\ulk}(\QQ)$ that satisfies $\ell_{\ulk}(\bfv_{\ulk}\otimes\bfu_{\ulk})=1$. 
Define 
\[\vep_{\ulk}\in\{\pm 1\}\text{ so that }\bfu_{\ulk}^\vth=\vep_{\ulk}^{}\bfv_{\ulk}^{}. \]
In our future application, we will have \[\vep_{\ulk}=(-1)^{k_1-k_2}\text{ if }n=2;\quad \vep_{\ulk}=1\text{ if }n=3\] (see (\ref{tag:61}), (\ref{tag:62}), \S \ref{ssec:b2} and \S \ref{ssec:b4}). 
\end{Remark}

We define $\xi_{p^\ell\frkN}=(\xi_{p^\ell\frkN,l})\in G(\widehat{\QQ})$ by setting $\xi_{p^\ell\frkN,p}=\imath_\frkp^{-1}(D_{n,p}^{-\ell}T)$ and $\xi_{p^\ell\frkN,l}=\xi_{\frkN,l}$ for $l\neq p$. 
Observe that 
\beq
(\xi_{p^\ell\frkN}^{-1}\calk_0(p^\ell\frkN)\xi_{p^\ell\frkN}^{})^\tht=\calk_0(p^\ell\frkN). \label{tag:53}
\eeq
Given $f\in\cala_{\ulk}^G(p^\ell\frkN,\chi)$, we define $\call_{\ulk^\vee}(\CC)$-valued function $f^\tht$ on $G(\QQ)\bsl G(\AA)$ by $f^\tht(g)=(f(g^\tht))^\vth$. 
One can obtain $\tau_{p^\ell\frkN}f\in\cala_{\ulk^\vee}^G(p^\ell\frkN,\chi^{-1})$ by  
\[\tau_{p^\ell\frkN}f:=r(\xi_{p^\ell\frkN})f^\tht \]
in view of (\ref{tag:53}). 
It follows from (\ref{tag:52}) that 
\begin{align}
(\tau_{p^\ell\frkN}f)_{\bfu_{\ulk}}(g)
&=\ell_{\ulk^\vee}((r(\xi_{p^\ell\frkN}^\tht)f(g^\tht))^\vth\otimes\bfu_{\ulk}) \notag\\
&=\ell_{\ulk}(\bfu_{\ulk}^\vth\otimes r(\xi_{p^\ell\frkN}^\tht)f(g^\tht)) 
=\vep_{\ulk}f_{\bfv_{\ulk}}((g\xi_{p^\ell\frkN})^\tht). \label{tag:54}
\end{align}


\subsection{$p$-adic modular forms on $\U(n)$}\label{ssec:52}

Define an embedding $\iot_\frkp:G(\QQ_p)\hookrightarrow\GL_n(\CC_p)$ and a representation $\rho_{\ulk,p}$ of $G(\QQ_p)$ on $\call_{\ulk}(\CC_p)$ by 
\begin{align*}
\iot_p^T(g)&=\sqrt{T}\imath_\frkp^{}(g)\sqrt{T}^{-1}, &
\rho_{\ulk,p}&=\rho_{\ulk}(\iot_p^T(g)). 
\end{align*} 
By definition $\iot_p^T$ is compatible with $\iot_\infty^T$ in the sense that  
\begin{align*}
\iot_p^T(g^\tht)&=\iot_p^T(g)^\vth, & 
\iot_p^T(\gam)&=\iot_\infty^T(\gam) 
\end{align*}
for $g\in G(\QQ_p)$ and $\gam\in G(\QQ)$. 
Given $g\in G(\widehat{\QQ})$ and a function $\calf:G(\QQ)\bsl G(\widehat{\QQ})\to\call_{\ulk}(\CC_p)$, we define a function $r_{\ulk}(g)\calf$ on $G(\QQ)\bsl G(\widehat{\QQ})$ by 
\[[r_{\ulk}(g)\calf](h)=\rho_{\ulk,p}(g_p)\calf(hg). \]
We associate a function $\widehat{f}:G(\QQ)\bsl G(\widehat{\QQ})\to \call_{\ulk}(\CC_p)$ defined by 
\[\widehat{f}(g)=\rho_{\ulk,p}(g_p)^{-1}\iot_p(f(g))\]
to a function $f:G(\QQ)\bsl G(\AA)\to \call_{\ulk}(\CC)$. 
Notice that 
\begin{align}
\widehat{r(g)f}&=r_{\ulk}(g)\widehat{f}, &
f(gu_\infty)&=\rho_{\ulk,\infty}(u_\infty)^{-1}\iot_p^{-1}(\rho_{\ulk,p}(g_p)\widehat{f}(g))  \label{tag:55}
\end{align}
for $g\in G(\widehat{\QQ})$ and $u_\infty\in G(\RR)$. 

\begin{definition}\label{def:55}
Let $\chi$ be an $A$-valued character of $\calk_0(p^\ell\frkN)/\calk_1(p^\ell\frkN)$. 
The space $\cals_{\ulk}^G(p^\ell\frkN,\chi,A)$ of $p$-adic modular forms on $G$ of weight $\ulk$, level $p^\ell\frkN$ and nebentypus $\chi$ over $A$ consists of vector-valued functions $\widehat{f}:G(\widehat{\QQ})\to \call_{\ulk}(A)$ such that for $\gam\in G(\QQ)$, $g\in G(\widehat{\QQ})$ and $u\in\calk_0(p^\ell\frkN)$ 
\[\widehat{f}(\gam gu)=\rho_{\ulk,p}(u_p)^{-1}\widehat{f}(g)\chi(u)^{-1}, \]
where we denote the $p$-component of $u$ by $u_p$. 
\end{definition}

If $f\in\cala_{\ulk}^G(p^n\frkN,\chi)$, then $\widehat{f}\in\cals_{\ulk}^G(p^\ell\frkN,\chi,\CC_p)$, and $\widehat{f}$ is called the $p$-adic avatar of $f$. 
On the other hand, we will call $f$ the ad\`{e}lic lift of $\widehat{f}$. 
Given $\widehat{f}\in\cals_{\ulk}^G(p^\ell\frkN,\chi,\CC_p)$ and $\bfv\in\call_{\ulk^\vee}(\QQ)$, we define a scalar valued function $\widehat{f}_\bfv$ on $G(\QQ)\bsl G(\widehat{\QQ})/\calk_1(p^\ell\frkN)$ by 
\begin{align*}
\widehat{f}_\bfv(g)&=\ell_{\ulk}(\bfv\otimes\widehat{f}(g)), &
g&\in G(\widehat{\QQ}). 
\end{align*}
If $\bfv_{\ulk}\in\call_{\ulk^\vee}(\QQ)$ is a highest weight vector, then 
\[\widehat{f}_{\bfv_{\ulk}}(gt_p)=t_{p,1}^{-k_1}t_{p,2}^{-k_2}\cdots t_{p,n}^{-k_n}\widehat{f}_{\bfv_{\ulk}}(g)\chi(t_p)^{-1}\]
for $g\in G(\widehat{\QQ})$ and $\imath_\frkp(t_p)=\diag(t_{p,1},t_{p,2},\dots,t_{p,n})$ with $t_{p,i}\in\ZZ_p^\times$. 

We define a perfect pairing
\[(\;,\;)_{p^\ell\frkN}:\cals_{\ulk^\vee}^G(p^\ell\frkN,\chi^{-1},\CC_p)\times\cals_{\ulk}^G(p^\ell\frkN,\chi,\CC_p)\to\CC_p\]
by 
\[(\widehat{f}',\widehat{f})_{p^\ell\frkN}=\sum_{[x]\in G(\QQ)\bsl G(\widehat{\QQ})/\calk_0(p^\ell\frkN)}\frac{\ell_{\ulk}(\widehat{f}'(x)\otimes\widehat{f}(x))}{\sharp\Gam_{p^\ell\frkN,x}}. \]
By definition, \beq
(\widehat{f}',\widehat{f})_{p^\ell\frkN}=\iot_p((f',f)_{p^\ell\frkN}). \label{tag:56}
\eeq

To an $\call_{\ulk}(\CC_p)$-valued function $f$ on $G(\QQ)\bsl G(\widehat{\QQ})$ we associate an $\call_{\ulk^\vee}(\CC_p)$-valued function $\widehat{f}^\tht$ on $G(\QQ)\bsl G(\widehat{\QQ})$ defined by $\widehat{f}^\tht(g)=(\widehat{f}(g^\tht))^\vth$. 
One easily sees that $\widehat{f}^\tht=\widehat{f^\tht}$. 
We associate to $f\in\cala_{\ulk}^G(p^\ell\frkN,\chi)$ a $p$-adic modular form $\tau_{p^\ell\frkN}\widehat{f}$ of weight $\ulk^\vee$ and level $p^\ell\frkN$ defined by 
\beq
\tau_{p^\ell\frkN}\widehat{f}
:=r_{\ulk^\vee}(\xi_{p^\ell\frkN})\widehat{f}^\tht
=\widehat{\tau_{p^\ell\frkN}f}. \label{tag:57}
\eeq  
We define the operator $U_p$ on $\widehat{f}\in\cals_{\ulk}^G(p^\ell\frkN,\chi,\CC_p)$ by 
\[[U_p\widehat{f}](g)=\sum_{\begin{smallmatrix} u=(u_{i,j})\in N_n\\ u_{i,j}\in\ZZ_p/p^{j-i}\ZZ_p\text{ for }i<j \end{smallmatrix}}\rho_{\ulk,p}(\imath_\frkp^{-1}(uD_{n,p}))\widehat{f}(g\imath_\frkp^{-1}(uD_{n,p})). \]
Put 
\beq
\calu_p=p^{-(n-1)k_1-(n-2)k_2-\cdots-k_{n-1}}\cdot U_p. \label{tag:58}
\eeq
Then 
\begin{align}
\widehat{U_pf}&=U_p\widehat{f}, & 
[\calu_p\widehat{f}]_{\bfv_{\ulk}}&=\calu_p\widehat{f}_{\bfv_{\ulk}}. \label{tag:59}
\end{align}

\begin{definition}
We define an operator $\calu_p$ on the space of $\CC_p$-valued functions on $G(\QQ)\bsl G(\widehat{\QQ})/\imath_\frkp^{-1}(N_n(\ZZ_p))$ by
\[[\calu_p\calf](g)=\sum_{\begin{smallmatrix} u=(u_{i,j})\in N_n\\ 
u_{i,j}\in\ZZ_p/p^{j-i}\ZZ_p\text{ for }i<j \end{smallmatrix}}\calf(g\imath_\frkp^{-1}(uD_{n,p})). \]
\end{definition}

\begin{remark}\label{rem:53}
By definition 
\begin{align*}
[\calu_p\widehat{f}_{\bfv_{\ulk}}](g)
&=\sum_u\ell_{\ulk}(\rho_{\ulk^\vee,p}(g_p\imath_\frkp^{-1}(uD_{n,p}))\bfv_{\ulk}\otimes\iot_p(f(g\imath_\frkp^{-1}(uD_{n,p}))))\\
&=p^{-(n-1)k_1-(n-2)k_2-\cdots-k_{n-1}}\ell_{\ulk}(\rho_{\ulk^\vee,p}(g_p)\bfv_{\ulk}\otimes\iot_p([U_pf](g))). 
\end{align*}
If $\pi$ is a $p$-ordinary irreducible automorphic representation of $G(\AA)$ with respect to $\iot_p$, then Proposition \ref{prop:31} below gives an eigenform $f\in\cala_{\ulk}^G(p^\ell\frkN)$ of $U_p$ attached to $\pi$ such that $\widehat{f}_{\bfv_{\ulk}}$ is an eigenform of $\calu_p$ with unit eigenvalue 
\[\alp_f=p^{-(n-1)k_1-(n-2)k_2-\cdots-k_{n-1}}\alp_{\pi_p}. \]  
\end{remark}

If $A$ is $p$-adically complete, then the ordinary projector $e_{\ord}=\displaystyle\lim_{m\to\infty}\calu_p^{m!}$ converges to an idempotent in $\End_A\cals_{\ulk}(p^n\frkN,\chi,A)$.  


\subsection{Hida families on $\U(n)$}\label{ssec:53}
We define $X_1^G(p^\ell\frkN)$ as the finite set
\[X_1^G(p^\ell\frkN)=G(\QQ)\bsl G(\widehat{\QQ})/\calk_1(p^\ell\frkN) \]
for each positive integer $\ell$. Recall that $\calo$ is the ring of the integers of a finite extension $F$ of $\QQ_p$. Let $\calo[X_1^G(p^\ell\frkN)]=\oplus_{x\in X_1^G(p^\ell\frkN)}\calo x$ be the finitely generated $\calo$-module spanned by divisors of $X_1^G(p^\ell\frkN)$. 
Put 
\[X_1^G(p^\infty\frkN):=\displaystyle\lim_{\longleftarrow \ell}X_1^G(p^\ell\frkN). \]

We retain the notation from the introduction. Let $\bfI$ be a normal ring finite flat over $\Lam_n=\calo\powerseries{T_n(\ZZ_p)}$. 
Write $z\mapsto[z]_{\Lam_n}$ for the inclusion of group-like elements $T_n(\ZZ_p)\to\calo\powerseries{T_n(\ZZ_p)}^\times=\Lam_n^\times$. 
For $z\in T_n(\ZZ_p)$ we denote the image of $[z]_{\Lam_n}$ in $\bfI$ under the structure morphism $\Lam_n\to\bfI$ by $[z]_\bfI$. 


For $\ulQ=(Q_1,Q_2,\dots,Q_n)\in\frkX_\calr$ we put 
\begin{align*}
k_{\ulQ}&=(k_{Q_1},k_{Q_2},\dots,k_{Q_n}), & 
\ell&=\max\{1,c(\eps_{Q_1}),c(\eps_{Q_2}),\dots,c(\eps_{Q_n})\}
\end{align*}
and define finite order characters of $\calk_0(p^\ell\frkN)$ by
\begin{align*}
\eps_{\ulQ}(t_p)&=\eps_{Q_1}(t_{p,1})\eps_{Q_2}(t_{p,2})\cdots\eps_{Q_n}(t_{p,n}), 
\end{align*}
where $\imath_\frkp(t_p)=\diag(t_{p,1},t_{p,2},\dots,t_{p,n})$ with $t_{p,i}\in\ZZ_p^\times$ for $i=1,2,\dots,n$. 

Let $P^{(n)}_\ell$ be the ideal of $\Lam_n$ generated by $[t]_{\Lam_n}^{p^\ell}-1$ for $t\in T_n(\ZZ_p)$ 
Let the ring $\Lam_n$ act on $\calo[X_1^G(p^\ell\frkN)]$ by 
\[[t]_{\Lam_n} x:=x\cdot \imath_\frkp^{-1}(t),\quad t\in T_n(\ZZ_p).\]

\begin{definition}\label{def:56}
Put $\Del=(\frkr/\frkN)^\times$. 
For $d\in\Del$ the diamond operator $\sig_d$ acts on the module $\calo[X_1^G(p^\ell\frkN)]$ by 
\[\sig_dx:=x\prod_{\frkl|\frkN}\imath_\frkl^{-1}\left(\begin{bmatrix} \ono_{n-1} & \\ & \til d_l\end{bmatrix}\right), \]
where $\til d=(\til d_l)\in\prod_{l|N}\ZZ_l^\times$ is a lift of $d$. 
\end{definition}

Thus $\calo[X_1^G(p^\ell\frkN)]$ is a finitely generated $\Lam_n[\Del]$-module. 

\begin{definition}\label{def:57}
The module $\calo[X_1^G(p^\ell\frkN)]$ is equipped with the operator $\calu_p$ defined by 
{\small\[\calu_p x=\sum_{u_{i,j}\in\ZZ_p/p^{j-i}\ZZ_p}x\cdot\begin{bmatrix} p^{n-1} & p^{n-2}u_{1,2} & \dots & p^2u_{1,n-2} & pu_{1,n-1} & u_{1,n}  \\ 0 & p^{n-2} & \dots & p^2u_{2,n-2} & pu_{2,n-1} & u_{2,n} \\ \vdots & \vdots & \ddots & \vdots & \vdots & \vdots \\ 0 & 0 & \dots & p^2 & pu_{n-2,n-1} & u_{n-2,n} \\ 0 & 0 & \dots & 0 & p & u_{n-1,n} \\ 0 & 0 & \dots & 0 & 0 & 1 \end{bmatrix}, \]}
where we define the action of $g\in\GL_n(\QQ_p)$ on $x\in X_1^G(p^\ell\frkN)$ by $x\cdot g=x\imath_\frkp^{-1}(g)$. 
The limit 
\[e_{\ord}=\displaystyle\lim_{m\to\infty}\calu_p^{m!}\] 
converges to an idempotent in $\End_{\Lam_n}\calo[X_1^G(p^\ell\frkN)]$. 
\end{definition}

\begin{definition}\label{def:58}
A $\Lam_n$-adic modular form on $G$ of tame level $\frkN$ is a function $\bdsf:X_1^G(p^\infty\frkN)\to\Lam_n$ which satisfies 
\begin{align*}
\bdsf(x\cdot \imath_\frkp^{-1}(t))&=\bdsf(x)[t]_{\Lam_n}^{-1}, & 
t&\in T_n(\ZZ_p)
\end{align*} 
such that the function $\bdsf\pmod{P_\ell^{(n)}}:X_1^G(p^\infty\frkN)\to\Lam_n/P_\ell^{(n)}$ factors through $X_1^G(p^\ell\frkN)$ for any $\ell$ sufficiently large. 
Let $\bfS^G(\frkN,\Lam_n)$ be the space of $\Lam_n$-adic modular forms on $G$ of tame level $\frkN$. 
\end{definition}

The $\Lam_n$-module $\bfS^G(\frkN,\Lam_n)$ is equipped with the natural actions of Hecke and diamond operators given by 
\begin{align*}
\calu_p\bdsf(x)&=\bdsf(\calu_p x), &
\sig_d\bdsf(x)&=\bdsf(\sig_d x). 
\end{align*}
The ordinary projector $e_{\ord}=\displaystyle\lim_{m\to\infty}\calu_p^{m!}$ converges in $\End_{\Lam_n}\bfS^G(\frkN,\Lam_n)$.  
For a character $\chi:\Del\to\calo^\times$ we put 
\[\bfS^G(\frkN,\chi,\Lam_n)
=\{\bdsf\in\bfS^G(\frkN,\Lam_n)\;|\;\sig_d\bdsf=\chi(d)^{-1}\bdsf\text{ for }d\in\Del\}. \]
For a normal ring $\calr$ finite flat over $\Lam_n$ we set 
\begin{align*}
\bfS^G(\frkN,\calr)&=\bfS^G(\frkN,\Lam_n)\otimes_{\Lam_n}\calr, \\ 
\bfS^G(\frkN,\chi,\calr)&=\bfS^G(\frkN,\chi,\Lam_n)\otimes_{\Lam_n}\calr. 
\end{align*}

\begin{theorem}\label{thm:51}
Put $\Nr_\chi=\sum_{d\in\Del}\chi(d)\sig_d\in\calo[\Del]$. 
Let $P_\chi$ be the ideal of $\calr[\Del]$ generated by $\{\chi(d)\sig_d-1\}_{d\in\Del}$. 
Suppose that $p>3$. 
\begin{enumerate}
\item $\bfS^G(\frkN,\chi,\calr)$ is a free $\calr$-module, and the norm map 
\[\Nr_\chi:e_{\ord}\bfS^G(\frkN,\calr)/P_\chi\simeq e_{\ord}\bfS^G(\frkN,\chi,\calr)\] 
is an isomorphism. 
\item For every $\ulQ\in\frkX_\calr^+$ we have a Hecke equivariant isomorphism 
\begin{align*}
e_{\ord}\bfS^G(\frkN,\chi,\calr)\otimes_\calr\calr/\wp_{\ulQ}
&\simeq e_{\ord}\cals^G_{k_{\ulQ}}(p^\ell\frkN,\chi\eps_{\ulQ},\calr(\ulQ))\\
\bdsf\pmod{\wp_{\ulQ}}&\mapsto\bdsf_{\ulQ}, 
\end{align*}
where $\bdsf_{\ulQ}$ is the unique $p$-adic modular form such that 
\[\ulQ(\bdsf(x))=(\bdsf_{\ulQ})_{\bfv_{\ulk_{\ulQ}}}(x). \]
\end{enumerate}
\end{theorem} 
\begin{proof}This is essentially proved in \cite{Geraghty} by adapting the arguments in \cite{Hida88Annals} in the case of $\GL(2)$ to the case of unitary groups. 
For any abelian group $A$, let $\cals_0(\alpha,A)$ be the space of $A$-valued functions on the finite set $X_1^G(p^\alpha\frkN)$. 
Let $\scrv^{\ord}(\frkN):=\varinjlim_\beta\varinjlim_\alpha e_{\ord}\cals_0(\alpha,p^{-\beta}\calo/\calo)$ be the discrete $\Lam_n$-module. 
Note that $\scrv^{\ord}(\frkN)$ is nothing but the space $\cals_{0,\{\chi_v\}}^{\ord}(U(\frkl^\infty),K/\calo)$ in \cite[p.~ 1358]{Geraghty}. 
Let \[V^{\ord}(\frkN)=\varprojlim_\alpha e_{\ord}\calo[X_1^G(p^\alpha\frkN)]\] be the Pontryagin dual of $\scrv^{\ord}(\frkN)$. Then we have \[\bfS^G(\frkN,\Lam_n)=\Hom_{\Lam_n}(V^{\ord}(\frkN),\Lam_n). \] 
Therefore (1) follows from \cite[Proposition 2.20]{Geraghty} and (2) is proved in \cite[Proposition 2.22, Lemma 2.25]{Geraghty}. 
\end{proof}

\begin{definition}[Hida families]\label{Def:Hida}
A non-zero $\calr$-adic modular form $\bdsf\in e_{\ord}\bfS^G(\frkN,\chi,\calr)$ is an $\calr$-adic Hida family if $\bdsf_{\ulQ}$ is a simultaneous eigenform of Hecke operators for $\ulQ\in\frkX_\calr^{++}$. 
Let $\frkX_{\bdsf}^{++}=\{\ulQ\in\frkX_\calr^{++}\;|\;\ulQ(\bdsf)\neq 0\}$ be a Zariski dense subset of $\frkX_\calr^{++}$.  
\end{definition}
\begin{lemma}\label{lem:51}
If $\bdsf\in e_{\ord}\bfS^G(\frkN,\chi,\calr)$ is a $\calr$-adic Hida family, then $\ulQ(\bdsf)$ generates an irreducible $p$-ordinary automorphic representation of $G(\AA)$ for $\ulQ\in\frkX_{\bdsf}^{++}$. 
\end{lemma}

\begin{proof}
Let $\pi\simeq\otimes_v'\pi_v^{}$ be an irreducible constituent of the automorphic representation generated by $\ulQ(\bdsf)$. 
If $l$ and $Np$ are coprime, then $\pi_l$ is determined by the Hecke eigenvalues of $\ulQ(\bdsf)$. 
Thus $\pi$ belongs to the A-packet associated to these eigenvalues. 
Moreover, $\pi_q$ is uniquely determined for each prime factor $q$ of $N$ by the assumption on $N$ as the associated local A-packet is a singleton for each split prime. 
Therefore the equivalence class $\pi$ is determined by $\ulQ(\bdsf)$. 
Thus $\ulQ(\bdsf)$ generates an irreducible representation of $G(\AA)$ by the multiplicity one for unitary groups (cf. \cite[Theorem 13.3.1]{Rog90}). 
\end{proof}

If $\bdsf$ is a Hida family, then $\ulQ(\bdsf)$ is an eigenform of the operator $\calu_p$ with unit eigenvalue $\alp_{\bdsf_{\ulQ}}$ by Remarks \ref{rem:51} and \ref{rem:53} for $\ulQ\in\frkX_{\bdsf}^{++}$, and we denote by $\bdpi_{\ulQ}$ the automorphic representation of $G(\AA)$ associated to $\bdsf_{\ulQ}$, which is $p$-ordinary with respect to $\iot_p$. 


\subsection{A paring on the space of ordinary $\calr$-adic modular forms}\label{ssec:55}
In this subsection, we do the $p$-adic interpolation of the bilinear pairing. We first introduce the regularized diagonal cycles for $\U(n)\times\U(n)$.

Define the finite sets 
\begin{align*}
\calx_\ell^{}&=X_0^G(p^\ell\frkN), & 
\XX_\ell^{}&=\calx_\ell^{}\times\calx_\ell^{}
\end{align*}
for each positive integer $\ell$. 
Given $x,y\in G(\widehat{\QQ})$, we write $[(x,y)]\in\XX_\ell$ for the double coset represented by $(x,y)$. 
The following definition makes sense in view of (\ref{tag:53}). 

\begin{definition}\label{def:59}
Let $\Diamond_\ell\in\ZZ_p[\XX_\ell]$ be the twisted diagonal cycle defined by 
\[\Diamond_\ell=\sum_{[x]\in\calx_\ell}[(x,(x\xi_{p^\ell\frkN})^\tht)]. \]
\end{definition}

The homomorphism 
\[\Nr_{\ell+1,\ell}:\ZZ_p[\XX_{\ell+1}]\to\ZZ_p[\XX_\ell]\] 
is induced by the projection $\XX_{\ell+1}\to\XX_\ell$. 

\begin{lemma}\label{lem:52}
For $\ell\geq 1$ we have 
\[\Nr_{\ell+1,\ell}(\Diamond_{\ell+1})=(1\otimes\calu_p)\Diamond_\ell.\]
\end{lemma}

\begin{proof} 
We abbreviate $S_\ell=\ZZ_p/p^\ell\ZZ_p$. 
Recall the unipotent radical $N_n^-$ of the Borel subgroup opposite to $B_n$. 
Put 
\[\imath_\frkp(\vSi_\ell):=\bigl\{\bigl(p^{(i-j)\ell}v_{ij}\bigl)\in N_n^-(\QQ_p)\;\bigl|\;v_{ij}\in S_{i-j}\text{ for }i>j\bigl\}. \]
Then $\vSi_\ell$ is a complete set of representatives for $\calk_0(p^\ell\frkN)/\calk_0(p^{\ell+1}\frkN)$. 
We may assume that $\Gam_{p^\ell\frkN,x}=\{1\}$ for every $x\in G(\widehat{\QQ})$ (see the proof of Lemma 4.4 of \cite{MH}). 
Then $\calx_{\ell+1}$ consists of elements represented by $xk$ with $x\in\calx_\ell$ and $k\in\vSi_\ell^{}$. 
Since 
\begin{align*}
\imath_\frkp^{}((\xi_{p^\ell\frkN,p}^{-1}\vSi_\ell^{}\xi_{p^\ell\frkN,p}^{})^\tht)
&=\trs(D_{n,p}^\ell\imath_\frkp^{}(\vSi_\ell^{})D_{n,p}^{-\ell})^{-1}\\
&=\{u\in N_n(\QQ_p)\;|\;u_{ij}\in S_{j-i}\text{ for }i<j\bigl\}, 
\end{align*}
we get the distribution property stated above. 
\end{proof}

Define the regularized diagonal cycle by $\Diamond_\ell^\dagger=(1\otimes\calu_p^{-\ell})e_{\ord}^{}\Diamond_\ell^{}$. 
Lemma \ref{lem:52} says that $\Nr_{\ell+1,\ell}^{}(\Diamond_{\ell+1}^\dagger)=\Diamond_\ell^\dagger$. 
We can therefore define 
\[\Diamond_\infty^\dagger=\lim_{\longleftarrow \ell}\Diamond_\ell^\dagger\in\lim_{\longleftarrow \ell}\ZZ_p[\XX_\ell]. \]


\begin{definition}\label{def:510}
Let $\bdsf,\bdsg\in e_{\ord}\bfS^G(\frkN,\chi,\calr)$. 
We define an $\calr$-adic modular form on $G\times G$ by 
\[\bdscF=\bdsf\boxtimes\bdsg:G(\QQ)\bsl G(\widehat{\QQ})\times G(\QQ)\bsl G(\widehat{\QQ})\to\calr\] 
by $\bdscF(x,y)=\bdsf(x)\bdsg(y)$. 
Then $\bdscF$ naturally induces a $\Lam_n$-linear map 
\[\bdscF_*:\displaystyle\lim_{\longleftarrow \ell} \calo[\XX_\ell]\to\calr. \]  
Define an $\calr$-bilinear pairing 
\[\bfB_\frkN: e_{\ord}\bfS^G(\frkN,\chi,\calr)\times e_{\ord}\bfS^G(\frkN,\chi,\calr)\to\calr\]
by   
\[\bfB_\frkN(\bdsf,\bdsg)=\bdscF_*(\Diamond_\infty^\dagger)\in\calr. \]
\end{definition}

The following result generalizes \cite[Lemma 4.4]{MH}. 

\begin{proposition}\label{prop:51}
For each $\ulQ\in\frkX_\calr^+$ and sufficiently large $\ell$ we have 
\[\ulQ(\bfB_\frkN(\bdsf,\bdsg))=\vep_{k_{\ulQ}}\cdot (\tau_{p^\ell\frkN}[U_p^{-\ell}\bdsg_{\ulQ}],\bdsf_{\ulQ})_{p^\ell\frkN}. \] 
Here $\vep_{k_{\ulQ}}\in\{\pm 1\}$ is the sign defined in Remark \ref{rem:52}.
\end{proposition}

\begin{proof}
To lighten notation, we put 
\begin{align*}
\ulk&=k_{\ulQ}, & \kap&=(n-1)k_{Q_1}+(n-2)k_{Q_2}+\cdots+k_{Q_{n-1}}, \\
\widetilde{D}_{n,p}&=\imath_\frkp^{-1}(D_{n,p}), & 
\widehat{f}
&=\bdsf_{\ulQ},\;\widehat{g}=\bdsg_{\ulQ}\in e_{\ord}\cals^G_{\ulk}(p^e\frkN,\chi\eps_{\ulQ},\calr(\ulQ)). 
\end{align*}
If $l$ is sufficiently large, then $\Gam_{p^l\frkN,x}=\{1\}$ for every $x\in G(\widehat{\QQ})$, and 
\begin{align*}
&(\tau_{p^{l+1}\frkN}[\calu_p^{-l-1}\widehat{g}],\widehat{f})_{p^{l+1}\frkN}\\
=&\sum_{[x]\in\calx_{l+1}}\ell_{\ulk}(\rho_{\ulk^\vee,p}(\xi_{p^{l+1}\frkN,p})([\calu_p^{-l-1}\widehat{g}]((x\xi_{p^{l+1}\frkN})^\tht))^\vth\otimes\widehat{f}(x))\\
=&\sum_{[x]\in\calx_l}\sum_{v\in\vSi_l}\ell_{\ulk}(\rho_{\ulk^\vee,p}(\xi_{p^{l+1}\frkN,p})([\calu_p^{-l-1}\widehat{g}]((xv\xi_{p^{l+1}\frkN})^\tht))^\vth\otimes\rho_{\ulk,p}(v)^{-1}\widehat{f}(x))\\
=&\sum_{[x]\in\calx_l}\sum_{v\in\vSi_l}\ell_{\ulk}((\rho_{\ulk,p}((v\xi_{p^l\frkN,p})^\tht\widetilde{D}_{n,p})[\calu_p^{-l-1}\widehat{g}]((xv\xi_{p^l\frkN})^\tht\widetilde{D}_{n,p}))^\vth\otimes\widehat{f}(x)). 
\end{align*}
The last sum is $p^\kap(\tau_{p^l\frkN}[\calu_p^{-l}\widehat{g}],\widehat{f})_{p^l\frkN}$ by the proof of Lemma \ref{lem:52}. 
Thus the value $(\tau_{p^l\frkN}[U_p^{-l}\widehat{g}],\widehat{f})_{p^l\frkN}$ is independent of any sufficiently large integer $l$ by the relation (\ref{tag:58}). 

Recall that $\ulQ(\bdsf)=\widehat{f}_{\bfv_{\ulk}}$ and $\ulQ(\bdsg)=\widehat{g}_{\bfv_{\ulk}}$.  
For an arbitrarily large integer $m$ there exists a sufficiently larger integer $l$ such that   
\[\ulQ(\bfB_\frkN(\bdsf,\bdsg))\equiv\sum_{[x]\in\calx_l}[\calu_p^{-l}\widehat{g}_{\bfv_{\ulk}}]((x\xi_{p^l\frkN})^\tht)\widehat{f}_{\bfv_{\ulk}}(x) \pmod{p^m}. \]
Take a basis $\calb_{\ulk^\vee}=\{\bfv_i\}$ of $\call_{\ulk^\vee}(\QQ)$ which consists of weight vectors and contains the highest weight vector $\bfv_{\ulk}$. 
Let $\calb_{\ulk}=\{\bfu_i\}$ be a basis of $\call_{\ulk}(\QQ)$ dual to $\calb_{\ulk^\vee}$. 
Then we can write 
\[\calu_p^{-l}\widehat{g}=\sum_i[\calu_p^{-l}\widehat{g}]_{\bfv_i}\cdot \bfu_i. \] 
If $l$ is sufficiently larger than $m$, then  
\begin{align*}
\rho_{\ulk}(D_{n,p}^l)[U_p^{-l}\widehat{g}]
&=\sum_i[\calu_p^{-l}\widehat{g}]_{\bfv_i}\cdot p^{-\kap l}\rho_{\ulk}(D_{n,p}^l)\bfu_i\\
&\equiv [\calu_p^{-l}\widehat{g}]_{\bfv_{\ulk}}\cdot \bfu_{\ulk} \pmod{p^m}.  
\end{align*}
It therefore follows that 
\begin{align*}
(\tau_{p^l\frkN}[U_p^{-l}\widehat{g}],\widehat{f})_{p^l\frkN}
=&\sum_{[x]\in\calx_l}\ell_{\ulk}((\rho_{\ulk,p}(\xi_{p^l\frkN,p}^\tht)[U_p^{-l}\widehat{g}]((x\xi_{p^l\frkN})^\tht))^\vth\otimes\widehat{f}(x)) \\
\equiv&\sum_{[x]\in\calx_l}[\calu_p^{-l}\widehat{g}_{\bfv_{\ulk}}]((x\xi_{p^l\frkN})^\tht)\ell_{\ulk}^{}(\bfu_{\ulk}^\vth\otimes\widehat{f}(x))\pmod{p^m}
\end{align*}
by (\ref{tag:59}). 
Thus 
\[(\tau_{p^\ell\frkN}[U_p^{-\ell}\widehat{g}],\widehat{f})_{p^\ell\frkN}
=(\tau_{p^l\frkN}[U_p^{-l}\widehat{g}],\widehat{f})_{p^l\frkN}\equiv \vep_{k_{\ulQ}}\ulQ(\bfB_\frkN(\bdsf,\bdsg))\pmod{p^m}\] 
(see Remark \ref{rem:52}) for arbitrarily large $m$, from which the formula follows. 
\end{proof}



\begin{definition}\label{def:511}
Let $\bdsf\in e_{\ord}\bfS^G(\frkN,\chi,\calr)$ be an $\calr$-adic Hida family. 
We define $\eta_{\bdsf}\in\calr$ by 
\[\eta_{\bdsf}=\bfB_\frkN(\bdsf,\bdsf). \]
\end{definition}

\begin{proposition}\label{prop:52} 
Suppose that $\bdsf$ is an eigenvector of the $U_p$-operator with the eigenvalue $\alpha_{\bdsf}\in \calr^\times$. 
For $\ulQ\in\frkX_{\calr}^+$ and $\ell\gg 0$, we have 
\[\ulQ(\eta_{\bdsf})
=\vep_{k_{\ulQ}}[\calk:\calk_0(p^\ell\frkN_{\bdpi_{\ulQ}})]\ulQ(\alp_{\bdsf})^{-\ell}\iot_p((\tau_{p^\ell\frkN}f_{\bfv_{k_{\ulQ}}},f_{\bfv_{k_{\ulQ}}})_\calk)\dim\call_{k_{\ulQ}},  \] 
where $f\in\cala_{\ulk_{\ulQ}}^G(p^\ell\frkN,\chi)$ is such that $\widehat{f}=\bdsf_{\ulQ}$. 
\end{proposition}

\begin{proof}
If $\ell$ is sufficiently large, then Proposition \ref{prop:51} gives 
\begin{align*}
\vep_{k_{\ulQ}}\ulQ(\eta_{\bdsf})
&=(\tau_{p^\ell\frkN}[U_p^{-\ell}\bdsf_{\ulQ}],\bdsf_{\ulQ})_{p^\ell\frkN}\\
&=\ulQ(\alp_{\bdsf})^{-\ell}(\tau_{p^\ell\frkN}\bdsf_{\ulQ},\bdsf_{\ulQ})_{p^\ell\frkN}
=\ulQ(\alp_{\bdsf})^{-\ell}\iot_p((\tau_{p^\ell\frkN}f,f)_{p^\ell\frkN}) 
\end{align*}
by (\ref{tag:56}) and (\ref{tag:57}). 
We can rewrite this identity as 
\[\vep_{k_{\ulQ}}\ulQ(\eta_{\bdsf})
=[\calk:\calk_0(p^\ell\frkN)]\ulQ(\alp_{\bdsf})^{-\ell}\iot_p((\tau_{p^\ell\frkN}f_{\bfv_{k_{\ulQ}}},f_{\bfv_{k_{\ulQ}}})_\calk)\dim\call_{k_{\ulQ}}, \]
by (\ref{tag:51}) and (\ref{tag:54}). 
\end{proof}


\section{Regularized diagonal cycles and theta elements}\label{sec:6}

\subsection{Twisted diagonal cycles for $\U(3)\times\U(2)$}\label{ssec:61} 
We retain the notation in \S \ref{ssec:26}. 
We will frequently add $'$ to the notation for various objects to indicate that they are attached to $H$. 
Fix normal rings $\bfI_1^{},\bfI_2^{},\bfI_3^{},\bfI_1',\bfI_2'$ finite flat over $\Lam=\calo\powerseries{\ZZ_p^\times}$.
Put 
\begin{align*}
\calr&=\bfI_1^{}\widehat{\otimes}_\calo\bfI_2^{}\widehat{\otimes}_\calo\bfI_3^{}, & 
\calr'&=\bfI_1'\widehat{\otimes}_\calo\bfI_2'. 
\end{align*}

For each prime factor $l$ of $NN'$ we fix an open compact subgroup $\scrj_l$ of $G(\QQ_l)$ which contains the subgroup $\imath_\frkl^{-1}(\iot'(\calk^{(2)}_0(N'\frkr_\frkl)))$.
Put 
\[\calj_\ell^{}=\imath_\frkp^{-1}(\cali^{(3)}_1(\frkp^\ell))\times\prod_{q\nmid pNN'}\calk_q^{}\times\prod_{l|NN'}\scrj_l. \]  
Define an open compact subgroup $\calk^{(2)}_{01}(p^\ell\frkN'')$ of $H(\widehat{\QQ})$ by 
\[\calk^{(2)}_{01}(p^\ell\frkN'')=\{(h_l)\in\calk_0^{(2)}(\frkN'')\;|\; \imath_\frkp(h_p^{})\in\cali_1^{(2)}(\frkp^\ell),\;h_q\in\calk_q''\text{ for }q\in\vSi_T^-\}, \]
where $\calk_q''=\calk_q^{}\cap\calk_q'$. 
We consider the projective systems of the finite sets 
\begin{align*}
X_\ell^{}&=G(\QQ)\bsl G(\widehat{\QQ})/\calj_\ell^{}, &
X_\ell'&=H(\QQ)\bsl H(\widehat{\QQ})/\calk^{(2)}_{01}(p^\ell\frkN''), & 
\bfX_\ell^{}&=X_\ell^{}\times X_\ell'.
\end{align*}
Consider the finitely generated $\calo$-module $\calo[\bfX_\ell]$ equipped with the operator $\bdscU_p^{}:=\calu_p^{}\otimes\calu_p'$ and the ordinary projector $\bfe_{\ord}^{}:=e_{\ord}^{}\otimes e_{\ord}'$. 
Given $x\in G(\widehat{\QQ})$ and $x'\in H(\widehat{\QQ})$, we write $[(x,x')]\in\bfX_\ell$ for the double coset represented by $(x,x')$. 
We define the embedding 
\begin{align*}
&\jmath:H(\widehat{\QQ})\hookrightarrow G(\widehat{\QQ}), & 
\jmath(x')&=\iot(x')\vsi^{(p)},  
\end{align*}
where $\vsi^{(p)}\in G(\widehat{\QQ})$ is defined in (\ref{tag:12}). 
Set 
\begin{align*}
\Ups_\ell&=\imath_\frkp^{-1}\left(\begin{bmatrix} p^{2\ell} & p^\ell & 0 \\ 0 & p^\ell & 1 \\ 0 & 0 & 1 \end{bmatrix}\right)\in G(\QQ_p), & 
\tau_\ell&=\imath_\frkp^{\prime-1}\left(\begin{bmatrix} 0 & 1 \\ -p^\ell & 0 \end{bmatrix}\right)\in H(\QQ_p). 
\end{align*}
For $z\in\QQ_p$ we put $\bfn(z)=\imath_\frkp^{\prime-1}\left(\begin{bmatrix} 1 & z \\ 0 & 1 \end{bmatrix}\right)\in H(\QQ_p)$ and  
\[\cali_1^-(\frkp^\ell)=\biggl\{\begin{bmatrix} a & b \\ c & d \end{bmatrix}\in\cali_1^{(2)}(\frkp^\ell)\;\biggl|\;b=0\biggl\}. \]

\begin{definition}\label{def:61}
Let $\Del_\ell\in\ZZ_p[\bfX_\ell]$ be the twisted diagonal cycle defined by 
\[\Del_\ell=\sum_{[x']\in X_\ell'}\sum_{z\in\ZZ_p/p^{2\ell}\ZZ_p}[(\jmath(x'\bfn(z))\Ups_\ell,x'\tau_\ell)]. \]
This definition makes sense in view of (\ref{tag:13}) and the fact that for each $\gam\in\cali_1^-(\frkp^\ell)$ there is $z\in\ZZ_p/p^{2\ell}\ZZ_p$ such that $\Ups_\ell^{-1}\iot(\imath_\frkp^{-1}(\gam\bfn(z)))\Ups_\ell^{}\in\imath_\frkp^{-1}(\cali_1^{(3)}(\frkp^\ell))$. 
\end{definition}


\subsection{Regularized diagonal cycles for $\U(3)\times\U(2)$}\label{ssec:62} 

The homomorphism 
\[\bfN_{\ell+1,\ell}:\ZZ_p[\bfX_{\ell+1}]\to\ZZ_p[\bfX_\ell]\] 
is induced by the projection $\bfX_{\ell+1}\to\bfX_\ell$. 

\begin{lemma}[Distribution property]\label{lem:61}
For $\ell\geq 1$ we have 
\[\bfN_{\ell+1,\ell}(\Del_{\ell+1})=\bdscU_p\Del_\ell.\]
\end{lemma}

\begin{proof} 
The proof is similar to \cite[Lemma 4.7]{MH}. 
We abbreviate $S_\ell=\ZZ_p/p^\ell\ZZ_p$. 
Since $\ell\geq 1$, 
\[\Sig_\ell:=\biggl\{\begin{bmatrix} 1+p^\ell u & 0 \\ 0 & 1+p^\ell v\end{bmatrix}\begin{bmatrix} 1 & 0 \\ p^\ell b & 1\end{bmatrix}\biggl|\;u,v,b\in S_1\biggl\}\]
is a complete set of representatives for $\calk'_{01}(p^\ell\frkN'')/\calk'_{01}(p^{\ell+1}\frkN'')$. 
For simplicity we assume that $H(\QQ)\cap x'\calk'_{01}(p^\ell\frkN'')x^{\prime-1}=\{1\}$ for every $x'\in H(\widehat{\QQ})$. 
Then $X_{\ell+1}'$ consists of elements represented by $x'k'$ with $x'\in X_\ell'$ and $k'\in\Sig_\ell^{}$. 
There are $s,t\in 1+p^\ell\ZZ_p$ and $w\in\ZZ_p$ such that 
\[\iot\left(\begin{bmatrix} 1 & 0 \\ p^\ell b & 1 \end{bmatrix}\bfn(z)\right)\Ups_{\ell+1}
\in\begin{bmatrix}
p^{2\ell+2} & p^{\ell+1}s & w \\
0 & p^{\ell+1} & t \\
0 & 0 & 1 \end{bmatrix}\calk'_{01}(p^\ell\frkN''). \]

We therefore find that $\bfN_{\ell+1,\ell}(\Del_{\ell+1})$ equals 
\begin{align*}
&\sum_{[x']\in X_{\ell+1}'}\sum_{
z\in S_{2\ell}}\sum_{z_1\in S_2}
[(\jmath(x'\bfn(z+p^{2\ell}z_1))\Ups_{\ell+1},x'\tau_{\ell+1})]\\
=&\sum_{[x']\in X_\ell'}\sum_{k'\in\Sig_\ell}\sum_{
z\in S_{2\ell}}\sum_{z_1\in S_2}
[(\jmath(x'k'\bfn(z+p^{2\ell}z_1))\Ups_{\ell+1},x'k'\tau_{\ell+1})]\\
=&\sum_{[x']\in X_\ell'}\sum_{
z\in S_{2\ell}}\sum_{z_1\in S_2}
\sum_{u,v,b\in S_1}\\
&\left[\left(\jmath(x')
\begin{bmatrix}
p^{2\ell+2} & p^{\ell+1}(1+p^\ell u) & z+p^{2\ell}z_1 \\
0 & p^{\ell+1} & 1+p^\ell v \\
0 & 0 & 1 \end{bmatrix}, 
x'\tau_\ell 
\begin{bmatrix} p & b \\ 0 & 1 \end{bmatrix}\right)\right]\\
=&\sum_{[x']\in X_\ell'}\sum_{
z\in S_{2\ell}}
\sum_{z_1\in S_2}\sum_{u,v\in S_1}
(\ono\otimes\calu_p')\left[\left(\jmath(x'\bfn(z))\Ups_\ell
\begin{bmatrix}
p^2 & pu & z_1 \\
0 & p & v \\
0 & 0 & 1 \end{bmatrix}, 
x'\tau_\ell\right)\right]\\
=&\bdscU_p\Del_\ell 
\end{align*}
by Definition \ref{def:57}.  
\end{proof}

\begin{definition}\label{def:62}
Define the regularized diagonal cycle by $\Del_\ell^\dagger=\bdscU_p^{-\ell}\bfe_{\ord}^{}\Del_\ell^{}$. 
\end{definition}

Since $\bfN_{\ell+1,\ell}^{}(\Del_{\ell+1}^\dagger)=\Del_\ell^\dagger$ by Lemma \ref{lem:61}, we can define 
\[\Del_\infty^\dagger=\lim_{\longleftarrow \ell}\Del_\ell^\dagger\in\lim_{\longleftarrow \ell}\ZZ_p[\bfX_\ell]. \]


\subsection{Theta elements}\label{ssec:63}

Let $\chi'$ be a Dirichlet character of $(\frkr/\frkN')^\times$ of conductor $M$. 
Recall that a character of $\calk_0'(\frkN')$ is associated to $\chi'$ by 
\[\begin{bmatrix} a & b \\ c & d \end{bmatrix}\mapsto\chi'(d) \]
(cf. Definition \ref{def:56}). 
Take a divisor $\frkM$ of $\frkN'$ such that $\frkr/\frkM\simeq\ZZ/M\ZZ$. 

Let $\bdsf\in e_{\ord}\bfS^G(\frkN,\chi,\calr)$ and $\bdsg\in e_{\ord}'\bfS^H(\frkN',\chi',\calr')$ be Hida families. 
We define an $\calr$-adic modular form $\Tht^{\chi'}_{\frkN'}\bdsf$ on $G$ by 
\[\Tht^{\chi'}_{\frkN'}\bdsf(x)=\sum_{i,j\in(\ZZ/M\ZZ)^\times}\sum_{y\in\ZZ/M^2\ZZ}
\chi'(ij)^{-1}\bdsf\left(x\cdot\begin{bmatrix} \xi_{\frkN'}'w_2 & \\ & 1 \end{bmatrix}\begin{bmatrix} 1 & \frac{i}{M} & \frac{y}{M^2} \\ 0 & 1 & \frac{j}{M} \\ 0 & 0 & 1 \end{bmatrix}\right), \]
where we define the action of $g=(g_l)\in\prod_{l|M}\GL_3(\QQ_l)$ on $x\in G(\widehat{\QQ})$ by 
\[x\cdot g=x\prod_{l|M}\imath_\frkl^{-1}(g_l). \]
Proposition \ref{prop:43} below shows that for $u\in\calk_0'(\frkM^2)$
\[r(\iot(u))\Tht^{\chi'}_{\frkN'}\bdsf=\chi(u)\Tht^{\chi'}_{\frkN'}\bdsf. \] 

We hereafter assume Hypothesis ($H_3$) in \S \ref{ssec:213}. 
We construct the regularized diagonal cycle $\Del_\infty^\dagger$ by letting $\frkN''=\frkM^2$. 
Put $\bfG=G\times H$. 
We define an $\calr\widehat{\otimes}\calr'$-adic modular form on $\bfG$ by 
\[\bdsF=\Tht^{\chi'}_{\frkN'}\bdsf\boxtimes\bdsg:\bfG(\QQ)\bsl\bfG(\widehat{\QQ})\to\calr\widehat{\otimes}\calr'\] 
by 
\[\bdsF(x,x')=\Tht^{\chi'}_{\frkN'}\bdsf(x)\bdsg(x'). \]
Then $\bdsF$ naturally induces a $\Lam_3\boxtimes\Lam_2$-linear map 
\[\bdsF_*:\displaystyle\lim_{\longleftarrow \ell} \calo[\bfX_\ell]\to\calr\widehat{\otimes}\calr'. \]  
The theta element $\Tht_{\bdsF}$ attached to the product $\bdsF$ is then defined by the evaluation of $\bdsF_*$ at the regularized diagonal cycle, namely, 
\[\Tht_{\bdsF}=\bdsF_*(\Del_\infty^\dagger)\in\calr\widehat{\otimes}\calr'. \]


\subsection{Period integrals}\label{ssec:64}

The set $\frkX^{\rm crit}$ consists of pairs $(\ulQ,\ulQ')$ with 
\begin{align*}
\ulQ&=(Q_1,Q_2,Q_3)\in\frkX_\calr^{++}, &
\ulQ'&=(Q_1',Q_2')\in\frkX_{\calr'}^{++}
\end{align*} 
such that 
\[k_{Q_1}\leq -k_{Q_2'}\leq k_{Q_2}\leq -k_{Q_1'}\leq k_{Q_3}. \]
Fix $\calq=(\ulQ,\ulQ')\in\frkX^{\rm crit}$. 
We denote the automorphic representation of $G(\AA)$ associated to $\ulQ(\bdsf)$ by $\bdpi_{\ulQ}$ and the automorphic representation of $H(\AA)$ attached to $\ulQ'(\bdsg)$ by $\bdsig_{\ulQ'}$. 
We here abbreviate  
\begin{align*}
\pi&=\bdpi_{\ulQ}, &
\sig&=\bdsig_{\ulQ'}, & 
\vPi&=\pi\otimes\sig, & 
k_i&=k_{Q_i}, & 
k_j'&=k_{Q_j'}   
\end{align*}
for $i=1,2,3$ and $j=1,2$. 
Put 
\begin{align*}a&=k_3^{}-k_2^{}, & b&=k_2^{}-k_1^{}, & n&=k_3^{}+k_1', & l&=-k_2'-k_1^{}. 
\end{align*}
We use the notation in \S \ref{ssec:b6} to note that 
\begin{align*}
\pi_\infty&=\rho_{k^\vee_{\ulQ}}\simeq\frkH_{b,a}\otimes(\det)^{-k_2}, & 
\sig_\infty&=\rho_{k^\vee_{\ulQ'}}\simeq\vrh_{(-k_1,-k_2)}. 
\end{align*}
Take highest weight vectors 
\begin{align}
\bfv_{k_{\ulQ}}&=x_1^by_3^a\in\pi_\infty, & 
\bfv_{k_{\ulQ'}}&=X_1^{a-n+b-l}\in\sig_\infty. \label{tag:61}
\end{align}
Then the lowest weight dual vectors are
\begin{align}
\bfu_{k_{\ulQ}}&=y_1^{\prime b}x_3^{\prime a}\in\pi_\infty^\vee, & 
\bfu_{k_{\ulQ'}}&=(-Y_1')^{a-n+b-l}\in\sig_\infty^\vee \label{tag:62}
\end{align}
(see Remark \ref{rem:52}). 

Let $f\in\cala_{\ulk}^G(p^\ell\frkN,\chi)$ and $g\in\cala_{\ulk'}^H(p^\ell\frkN',\chi')$ be the ad\`{e}lic lifts of 
\begin{align*}
\widehat{f}&=\bdsf_{\ulQ}\in\cals^G_{k_{\ulQ}}(p^\ell\frkN,\chi\eps_{\ulQ},\calr(\ulQ)), \\ 
\widehat{g}&=\bdsg_{\ulQ'}\in\cals^H_{k_{\underline{Q'}}}(p^\ell\frkN',\chi'\eps_{\underline{Q'}},\calr'(\underline{Q'})) 
\end{align*}
(see (\ref{tag:55})). 
Recall that $f$ and $g$ are eigenforms of the operators $U_p^{}$ and $U_p'$ with eigenvalues $\alp_{\pi_p}$ and $\alp_{\sig_p}$ by \cite[Theorem 5.3]{Hid04} and Proposition \ref{prop:31} below. 
Define $\Tht^{\chi'}_{\frkN'}f$ by replacing $\bdsf$ with $f$ in the definition of $\Tht^{\chi'}_{\frkN'}\bdsf$. 
Clearly, 
\[\ulQ(\Tht^{\chi'}_{\frkN'}\bdsf)=(\widehat{\Tht^{\chi'}_{\frkN'}f})_{\bfv_{k_{\ulQ}}}. \]

Define the vector-valued modular form $F:\bfG(\AA)\to\vPi_\infty^\vee$ by 
\begin{align*}
F(x,x')&=\Tht^{\chi'}_{\frkN'}f(x)\otimes g(x') &
(x&\in G(\AA),\; x'\in H(\AA)). 
\end{align*}
Define a scalar-valued modular form $\vPh_F:G(\AA)\times H(\AA)\to\CC$ by 
\[\vPh_F(x,x')=\ell_{\vPi_\infty}^{}(\bfW_{\vPi_\infty}^H\otimes F(x,x')), \]
where $\bfW_{\vPi_\infty}^H\in\vPi_\infty^{}$ is an $H(\RR)$-invariant vector defined in \S \ref{ssec:b6}. 
Proposition \ref{prop:b1}(\ref{prop:b14}) shows that 
\[\vPh_F(x\iot(u_\infty),x'u_\infty)=\vPh_F(x,x') \]
for $u_\infty\in H(\RR)$. 
Define $\bft_\ell\in G(\QQ_p)$ by 
\[\bft_\ell=\imath_\frkp^{-1}\left(\begin{bmatrix} 0 & 0 & -p^{-\ell} \\ 0 & p^\ell & 1 \\ p^{2\ell} & p^\ell & 0 \end{bmatrix}\right). \]
We consider the following period integral
\[\scrp_{\calk'}(\pi_p(\bft_\ell)\vPh_F)=\int_{H(F)\bsl H(\AA)}\vPh_F(h\bft_\ell,h)\, \d_{\calk'} h. \]

\begin{proposition}\label{prop:61}
Let  $\calq\in\frkX^{\rm crit}$. 
Then 
\[\calq(\Tht_{\bdsF})=\frac{(-1)^{k_2-k_1}}{\zet_p(1)\zet_p(2)}[\calk':\calk_0'(\frkM^2)]\frac{\scrp_{\calk'}(\pi(\bft_\ell\vsi^{(p)})\vPh_F)}{(p^{-5}\alp_{\pi_p}\alp_{\sig_p})^\ell}\prod_{q\in\vSi_T^-\setminus\vSi_E^r}(q+1) \]
for sufficiently large $\ell$. 
\end{proposition}

\begin{proof}
The proof is similar to that of \cite[Proposition 4.9]{MH} (cf. \cite[Lemma 4.4]{CH}). 
Recall that $\ulQ(\bdsf)=\widehat{f}_{\bfv_{k_{\ulQ}}}$ and $\ulQ'(\bdsg)=\widehat{g}_{\bfv_{k_{\ulQ'}}}$ are eigenforms of the operators $\calu_p^{}$ and $\calu_p'$ with unit eigenvalues $\alp_f$ and $\alp_g$. 
Thus 
\[\calq(\bdsF)=\ulQ(\Tht^{\chi'}_{\frkN'}\bdsf)\boxtimes\ulQ'(\bdsg)=(\widehat{\Tht^{\chi'}_{\frkN'}f})_{\bfv_{k_{\ulQ}}}\boxtimes g_{\bfv_{k_{\ulQ'}}}\] 
is an eigenform of $\bdscU_p$ with unit eigenvalue 
\[\alp_F=\alp_f\alp_g=p^{-2k_1^{}-k_2^{}-k_1'}\alp_{\pi_p}\alp_{\sig_p} \]
by Remark \ref{rem:53}.  
Let $\widehat{F}=\widehat{\Tht^{\chi'}_{\frkN'}f}\otimes\widehat{g}$ be a $p$-adic modular form on $\bfG$. 
Then  
\[\calq(\bdsF)(x,x')=\ell_{\vPi_\infty}^{}(x_1^by_3^aX_1^{a-n+b-l}\otimes\widehat{F}(x,x')) \]
for $x\in G(\widehat{\QQ})$ and $x'\in H(\widehat{\QQ})$

By definition we have 
\begin{align*}
\calq(\Tht_{\bdsF})
&=\calq(\bdsF)(\bdscU_p^{-\ell}\bfe_{\ord}^{}\Del_\ell^{})\\
&=\alp_F^{-\ell}\sum_{[x']\in X_\ell'}\sum_{z\in\ZZ_p/p^{2\ell}\ZZ_p}\calq(\bdsF)(\jmath(x'\bfn(z))\Ups_\ell,x'\tau_\ell)
\end{align*}
for sufficiently large $\ell$, we see that  
\[\alp_F^\ell\calq(\Tht_{\bdsF})=\sum_{[x']\in X_\ell'}\sum_{z\in\ZZ_p/p^{2\ell}\ZZ_p}\ell_{\vPi_\infty}^{}(x_1^by_3^aX_1^{a-n+b-l}\otimes\widehat{F}(\jmath(x'\bfn(z))\Ups_\ell,x'\tau_\ell)). \]

Recall that $\vPi_\infty^\vee=\rho_{k_{\ulQ}}\otimes\rho_{k_{\ulQ'}}$. 
Since 
\[F(x,x')=\iot_p^{-1}(\vPi_\infty^\vee(\iot_p^T(x_p^{}),\iot_p^{T'}(x_p'))\widehat{F}(x,x')) \]
by (\ref{tag:55}), we see that for $x\in G(\widehat{\QQ})$ and $x'\in H(\widehat{\QQ})$ 
\begin{align*}
\vPh_F(x,x')
&=\ell_{\vPi_\infty}(\bfW_{\vPi_\infty}^H\otimes\iot_p^{-1}(\vPi_\infty^\vee(\iot_p^T(x_p^{}),\iot_p^{T'}(x_p'))\widehat{F}(x,x')))
\end{align*}
Proposition \ref{prop:b1}(\ref{prop:b14}) shows that for $x'\in H(\QQ_p)$, $g\in G(\QQ_p)$ and $h\in H(\QQ_p)$ 
\begin{align*}
&\ell_{\vPi_\infty}(\bfW_{\vPi_\infty}^H\otimes F(\jmath(x')g,x'h))\\
=&\ell_{\vPi_\infty}(\bfW_{\vPi_\infty}^H\otimes\vPi_\infty^\vee(\iot_p^T(\iot(x_p')),\iot_p^{T'}(x_p'))^{-1}F(\jmath(x')g,x'h))\\
=&\iot_p^{-1}(\ell_{\vPi_\infty}(\vPi_\infty^{}(\iot_p^T(g),\iot_p^{T'}(h))^{-1}\bfW_{\vPi_\infty}^H\otimes\widehat{F}(\jmath(x')g,x'h))). 
\end{align*}
Let $g=\iot(\bfn(z))\Ups_\ell$ and $h=\tau_\ell$. 
Observe that 
\begin{align*}
\imath_\frkp(g^{-1}\iot(h))&=\frac{1}{p^{2\ell}}\begin{bmatrix} 
1 & -1 & 1-z \\
0 & p^\ell & -p^\ell \\
0 & 0 & p^{2\ell} \end{bmatrix}
\begin{bmatrix} 0 & 0 & 1 \\ 0 & 1 & 0 \\ -p^\ell & 0 & 0 \end{bmatrix}
=\frac{1}{p^{2\ell}}\begin{bmatrix} 
(z-1)p^\ell & -1 & 1 \\ p^{2\ell} & p^\ell & 0 \\ -p^{3\ell} & 0 & 0 \end{bmatrix}, \\
\imath_\frkp(\iot(h)^{-1}g)&=\begin{bmatrix} 0 & 0 & -p^{-\ell} \\ 0 & 1 & 0 \\ 1 & 0 & 0 \end{bmatrix}\begin{bmatrix} p^{2\ell} & p^\ell & z \\ 0 & p^\ell & 1 \\ 0 & 0 & 1 \end{bmatrix}
=\begin{bmatrix}
0 & 0 & -p^{-\ell} \\ 0 & p^\ell & 1 \\ p^{2\ell} & p^\ell & z 
\end{bmatrix}.  
\end{align*}
Recall that 
\[\bfW_\vPi^H\equiv\det\begin{bmatrix} x_1 & x_3 \\ X_2 & Y_2 \end{bmatrix}^{b-l}\det\begin{bmatrix} X_2 & Y_2 \\ y_3 & -y_1 \end{bmatrix}^{a-n}x_2^ly_2^n \pmod{\calt_{b,a}(\QQ)}. \]
Since 
\begin{align*}
&(x_1,x_2,x_3)\imath_\frkp(g^{-1}\iot(h))=p^{-2\ell}(x_1(z-1)p^\ell+x_2p^{2\ell}-x_3p^{3\ell},x_2p^\ell-x_1,x_1), \\
&(y_1,y_2,y_3)\imath_\frkp(\trs(\iot(h)^{-1}g))=(-y_3p^{-\ell},y_3+y_2p^\ell,y_3z+y_2p^\ell+y_1p^{2\ell}),  
\end{align*}
we see that 
\begin{align*}
\vPi_\infty(\iot(\bfn(z))\Ups_\ell,\tau_\ell)^{-1}\bfW_{\vPi_\infty}^H
=&\frac{1}{p^{2\ell(b-l)}}\det\begin{bmatrix} x_1(z-1)p^\ell+x_2p^{2\ell}-x_3p^{3\ell} & x_1 \\ X_2 & Y_2 \end{bmatrix}^{b-l}\\
&\times \det\begin{bmatrix} X_2 & Y_2 \\ y_3z+y_2p^\ell+y_1p^{2\ell} & y_3p^{-\ell} \end{bmatrix}^{a-n}\\
&\times p^{-2l\ell}(x_2p^\ell-x_1)^l(y_3+y_2p^\ell)^n\det(g^{-1}\iot(h))^{-k_2},   
\end{align*} 
from which we find the following congruence relation 
\begin{align*}
&\bigl(p^{2k_1^{}+k_2^{}+k_1'}\bigl)^{-\ell}\vPi_\infty(\jmath(\bfn(z))\Ups_\ell,\tau_\ell)^{-1}\bfW_{\vPi_\infty}^H\\
\equiv&(-x_1)^by_3^aX_2^{b-l+a-n}\pmod{p^s} 
\end{align*}
if $\ell$ is sufficiently larger than $s$. 
We conclude that 
\begin{align*}
&\bigl(p^{2k_1^{}+k_2^{}+k_1'}\bigl)^{-\ell}\iot_p(\ell_{\vPi_\infty}(\bfW_{\vPi_\infty}^H\otimes F(x'\jmath(\bfn(z))\Ups_\ell,x'\tau_\ell)))\\
=&\ell_{\vPi_\infty}((-x_1)^by_3^aX_2^{b-l+a-n}\otimes\widehat{F}(x'\jmath(\bfn(z))\Ups_\ell,x'\tau_\ell)). 
\end{align*}
 
Substituting this expression, we see that
\begin{align*}
&\bigl(p^{2k_1^{}+k_2^{}+k_1'}\alp_F\bigl)^\ell\calq(\Tht_{\bdsF}) \\
=&(-1)^b\sum_{[x']\in X_\ell'}\sum_{z\in\ZZ_p/p^{2\ell}\ZZ_p}\vPh_F(\jmath(x'\bfn(z))\Ups_\ell,x'\tau_\ell)\\
=&(-1)^b[\calk':\calk_{01}^{(2)}(p^\ell\frkM^2)]\sum_{z\in\ZZ_p/p^{2\ell}\ZZ_p}\int_{H(\QQ)\bsl H(\AA)}\vPh_F(\jmath(h\bfn(z))\Ups_\ell,h\tau_\ell)\,\d_{\calk'} h.   
\end{align*} 
Observe that
\begin{align*}
&\int_{H(\QQ)\bsl H(\AA)}\vPh_F(\iot(h\bfn(z))\Ups_\ell,h\tau_\ell)\,\d_{\calk'} h \\
=&\int_{H(\QQ)\bsl H(\AA)}\vPh_F(\iot(h)\Ups_\ell,h\bfn(-z)\tau_\ell)\,\d_{\calk'} h
=\scrp_{\calk'}(\vPi_p(\Ups_\ell,\tau_\ell)\vPh_F). 
\end{align*}
Since $[\calk_q':\calk_q'']=q+1$ for $q\in\vSi_T^-\setminus\vSi_E^r$ by Lemma 3.14 of \cite{Shimura}, we have 
\[[\calk':\calk_{01}^{(2)}(p^\ell\frkM)]=p^{3\ell}(1-p^{-2})(1-p^{-1})[\calk':\calk_0^{(2)}(\frkM^2)]\prod_{q\in\vSi_T^-\setminus\vSi_E^r}(q+1). \] 
Since $\iot(\tau_\ell^{-1})\Ups_\ell^{}=\bft_\ell$, we obtain the stated formula. 
\end{proof}


\section{The central value formulae}\label{sec:2}


\subsection{The Ichino-Ikeda formula}\label{ssec:24}

Let $G$ be the unitary group of the Hermitian form $(\;,\;)_T$ on $W=E^n$.  
Let $H$ be the unitary group of a subspace $W'$ of $W$ of dimension $n-1$ on which $(\;,\;)_T$ is non-degenerate.  
We view $H$ as a subgroup of $G$. 
Let $\pi\simeq\otimes_v'\pi_v^{}$ be an irreducible cuspidal automorphic representation of $G(\AA)$ and $\sig\simeq\otimes_v'\sig_v^{}$ an irreducible cuspidal automorphic representation of $H(\AA)$. 
Set 
\begin{align*}
\bfG&=G\times H, & 
\vPi&=\pi\otimes\sig, &
\vPi_v&=\pi_v\otimes\sig_v. 
\end{align*} 
 
We define the product $L$-series associated to $\pi$ and $\sig$ by 
\[L(s,\pi\times\sig)=L^\GL(s,\BC(\pi)\times\BC(\sig)), \]
where $\BC(\pi)$ (resp. $\BC(\sig)$) is the functorial lift of $\pi$ (resp. $\sig$) to an automorphic representation of $\GL_n(\EE)$ (resp. $\GL_{n-1}(\EE)$). 
The right hand side is the $L$-factor defined by Jacquet, Piatetski-Shapiro and Shalika in \cite{JPSS2}. 
Let $L(s,\pi,\mathrm{Ad})$ denote the adjoint $L$-series for $\pi$. 
Assume that both $\pi$ and $\sig$ are tempered. 
Put 
\begin{align*}
\scrl(\pi\times\sig)&=\frac{L\left(\frac{1}{2},\pi\times\sig\right)}{L(1,\pi,\mathrm{Ad})L(1,\sig,\mathrm{Ad})}\prod_{i=1}^nL(i,\eps_{E/\QQ}^i), \\
\scrl(\pi_v\times\sig_v)&=\frac{L\left(\frac{1}{2},\pi_v\times\sig_v\right)}{L(1,\pi_v^{},\mathrm{Ad})L(1,\sig_v,\mathrm{Ad})}\prod_{i=1}^nL(i,\eps_{E_v/\QQ_v}^i). 
\end{align*}

\begin{remark}\label{rem:21}
If $n$ is even, then $L(s,\pi,\mathrm{Ad})=L(s,\BC(\pi),\As)$ is the Asai $L$-series for $\BC(\pi)$ while if $n$ is odd, then $L(s,\pi,\mathrm{Ad})=L(s,\BC(\pi),\As^-)$ is the twisted Asai $L$-series. 
\end{remark}

We define the Petersson pairings by 
\[(\vPh,\vPh')=\int_{\bfG(\QQ)\bsl\bfG(\AA)}\vPh(g,h)\vPh'(g,h)\,\d^\tau\! g\,\d^\tau\! h \]
for cusp forms $\vPh$ and $\vPh'$ on $\bfG$, where $\d^\tau\! g$ and $\d^\tau\! h$ are the Tamagawa measures on $G(\AA)$ and $H(\AA)$. 
Given a cusp form $\vPh$ on $\bfG$, we consider the integral 
\[\scrp(\vPh)=\int_{H(\QQ)\bsl H(\AA)}\vPh((h,h))\,\d^\tau\! h. \]  
Fix a local perfect pairing 
\[\La\!\La\;,\;\Ra\!\Ra_v^{}:\vPi_v^{}\otimes\vPi_v^\vee\to\CC.  \]
If $\vPi_v$ is tempered, then the integral
\[I(\bfW_1,\bfW_2)=\int_{H(\QQ_v)}\La\!\La\vPi_v((h_v,h_v))\bfW_1,\bfW_2\Ra\!\Ra_v^{}\,\d h_v  \]
is convergent for $\bfW_1\in\vPi_v^{}$ and $\bfW_2\in\vPi_v^\vee$. 
The local Haar measure $\d h_v$ is defined so that a maximal compact subgroup $\calk_v'$ of $H(\QQ_v)$, which we will specify later for $n=3$, has volume $1$. Let $C_H$ be the ratio between the Tamagawa measure and the product measure of local measures. Namely $C_H$ is defined so that $\d^\tau\! h=C_H\prod_v\d h_v$.

Ichino and Ikeda \cite{II} have refined the global Gross-Prasad conjecture and predicted an explicit relation between the central value and the period for orthogonal groups. The analogue of the Ichino-Ikeda conjecture for unitary groups was formulated in \cite{NHarris} and proved by \cite{BLZZ} in the stable case and by \cite{BCZ} in the endoscopic case.

\begin{theorem}[\cite{Z,BLZZ,BCZ}]\label{thm:21}
Let $\vPi$ be an irreducible tempered cuspial automorphic representation of $\bfG(\AA)$. 
If $\vPh=\otimes_v'\bfW_v^{}\in\vPi$ and $\vPh'=\otimes_v'\bfW_v'\in\vPi^\vee$ are factorizable, then 
\[\frac{\scrp(\vPh)\scrp(\vPh')}{(\vPh,\vPh')}
=C_H\frac{\scrl(\pi\times\sig)}{2^{\vka_\pi+\vka_\sig}}\prod_v\frac{I(\bfW_v^{},\bfW_v')}{\scrl(\pi_v\times\sig_v)\La\!\La \bfW_v^{},\bfW_v'\Ra\!\Ra_v^{}}, \]
where $2^{\vka_\pi}$ (resp. $2^{\vka_\sig}$) is the order of the component group associated to the $L$-parameter of $\pi$ (resp. $\sig$). 
\end{theorem}
Roughly speaking, this theorem tells us that the product of global period integrals is a product of the Rakin-Selberg central value $\scrl(\pi\times\sig)$ and the local zeta integrals $I(\bfW_v^{},\bfW_v')$. Therefore, by Proposition \ref{prop:61} the task of obtaining the interpolation formula of $\Tht_{\bdsF}$ boils down to (i) choices of test vectors $\bfW_v$ and $\bfW_v'$ and (ii) the explicit evaluation of $I(\bfW_v^{},\bfW_v')$. The purpose of this section is to carry out the step (i) and give the explicit formula of the relevant local zeta integrals $I(\bfW_v^{},\bfW_v')$. The details of the step (ii) are left to 
\S\ref{sec:3}, \S\ref{sec:4} and Appendices. 


\subsection{Shimura's mass formula}\label{ssec:25}

We now assume that $T$ is positive definite. 
Fix a suitable maximal compact subgroup $\calk'=\prod_q\calk_q'$ of $H(\widehat{\QQ})$. 
The space $\scra(G)$ of automorphic forms on $G$ consists of left $G(\QQ)$-invariant, right $G(\RR)\calk$-finite functions on $G(\AA)$. 
We normalize the Haar measure $\d h_\infty$ by $\int_{H(\RR)}\,\d h_\infty=1$. 
Actually, it is more suitable for our later application to use the Haar measure $\d_{\calk'}h=\prod_v\d h_v$, which gives the maximal compact subgroup $H(\RR)\calk'$ volume 1. 
We similarly define the Haar measure $\d_\calk g$ and choose the constants $C_H$ and $C_G$ so that 
\begin{align*}
\d^\tau\! h&=C_H\d_{\calk'} h, & 
\d^\tau\! g&=C_G\d_\calk g. 
\end{align*}

We normalize the period integrals by 
\begin{align*}
(\vPh,\vPh')_{\calk\times\calk'}&=\int_{\bfG(\QQ)\bsl\bfG(\AA)}\vPh(g,h)\vPh'(g,h)\,\d_\calk g\,\d_{\calk'} h, \\ 
\scrp_{\calk'}(\vPh)&=\int_{H(\QQ)\bsl H(\AA)}\vPh((h,h))\,\d_{\calk'} h 
\end{align*}
for $\vPh,\vPh'\in\scra(\bfG)$. 
We rewrite the formula in Theorem \ref{thm:21} as 
\[\frac{\scrp_{\calk'}(\vPh)\scrp_{\calk'}(\vPh')}{(\vPh,\vPh')_{\calk\times\calk'}}
=C_G\frac{\scrl(\pi\times\sig)}{2^{\vka_\pi+\vka_\sig}}\prod_v\frac{I(\bfW_v^{},\bfW_v')}{\scrl(\pi_v\times\sig_v)\La\!\La \bfW_v^{},\bfW_v'\Ra\!\Ra_v^{}}. \]
Shimura \cite{Shimura2} has explicitly computed the mass  
\[\frac{2}{C_G}
=\sum_{x\in G(\QQ)\bsl G(\widehat{\QQ})/\calk}\frac{1}{\sharp(G(\QQ)\cap x\calk x^{-1})}
=2^{1-n}\prod_{i=1}^nL_\bff(1-i,\eps_{E/\QQ}^i)\prod_v\lam_v \]
(see Propositions 4.4 and 4.5 of \cite{GHY}), where $L_\bff(s, \eps_{E/\QQ}^i)$ is the {\it non-complete} Dirichlet $L$-series associated to the Dirichlet character $\eps_{E/\QQ}^i$. 
Observe that 
\beq
\gam_G:=\pi^{-\frac{n(n+1)}{2}}C_G\prod_{i=1}^nL_\bff(i,\eps_{\CC/\RR}^i)\in\sqrt{D_E}^a\QQ^\times, \label{tag:20}
\eeq
where $D_E$ is the absolute value of the discriminant of $E$, and $a=0$ if $n\equiv 0,3\pmod 4$ and $a=1$ if $n\equiv 1,2\pmod 4$. 

We hereafter suppose that 
\beq
\Hom_{H(\RR)}(\vPi_\infty^{},\CC)\neq\{0\}. \tag{$H_2$}
\eeq
Namely, there are $H(\RR)$-invariant vectors $\bfW_\infty^H\in\vPi_\infty^{}$ and $\bfW_\infty^{H\prime}\in\vPi_\infty^\vee$. 
Then 
\[I(\bfW_\infty^H,\bfW_\infty^{H'})=\La\!\La \bfW_\infty^H,\bfW_\infty^{H\prime}\Ra\!\Ra_\infty^{}. \] 
It follows from (\ref{tag:b0}) that 
\[c_\infty:=\frac{L(1,\pi_\infty,\mathrm{Ad})L(1,\sig_\infty,\mathrm{Ad})}{L\left(\frac{1}{2},\pi_\infty\times\sig_\infty\right)}\pi^{\frac{n(n+1)}{2}}\in\QQ^\times. \]

\begin{corollary}\label{cor:21}
Assume that $T$ is positive definite. 
Let $\vPi$ be an irreducible tempered cuspial automorphic representation of $\bfG(\AA)$ whose archimedean part $\vPi_\infty$ satisfies ($H_2$). 
Let $\vPh=\otimes_v'\bfW_v^{}\in\vPi$ and $\vPh'=\otimes_v'\bfW_v'\in\vPi^\vee$ be factorizable. 
If $\bfW_\infty^{}=\bfW_\infty^H$ and $\bfW_\infty'=\bfW_\infty^{H\prime}$, then 
\begin{align*}
\frac{\scrp_{\calk'}(\vPh)\scrp_{\calk'}(\vPh')}{(\vPh,\vPh')_{\calk\times\calk'}}
=&\frac{\gam_G}{2^{\vka_\pi+\vka_\sig}}c_\infty\frac{L\left(\frac{1}{2},\pi\times\sig\right)}{L(1,\pi,\mathrm{Ad})L(1,\sig,\mathrm{Ad})}\prod_l\frac{I(\bfW_l^{},\bfW_l')}{\scrl(\pi_l\times\sig_l)\La\!\La \bfW_l^{},\bfW_l'\Ra\!\Ra_l^{}}. 
\end{align*} 
\end{corollary}


\subsection{Rationality of central values}\label{ssec:211}

We write $\Aut(\CC)$ for the group of automorphisms of $\CC$. 
Let $\tau\in\Aut(\CC)$. 
For a complex representation $\vPi$ of a group $\calg$ on the space $V_\vPi$ of $\vPi$, let $^\tau\!\vPi$ be the representation of $\calg$ defined by $^\tau\!\vPi(g)=t^{}\circ\vPi(g)\circ t^{-1}$, where $t:V_\vPi\to V_\vPi$ is a $\tau$-linear isomorphism.  
Note that the isomorphism class of $^\tau\!\vPi$ is independent of the choice of $t$. 
Given $\vph\in\scra(G)$, we define $^\tau\!\vph\in\scra(G)$ by $^\tau\!\vph(g)=\tau(\vph(g))$ for $g\in G(\AA)$. 
The representation $^\tau\!\pi$ is realized in the space $\{^\tau\!\vph\;|\;\vph\in\pi\}$ by the multiplicity one for unitary groups. 
Similarly, $^\tau\!\sig$ is an automorphic representation of $H(\AA)$. 

\begin{proposition}\label{prop:22}
Let $\pi$ and $\sig$ be irreducible tempered automorphic representations of definite unitary groups $G(\AA)$ and $H(\AA)$.
If $\Hom_{H(\QQ_v)}(\pi_v\otimes\sig_v,\CC)\neq\{0\}$ for all $v$, then for every $\tau\in\Aut(\CC)$ 
\[\tau\left(\frac{L\bigl(\frac{1}{2},\pi\times\sig\bigl)}{\sqrt{D_E}^aL(1,\pi,\Ad)L(1,\sig,\Ad)}\right)=\frac{L\bigl(\frac{1}{2},{}^\tau\!\pi\times\!\,^\tau\!\sig\bigl)}{\sqrt{D_E}^aL(1,{}^\tau\!\pi,\Ad)L(1,{}^\tau\!\sig,\Ad)}, \]
where $a=0$ if $n\equiv 0,3\pmod 4$ and $a=1$ if $n\equiv 1,2\pmod 4$. 
\end{proposition}

\begin{remark}\label{rem:20}
 The relation (\ref{tag:11}) is equivalent to the existence of a non-zero $H(\RR)$-invariant functional on $\pi_\infty\otimes\sig_\infty$. 
 \end{remark}

\begin{proof}
It is evident that   
\begin{align*}
\tau((\vPh,\vPh')_{\calk\times\calk'})&=(^\tau\!\vPh,{}^\tau\!\vPh')_{\calk\times\calk'}, &
\tau(\scrp_{\calk'}(\vPh))&=\scrp_{\calk'}(^\tau\!\vPh) 
\end{align*}
(cf. (\ref{tag:51})). 
We take $\ulk$ and $\ulk'$ so that $\pi_\infty\simeq\call_{\ulk^\vee}(\CC)$ and $\sig_\infty\simeq\call_{\ulk^{\prime\vee}}(\CC)$. 
It is easy to see that $\pi$ is spanned by $f_{\bfv}$ with $\bfv\in\call_{\ulk^\vee}(\QQ)$ and $\overline{\QQ}$-rational $\call_{\ulk}(\CC)$-valued modular forms $f$ defined in Definition \ref{def:52} (see \cite{Gross}). 
One may therefore assume that modular forms $\vPh$ and $\vPh'$ are $\overline{\QQ}$-rational, namely, they have values in $\overline{\QQ}$.  
Then $(\vPh,\vPh')_{\calk\times\calk'}$, $\scrp_{\calk'}(\vPh)$ and $\scrp_{\calk'}(\vPh')$ are algebraic numbers. 

Given a matrix coefficient $\phi_l^{}$ of $\vPi_l^{}$, we define the matrix coefficient $^\tau\!\phi_l$ of $^\tau\!\vPi_l$ by $^\tau\!\phi_l(g_l)=\tau(\phi_l(g_l))$ for $g_l\in G(\QQ_l)$. 
Put  
\[\phi_l^{}(g)=\La\!\La\vPi_l^{}(g_l^{})\bfW_l^{},\bfW_l'\Ra\!\Ra_l^{}.  \]
The assumption allows us to choose $\vPh=\otimes_v^{}\bfW_v^{}$ and $\vPh'=\otimes_v^{}\bfW_v'$ so that 
\[I(\phi_l):=I(\bfW_l^{},\bfW_l')\neq 0. \]
If we write $(\vPi(g)\vPh,\vPh')_{\calk\times\calk'}=\La\!\La \bfW_\infty^H,\bfW_\infty^{H\prime}\Ra\!\Ra_\infty^{}\prod_l\phi_l(g_l)$, then $(^\tau\!\vPi(g){}^\tau\!\vPh,{}^\tau\!\vPh')_{\calk\times\calk'}=\La\!\La \bfW_\infty^H,\bfW_\infty^{H\prime}\Ra\!\Ra_\infty^{}\prod_l{}^\tau\!\phi_l(g_l)$ for $g=(g_l)\in \bfG(\widehat{\QQ})$. 
One can easily show that $\tau(I(\phi_l^{}))=I(^\tau\!\phi_l)$ (cf. the proofs of \cite[Theorem A]{Grobner} and \cite[ Lemma 2.6]{YL}). 
Since $\tau(\scrl(\pi_l,\sig_l))=\scrl(^\tau\!\pi_l,{}^\tau\!\sig_l)$ by \cite[Lemma 2.4]{YL}, we get  
\begin{align*}
\tau\left(\frac{L\left(\frac{1}{2},\pi\times\sig\right)}{L(1,\pi,\mathrm{Ad})L(1,\sig,\mathrm{Ad})}\right)=&\frac{\scrp_{\calk'}(^\tau\!\vPh)\scrp_{\calk'}(^\tau\!\vPh')}{(^\tau\!\vPh,{}^\tau\!\vPh')_{\calk\times\calk'}}c_\infty^{-1}\frac{2^{\vka_\pi+\vka_\sig}}{\tau(\gam_G)}\prod_l\frac{\scrl(^\tau\!\pi_l,{}^\tau\!\sig_l){}^\tau\!\vph_{\vPi_l}(1)}{I(^\tau\!\phi_l)}\\
=&\frac{\gam_GL\bigl(\frac{1}{2},{}^\tau\!\pi\times\!\,^\tau\!\sig\bigl)}{\tau(\gam_G)L(1,{}^\tau\!\pi,\Ad)L(1,{}^\tau\!\sig,\Ad)}, 
\end{align*} 
applying $\tau$ to the formula in Corollary \ref{cor:21}. 
The proof is complete by (\ref{tag:20}). 
\end{proof}


\subsection{An outer involution}\label{ssec:22}

To relate the central value to the square of the period, we introduce an involution of automorphic forms on unitary groups. 
Since $T$ has entries in $\QQ$, we can define an outer automorphism of $G$ by 
\[^\tht:g\mapsto g_{}^\tht=T^{-1}\trs g^{-1}T. \]
Take $\xi_\frkN=(\xi_{\frkN,v})\in G(\widehat{\QQ})$ so that  
\[(\xi_\frkN^{-1}\calk_0(\frkN)\xi_\frkN^{})^\tht=\calk_0(\frkN). \]
For a function on $\vph$ on $G(\AA)$, we define a function $\vph^\tht:G(\AA)\to\CC$ by $\vph^\tht(g)=\vph(g^\tht)$.

Let $\pi\simeq\otimes_v'\pi_v^{}$ be an irreducible automorphic representation of $G(\AA)$ such that $\pi_q$ admits a $\calk_q$-invariant vector for every non-split prime $q$. 
For each split prime $l$ we denote the conductor of $\pi_l$ by $c(\pi_l)$. 
Let $N_\pi=\prod_ll^{c(\pi_l)}$ be the conductor of $\pi$ and $N=\prod_{l\neq p}l^{c(\pi_l)}$ the tame conductor of $\pi$.  
Fix ideals $\frkN_\pi$ and $\frkN$ of $\frkr$ such that $\frkr/\frkN_\pi\simeq\ZZ/N_\pi\ZZ$ and $\frkr/\frkN\simeq\ZZ/N\ZZ$. 

Let $\pi^\vee\simeq\otimes_v'\pi_v^\vee$ denote the contragredient representation of $\pi$. 
We define representations $\pi^\tht$ of $G(\AA)$ and $\pi_v^\tht$ of $G(\QQ_v)$ as the twists $\pi^\tht(g)=\pi(g^\tht)$ and $\pi_v^\tht(g_v^{})=\pi(g_v^\tht)$ by $\tht$ for $g\in G(\AA)$ and $g_v\in G(\QQ_v)$. 
It is well-known that $\pi_v^\vee\simeq\pi_v^\tht$ (see \cite{MVW}). 
Since 
\[\vph^\tht(gh)=\vph((gh)^\tht)=(\pi^\tht(h)\vph)^\tht(g), \]
and $\pi^\vee\simeq \pi^\tht$, we have $\{\overline{\vph}\;|\;\vph\in\pi\}=\{\vph^\tht\;|\;\vph\in\pi\}$ by the global multiplicity one for unitary groups (\cite[Theorem 13.3.1]{Rog90}), where the automorphic form $\overline{\vph}$ is defined by $\overline{\vph}(g)=\overline{\vph(g)}$. 
Let $\vph_\pi\in\pi$ be a highest weight essential vector with respect to $\calk_0(\frkN)$ (see Definition \ref{def:41}). 
Then $\pi^\vee(\xi_{\frkN_\pi})\vph_\pi^\tht\in\pi^\vee$ is a lowest weight essential vector with respect to $\calk_0(\frkN_\pi)$. 

Assume that $H^\tht=H$. 
Similarly we associate $\vph^\tht\in\scra(H)$, $\vPh^\tht\in\scra(\bfG)$ and the automorphic representation $\sig^\tht\simeq\otimes_v'\sig_v^\tht$ to automorphic forms $\vph\in\scra(H)$, $\vPh\in\scra(\bfG)$ and an automorphic representation $\sig\simeq\otimes_v'\sig_v$ of $H(\AA)$. 


\subsection{A factorization of the dual representation}\label{ssec:23}
 
Define the longest Weyl element $w_n\in\GL_n(F)$ by 
\begin{align*}
w_1&=1, & 
w_n&=\begin{bmatrix} & 1 \\ w_{n-1} & \end{bmatrix} &
(n&\geq 2). 
\end{align*}
Let $l$ be a split rational prime. 
We view $\pi_l$ and $\sig_l$ as representations of $\GL_3(\QQ_l)$ and $\GL_2(\QQ_l)$ via $\imath_\frkl$. 
Let $W_{\pi_l^{}}\in\pi_l$ be the normalized essential Whittaker vector with respect to $\addchar_l$, and $W_{\pi_l^\vee}\in\pi_l^\vee$ the normalized essential Whittaker vector with respect to $\addchar^{-1}_l$. 
For $W\in\pi_l^{}$ we define $W^\tht\in\pi_l^\vee$ by 
\[W^\tht(g)=\pi_l^\vee(T^{-1})\widetilde{W}(g)=W(w_n\trs g^{-1}T), \]
where $\widetilde{W}(g)=W(w_n\trs g^{-1})$.  
It is important to note that 
\[\pi_l^\vee(h) W^\tht(g)=W^\tht(gh)=W((gh)^\tht)=(\pi_l^\tht(h)W)^\tht(g). \]
Proposition \ref{prop:41} below says that 
\beq
W_{\pi_l^\vee}=\vep(1/2,\pi_l,\addchar_l)^{-n}\pi_l^\vee(\xi_{\frkN_\pi,l}) W_{\pi_l}^\tht \label{tag:21}
\eeq
for a suitable choice of $\xi_{\frkN_\pi,l}$. 

When $q$ remains a prime in $\frkr$, we fix $\calk_q$-invariant vectors $W_{\pi_q}\in\pi_q$ and $W_{\pi_q^\vee}\in\pi_q^\vee$. 
Fix a highest weight vector $W_{\pi_\infty}\in\pi_\infty$ and a lowest weight vector $W_{\pi_\infty^\vee}\in\pi_\infty^\vee$. 
Let $^\tht:\pi_v^\tht\simeq\pi_v^\vee$ be the $G(\QQ_v)$-equivariant isomorphism determined by $W_{\pi_v}^\tht=W_{\pi_v^\vee}^{}$. 

To apply Corollary \ref{cor:21} to $\vPh'=\vPh^\tht$, we need explicate the factorization of $\vPh^\tht$. 
Fix isomorphisms $\pi\simeq\otimes_v'\pi_v^{}$ and $\pi^\vee\simeq\otimes_v'\pi_v^\vee$ so that 
\begin{align*}
\vph_\pi&=\otimes_v'W_{\pi_v}^{}, &
\vph_\pi^\tht&=\otimes_v'W_{\pi_v}^\tht. 
\end{align*}
Using this factorization $\pi^\vee\simeq\otimes_v'\pi_v^\vee$, we define a cusp form $\vph_{\pi^\vee}\in\pi^\vee$ by 
\[\vph_{\pi^\vee}=\otimes_v'W_{\pi_v^\vee}. \] 

\begin{lemma}\label{lem:21}
If $\vph=\otimes_v'W_v^{}\in\pi$ is factorizable and if $W_v=W_{\pi_v}$ for all non-split places $v$, then $\vph^\tht=\otimes_v'W_v^\tht$. 
\end{lemma}

\begin{proof}
Define a finite set $\frkS_\vph=\{v\;|\;W_v\neq W_{\pi_v}\}$. 
For $v\in\frkS_\vph$ there are $c_{v,i}\in\CC$ and $g_{v,i}\in G(\QQ_v)$ such that 
\begin{align*}
W_v&=\calu_v^{}W_{\pi_v}, & 
\calu_v^{}&=\sum_ic_{v,i}^{}\pi_v^{}(g_{v,i}^\tht). 
\end{align*}
We have $\vph=\otimes_{v\in\frkS_\vph}\calu_v^{}\cdot \vph_\pi$. 
Put $\calv_v^{}=\sum_ic_{v,i}^{}\pi_v^\vee(g_{v,i}^{})$. 
Then 
\[\vph^\tht=\otimes_{v\in\frkS_\vph}\calv_v^{}\cdot\vph_\pi^\tht
=(\otimes_{v\in\frkS_\vph}\calv_v^{}W_{\pi_v}^\tht)\otimes(\otimes_{v\notin\frkS_\vph}'W_{\pi_v}^\tht). \]
Lemma \ref{lem:21} follows from the observation that $\calv_v^{}W_{\pi_v}^\tht=(\calu_v^{}W_{\pi_v})^\tht=W_v^\tht$. 
\end{proof}

Let $\sig\simeq\otimes_v'\sig_v^{}$ be an irreducible automorphic representation of $H(\AA)$ such that $\sig_q$ admits $\calk_q'$-invariant vector for every non-split prime $q$. 
Put $N'=\prod_{l\neq p}l^{c(\sig_l)}$. 
We take an ideal $\frkN'$ of $\frkr$ such that $\frkr/\frkN'\simeq\ZZ/N'\ZZ$. 
Define an open compact subgroup $\calk'_0(p^n\frkN')$ of $H(\widehat{\QQ})$ and take $\xi_{\frkN'}'\in H(\AA)$ in a similar way. 
Let $\vph_\sig\in\sig$ be a highest weight essential vector with respect to $\calk_0'(\frkN')$. 
Fix the factorizations $\sig\simeq\otimes_v'\sig_v^{}$ and $\sig^\vee\simeq\otimes_v'\sig_v^\vee$ such that 
\begin{align*}
\vph_\sig^{}&=\otimes_v'W_{\sig_v}^{}, &
\vph_\sig^\tht&=\otimes_v'W_{\sig_v}^\tht. 
\end{align*}
Using the factorization, we define a cusp form $\vph_{\sig^\vee}\in\sig^\vee$ by $\vph_{\sig^\vee}=\otimes_v'W_{\sig_v^\vee}$. 

For each place $v$ we put 
\begin{align*}
\bfW_{\vPi_v}&=W_{\pi_v}\otimes W_{\sig_v}\in\vPi_v, & 
\bfW_{\vPi_v^\vee}&=W_{\pi_v^\vee}\otimes W_{\sig_v^\vee}\in\vPi_v^\vee  
\end{align*}
and define the $G(\QQ_v)$-equivariant isomorphism $^\tht:\vPi_v^\tht\simeq\vPi_v^\vee$ by 
\[(W\otimes W')^\tht=W^\tht\otimes W^{\prime\tht}. \]

Put 
\begin{align*}
\vPh_\vPi&=\vph_\pi\otimes\vph_\sig, &
\vPh_{\vPi^\vee}&=\vph_{\pi^\vee}\otimes\vph_{\sig^\vee}, &  
I_\infty&=c_\infty\frac{I(\bfW_\infty^H,\bfW_\infty^{H\tht})}{\La\!\La \bfW_{\vPi_\infty},\bfW_{\vPi_\infty^\vee}\Ra\!\Ra_\infty^{}}. 
\end{align*} 
Since $\scrp(\vPh^\tht)=\scrp(\vPh)$, one can deduce the following formula from Lemma \ref{lem:21} and Corollary \ref{cor:21}. 

\begin{corollary}\label{cor:22}
Let $\vPi$ be an irreducible tempered cuspial automorphic representation of $\bfG(\AA)$ such that $\vPi_q$ is unramified for all non-split rational primes $q$ and whose archimedean part $\vPi_\infty$ satisfies ($H_2$). 
If $\vPh=\otimes_v'\bfW_v^{}\in\vPi$ is factorizable and $\bfW_\infty^{}=\bfW_\infty^H$, then 
\[\frac{\scrp_{\calk'}(\vPh)^2}{(\vPh_\vPi^{},\vPh_{\vPi^\vee}^{})_{\calk\times\calk'}}
=\frac{\gam_G}{2^{\vka_\pi+\vka_\sig}}I_\infty\frac{L\left(\frac{1}{2},\pi\times\sig\right)}{L(1,\pi,\mathrm{Ad})L(1,\sig,\mathrm{Ad})}\prod_l\frac{I(\bfW_l^{},\bfW_l^\tht)}{\scrl(\pi_l\times\sig_l)\La\!\La \bfW_{\vPi_l},\bfW_{\vPi_l^\vee}\Ra\!\Ra_l^{}}. \]
\end{corollary}


\subsection{The case $m=3$}\label{ssec:26}

From now on we suppose that $G$ and $H$ are definite unitary groups in three and two variables. 
We can let
\begin{align*}
T&=\begin{bmatrix} t_1 & 0 & 0 \\ 0 & 1 & 0 \\ 0 & 0 & t_2 \end{bmatrix}, &
T'&=\begin{bmatrix} t_1 & 0 \\ 0 & t_2 \end{bmatrix} 
\end{align*}
by multiplying $T$ by an appropriate constant. 
We view $H=\U(T')$ as a subgroup of $G=\U(T)$ via the embedding 
\[\iot\biggl(\begin{bmatrix} a & b \\ c & d \end{bmatrix}\biggl)=\begin{bmatrix} a & 0 & b \\ 0 & 1 & 0 \\ c & 0 & d \end{bmatrix}. \]

The finite set $\vSi^-_T$ consists of non-split rational primes $q$ at which $(\;,\;)_{T'}$ does not split, i.e., $q\in\vSi^-_T$ if and only if $\eps_{E_q/\QQ_q}(-\det T)=-1$. 
For simplicity we assume that if $q\in\vSi^-_T$, then $q$ is odd and $\sig_q$ is the trivial representation of the compact unitary group $H(\QQ_q)$. 

When $q\notin\vSi^-_T\setminus\vSi_E^r$, we define maximal compact subgroups by 
\begin{align*} 
\calk_q&=G(\QQ_q)\cap\GL_3(\frkr_q^{}), & 
\calk'_q&=H(\QQ_q)\cap\GL_2(\frkr_q^{}),  
\end{align*} 
assuming that $t_1,t_2\in\ZZ_q^\times$. 
For each $q\in\vSi^-_T\setminus\vSi_E^r$ we take maximal compact subgroups   
\begin{align*}
\calk_q&=G(\QQ_q)\cap\GL(L_q), &
\calk_q'&=H(\QQ_q),  
\end{align*}
where $L_q$ is a maximal integral lattice of the Hermitian space $(W\otimes\QQ_q,q(\;,\;)_T)$ which contains $e_2:=\trs(0,1,0)$ and such that $q\cdot(e_2,L_q)_T=\frkr_q$ (see \S \ref{ssec:e1} for details). 
We take $\xi_\frkN=(\xi_{\frkN,l})\in G(\widehat{\QQ})$ and $\xi_{\frkN'}'=(\xi_{\frkN',l}')\in H(\widehat{\QQ})$ defined by
\begin{align*}
\xi_{\frkN,l}&=\begin{cases}
\imath_\frkl^{-1}\Biggl(\begin{bmatrix} N\ono_2 & \\ & 1 \end{bmatrix}^{-1}T\Biggl) &\text{if $l|N$, }\\
 \ono_3 &\text{otherwise, }
\end{cases}\\
\xi_{\frkN',l}'&=\begin{cases}
\imath_\frkl^{\prime-1}\Biggl(\begin{bmatrix} N' & 0 \\ 0  & 1 \end{bmatrix}^{-1}T'\Biggl) &\text{if $l|N'$, }\\
 \ono_2 &\text{otherwise. }
\end{cases}
\end{align*}


\subsection{Local integrals at $\infty$, $\vSi_E^r$ and $\vSi^-_T$}\label{ssec:27}

Proposition \ref{prop:b2} below gives 
\[I_\infty=2^{-2}(\dim\pi_\infty)(\dim\sig_\infty). \]
If $q\notin\vSi^-_T$, if $q$ and $pNN'$ are coprime, and if $\bfW_q^{}=\bfW_{\vPi_q}^{}$ is spherical, then since $\bfW_q^\tht=\bfW_{\vPi_q^\vee}^{}$, we have  
\[I(\bfW_q^{},\bfW_q^\tht)=\scrl(\pi_q\times\sig_q)\La\!\La \bfW_{\vPi_q},\bfW_{\vPi_q^\vee}\Ra\!\Ra_q^{} \]
by Theorem 2.12 of \cite{NHarris} and Proposition \ref{prop:d2} below (cf. Remark \ref{rem:a2}). 

When $q\in\vSi^-_T$, we let $\bfW_q^{}=\bfW_{\pi_q}^{}$ be a spherical spherical vector. 
Recall that $\sig_q$ is the trivial representation of $H(\QQ_q)$. 
Put 
\[I_-=\prod_{q\in\vSi^-_T}\frac{I(\bfW_q^{},\bfW_q^\tht)}{\scrl(\pi_q\times\sig_q)\La\!\La \bfW_{\vPi_q},\bfW_{\vPi_q^\vee}\Ra\!\Ra_q^{}}. \]
Proposition \ref{prop:e2} shows that  
\beq
I_-=\prod_{q\in\vSi^-_T\setminus\vSi_E^r}L(1,\eps_{E_q/\QQ_q})^2. \label{tag:22}
\eeq 
Since $\lam_q$ is either $\frac{1}{2}$ or $1$ according as $q$ is ramified in $E$ or not, we have 
\[\gam_G=D_E^{-3}2^{5+t_E}, \]
where $t_E$ is the number of prime numbers which are ramified in $E$. 

For each split prime $l$ we put 
\begin{align*}
B_{\pi_l}&=\La W_{\pi_l}^{},W_{\pi_l^\vee}^{}\Ra^{}_l, &
B_{\sig_l}&=\La W_{\sig_l}^{},W_{\sig_l^\vee}^{}\Ra_l', \\
\calb_{\pi_l}&=\frac{\zet_l(3)}{L^\GL(1,\pi_l^{}\times\pi_l^\vee)}B_{\pi_l}, &
\calb_{\sig_l}&=\frac{\zet_l(2)}{L^\GL(1,\sig_l^{}\times\sig_l^\vee)}B_{\sig_l}  
\end{align*} 
(cf. (\ref{tag:a3})), where $W_{\sig_l^{}}\in\sig_l$ is the normalized essential Whittaker vector with respect to $\addchar^{-1}_l$, $W_{\sig_l^\vee}\in\sig_l^\vee$ the normalized essential Whittaker vector with respect to $\addchar_l$, and the local pairings $\La\;,\;\Ra_l^{}$ and $\La\;,\;\Ra_l'$ are constructed by (\ref{tag:32}). 
To avoid possible confusion, we recall that 
\begin{align*}
L(s,\pi_l^{},\mathrm{Ad})&=L^\GL(s,\pi_l^{}\times\pi_l^\vee), & 
L(s,\sig_l^{},\mathrm{Ad})&=L^\GL(s,\sig_l^{}\times\sig_l^\vee). 
\end{align*}
We regard $\pi_l$ and $\sig_l$ as representations of unitary groups in the left hand side and representations of general linear groups in the right hand side. 
We denote the Petersson pairings with respect to $\d_\calk g$ and $\d_{\calk'}h$ by $(\;,\;)_\calk$ and $(\;,\;)_{\calk'}$.

\begin{corollary}\label{cor:23}
Assumption being as in Corollary \ref{cor:21}, if $\bfW_\infty^{}=\bfW_\infty^H$ and if $\bfW_l^{}=\bfW_{\vPi_l}^{}$ is a spherical vector unless $pNN'$ is divisible by $l$, then 
\begin{align*}
&\frac{\scrp_{\calk'}(\vPh)^2}{(\vph_\pi^{},\vph_{\pi^\vee})_\calk\dim\pi_\infty\cdot(\vph_\sig^{},\vph_{\sig^\vee})_{\calk'}\dim\sig_\infty}\\
=&\frac{2^{3+t_E}}{2^{\vka_\pi+\vka_\sig}D_E^3}I_-\frac{L\left(\frac{1}{2},\pi\times\sig\right)}{L(1,\pi,\mathrm{Ad})L(1,\sig,\mathrm{Ad})}\prod_{l|pNN'}\frac{I(\bfW_l^{},\bfW_l^\tht)}{\scrl(\pi_l\times\sig_l)B_{\pi_l}B_{\sig_l}}. 
\end{align*} 
\end{corollary}


\subsection{An application of the splitting lemma}\label{ssec:28}

Fix a split rational prime $l$. 
We regard $\pi_l$ and $\sig_l$ as representations of general linear groups via $\imath_\frkl^{}$ and $\imath_\frkl'$. 
Put $\zet_l(s)=(1-l^{-s})^{-1}$. 
Then the splitting lemma stated in \S \ref{ssec:a2} gives 
\[I(W_l^{}\otimes W_l', W_l^\tht\otimes W_l^{\prime\tht})
=\zet_l(1)Z\biggl(\frac{1}{2},\pi_l^{}(\vsi)W_l^{},W_l'\biggl)Z\biggl(\frac{1}{2},\pi^\vee_l(\vsi)W_l^\tht,W_l^{\prime\tht}\biggl), \]
where
\beq
\vsi=\begin{bmatrix}
1 & 0 & 0 \\ 
0 & 0 & 1 \\
0 & 1 & 0   
\end{bmatrix}. \label{tag:23}
\eeq
We regard $H(\QQ_p)$ as a subgroup of $G(\QQ_p)$ via the embedding $\iot$ while we use the embedding $\iot':\GL_2(\QQ_p)\hookrightarrow\GL_3(\QQ_p)$ defined by $\iot'(g)=\begin{bmatrix} g & \\ & 1 \end{bmatrix}$ to define the JPSS integral. 
Since $\vsi T^{-1}=\begin{bmatrix} T^{\prime-1} & \\ & 1 \end{bmatrix}\vsi$, we have 
\begin{align*}
Z\biggl(\frac{1}{2},\pi^\vee_l(\vsi)W_l^\tht,W_l^{\prime\tht}\biggl)
&=Z\biggl(\frac{1}{2},\pi^\vee_l(\vsi T^{-1})\widetilde{W_l},\sig^\vee_l(T^{\prime-1})\widetilde{W_l'}\biggl)\\
&=Z\biggl(\frac{1}{2},\pi^\vee_l(\vsi)\widetilde{W_l},\widetilde{W_l'}\biggl)\\
&=\gam^\GL\biggl(\frac{1}{2},\pi_l\times\sig_l,\addchar_l\biggl)Z\biggl(\frac{1}{2},\pi_l^{}(\vsi)W_l^{},W_l'\biggl) 
\end{align*}
by the invariance and the functional equation (\ref{tag:a2}) of the JPSS integrals. 
We can rewrite the identity above as  
\[\frac{I(W_l^{}\otimes W_l', W_l^\tht\otimes W_l^{\prime\tht})}{\scrl(\pi_l\times\sig_l)B_{\pi_l}^{}B_{\sig_l}^{}}=\frac{\scri(W_l^{}\otimes W_l')}{\calb_{\pi_l} \calb_{\sig_l}}, \]
where 
\[\scri(W_l^{}\otimes W_l')=\gam^\GL\biggl(\frac{1}{2},\pi_l\times\sig_l,\addchar_l\biggl)\frac{Z\bigl(\frac{1}{2},\pi_l^{}(\vsi)W_l^{},W_l'\bigl)^2}{L\bigl(\frac{1}{2},\pi_l\times\sig_l\bigl)}. \]

\begin{definition}\label{def:21}
Put
\begin{align*}
{\rm Pet}(\pi)&=2^{\vka_\pi}L(1,\pi,\Ad)\prod_{l|pN}\calb_{\pi_l}; &
{\rm Pet}(\sig)&=2^{\vka_\sig}L(1,\sig,\Ad)\prod_{l|pN'}\calb_{\sig_l},\\
\eta_{\vph_\pi}&=(\vph_\pi,\vph_{\pi^\vee})_\calk\dim\pi_\infty;&\eta_{\vph_\sig}&=(\vph_\sig,\vph_{\sig^\vee})_{\calk'}\dim\sig_\infty.
\end{align*}
\end{definition}


\begin{corollary}\label{cor:24}
Assumption being as in Corollary \ref{cor:23}, we have 
\[\frac{\scrp_{\calk'}(\vPh)^2}{\eta_{\vph_\pi}\eta_{\vph_\sig}}
=\frac{2^{3+t_E}}{D_E^3}I_-\frac{L\bigl(\frac{1}{2},\pi\times\sig\bigl)}{{\rm Pet}(\pi){\rm Pet}(\sig)}\prod_{l|pNN'}\scri(\bfW_l^{}). \]
\end{corollary}


\subsection{The local integral at a ramified split prime $l$}\label{ssec:29}

Let $l\neq p$ be a rational prime that is split in $E$. 
We write $\ome_{\sig_l}$ for the central character of $\sig_l$, which is viewed as a character of $\QQ_l^\times$ via $\imath_\frkl$. 
Observe that 
\[\scri(W_l^{}\otimes W_l')=\vep^\GL\biggl(\frac{1}{2},\pi_l\times\sig_l,\addchar_l\biggl)\Biggl(\frac{Z\bigl(\frac{1}{2},\pi_l^{}(\vsi)W_l^{},W_l'\bigl)}{L^\GL\bigl(\frac{1}{2},\pi_l\times\sig_l\bigl)}\Biggl)^2. \]

If $\sig_l$ is unramified, then Th\'{e}orem\`{e} on p.~208 of \cite{JPSS} gives 
\beq
Z(s,W_{\pi_l^{}},W_{\sig_l})=L^\GL(s,\pi_l\times\sig_l).  \label{tag:24}
\eeq
Put $\vsi_\frkl^{}=\imath_\frkl^{-1}(\vsi)$. 
We see that  
\[\scri(\pi_l(\vsi_\frkl)W_{\pi_l}^{}\otimes W_{\sig_l}^{})=\vep^\GL\biggl(\frac{1}{2},\pi_l\times\sig_l,\addchar_l\biggl)=:\frkf_{\pi_l,\sig_l}. \]

Finally, we consider the case $\sig_l$ is ramified. 
Put $f_l=c(\ome_{\sig_l})$. 
We define $\vTh^{\ome_{\sig_l}} W_{\pi_l}\in\scrw_{\addchar_l}(\pi_l)$ by 
\begin{align*}
\vTh^{\ome_{\sig_l}} W_{\pi_l}&=\sum_{i,j\in(\ZZ_l/l^{f_l}\ZZ_l)^\times}\sum_{y\in\ZZ_l/l^{2f_l}\ZZ_l}
\ome_{\sig_l}(ij)\pi_l\left(\begin{bmatrix} 1 & \frac{i}{l^{f_l}} & \frac{y}{l^{2f_l}} \\ 0 & 1 & \frac{j}{l^{f_l}} \\ 0 & 0 & 1 \end{bmatrix}\right)W_{\pi_l}, \\
\vTh^{\ome_{\sig_l}}_{\frkN'} W_{\pi_l}&=\pi_l(\iot'(\xi_{\frkN',l}'w_2^{}))\vTh^{\ome_{\sig_l}} W_{\pi_l}. 
\end{align*}
Proposition \ref{prop:44} below shows that if $c(\sig_l)\leq 2c(\ome_{\sig_l})$, then 
\[Z(s,\Tht^{\ome_{\sig_l}}_{\frkN'} W_{\pi_l}^{},W_{\sig_l})=\frac{l^{3c(\ome_{\sig_l})}\ome_{\sig_l}(l)^{2c(\ome_{\sig_l})}\vep^\GL\bigl(\frac{1}{2},\sig_l,\addchar_l\bigl)}{\vep\bigl(\frac{1}{2},\ome_{\sig_l},\addchar_l\bigl)^2[\GL_2(\ZZ_l):\calk_0^{(2)}(l^{2c(\ome_{\sig_l})}\ZZ_l)]}. \]
Put 
\[\frkf_{\pi_l,\sig_l}=\vep^\GL\biggl(\frac{1}{2},\pi_l\times\sig_l,\addchar_l\biggl)\frac{l^{6c(\ome_{\sig_l})}\ome_{\sig_l}(l)^{4c(\ome_{\sig_l})}\vep^\GL\bigl(\frac{1}{2},\sig_l,\addchar_l\bigl)^2}{\vep\bigl(\frac{1}{2},\ome_{\sig_l},\addchar_l\bigl)^4L^\GL\bigl(\frac{1}{2},\pi_l\times\sig_l\bigl)^2}. \]
It follows that 
\beq
\scri(\pi_l^{}(\vsi_\frkl^{})\vTh^{\ome_{\sig_l}}_{\frkN'}W_{\pi_l}^{},W_{\sig_l})=\frac{\frkf_{\pi_l,\sig_l}}{[\GL_2(\ZZ_l):\calk_0^{(2)}(l^{2c(\ome_{\sig_l})}\ZZ_l)]^2}. \label{tag:25}
\eeq


\subsection{The local integral at $p$ and the modified $p$-factor}\label{ssec:212}

Assume that $\pi$ and $\sig$ are an irreducible $p$-ordinary automorphic representations with respect to $\iot_p$, namely, $\pi_p$ is the irreducible generic constituent of a principal series $I(\nu_p,\rho_p,\mu_p)$, and $\sig_p$ is the irreducible generic constituent of $I(\mu'_p,\nu'_p)$, where $\nu_p^{}$, $\rho_p^{}$, $\mu_p^{}$; $\mu_p'$, $\nu_p'$ are algebraic characters of $\QQ_p^\times$, such that 
\begin{align*}
&p^{-k_3-1}\nu_p(p), & &p^{-k_2}\rho_p(p), & &p^{-k_1+1}\mu_p(p); & &p^{-k_2'-\frac{1}{2}}\mu'_p(p), & &p^{-k_1'+\frac{1}{2}}\nu'_p(p)
\end{align*} 
are $p$-units with respect to $\iot_p$. \begin{definition}\label{def:22}
Define the modified $p$-factor $\cale(\pi_p,\sig_p)$ by 
\begin{align*}
\cale(\pi_p,\sig_p)^{-1}
=&L\biggl(\frac{1}{2},\pi_p\times\sig_p\biggl)\gam\biggl(\frac{1}{2},\mu_p^{}\nu_p',\addchar_p\biggl)\gam\biggl(\frac{1}{2},\rho_p^{}\nu_p',\addchar_p\biggl)\gam\biggl(\frac{1}{2},\mu_p^{}\mu_p',\addchar_p\biggl)\\
&\times\gam\biggl(\frac{1}{2},(\nu_p\nu_p')^{-1},\addchar_p^{-1}\biggl)\gam\biggl(\frac{1}{2},(\rho_p\mu_p')^{-1},\addchar_p^{-1}\biggl)\gam\biggl(\frac{1}{2},(\nu_p\mu_p')^{-1},\addchar_p^{-1}\biggl). 
\end{align*}
\end{definition}
Let $W_{\pi_p}^{\ord}\in\pi_p^{}$ and $W_{\sig_p}^{\ord}\in\sig_p^{}$ the normalized  $p$-ordinary Whittaker functions in (\ref{E:Word}). These are eigenvectors of the $U_p$-operators with specified eigenvalues (see Proposition \ref{prop:31}, Remarks \ref{rem:51} and \ref{rem:53} below). Put 
\begin{align*}
\alp_{\pi_p}&=p^2\rho_p(p)\mu_p(p)^2, & 
\alp_{\sig_p}&=p^{1/2}\nu'_p(p).  
\end{align*}
Then 
\begin{align*}
U_p^{}W^{\ord}_{\pi_p}&=\alp_{\pi_p}W^{\ord}_{\pi_p}, & 
U_p'W^{\ord}_{\sig_p}&=\alp_{\sig_p}W^{\ord}_{\sig_p}
\end{align*}
by Proposition \ref{prop:31}. 
Define an element $\bft_\ell\in G(\QQ_p)$ by 
\[\bft_\ell=\imath_\frkp^{-1}\left(\begin{bmatrix}
0 & 0 & -p^{-\ell} \\
0 & p^\ell & 1 \\
p^{2\ell} & p^\ell & 0
\end{bmatrix}\right). \]
If $\ell$ is sufficiently large, then Proposition \ref{prop:33} below gives 
\begin{align}
&\left(\frac{\zet_p(1)}{\zet_p(2)(p^{-5}\alp_{\pi_p}\alp_{\sig_p})^\ell}\right)^2\scri(\pi_p(\bft_\ell)W^{\ord}_{\pi_p}\otimes W^{\ord}_{\sig_p}) \label{tag:26}\\
=&\frac{\gam^\GL\bigl(\frac{1}{2},\pi_p^{}\times\sig_p^{},\addchar_p\bigl)}{L\bigl(\frac{1}{2},\pi_p\times\sig_p\bigl)\bigl(\gam\bigl(\frac{1}{2},\mu_p^{}\nu_p',\addchar_p\bigl)\gam\bigl(\frac{1}{2},\rho_p^{}\nu_p',\addchar_p\bigl)\gam\bigl(\frac{1}{2},\mu_p^{}\mu_p',\addchar_p\bigl)\bigl)^2} \notag \\
=&\frac{\gam\bigl(\frac{1}{2},\nu_p^{}\nu_p',\addchar_p\bigl)\gam\bigl(\frac{1}{2},\rho_p^{}\mu_p',\addchar_p\bigl)\gam\bigl(\frac{1}{2},\nu_p^{}\mu_p',\addchar_p\bigl)}{L\bigl(\frac{1}{2},\pi_p\times\sig_p\bigl)\gam\bigl(\frac{1}{2},\mu_p^{}\nu_p',\addchar_p\bigl)\gam\bigl(\frac{1}{2},\rho_p^{}\nu_p',\addchar_p\bigl)\gam\bigl(\frac{1}{2},\mu_p^{}\mu_p',\addchar_p\bigl)} \notag\\
=&(\rho_p^{}\nu_p')(-1)\cale(\pi_p,\sig_p) \notag 
\end{align}
in view of the multiplicativity and functional equation of the gamma factor. 


\subsection{An explicit central value formula}\label{ssec:213}
 
Let $\pi\simeq\otimes'_v\pi_v^{}$ be an irreducible tempered automorphic representation of $G(\AA)$ and $\sig\simeq\otimes'_v\sig_v^{}$ that of $H(\AA)$ satisfying the following conditions: 
\begin{enumerate}
\item[$(H_0)$] if $q\in\vSi^-_T$, then $q$ is odd and $\sig_q$ is the trivial representation; 
\item[($H_1$)] $\pi_q$ is unramified for every non-split rational prime $q$; \\
$\sig_q$ is unramified for every non-split rational prime $q\notin\vSi^-_T$; 
\item[($H_2$)] $\pi_\infty$ and $\sig_\infty$ are discrete series such that 
\[\Hom_{H(\RR)}(\pi_\infty\otimes\sig_\infty,\CC)\neq\{0\}; \] 
\item[($H_3$)] $c(\sig_l)\leq 2c(\ome_{\sig_l})$ for every split rational prime $l\neq p$;  
\item[($H_4$)] 
$\pi_p$ is a generic constituent of a principal series $I(\nu_p,\rho_p,\mu_p)$; \\
$\sig_p$ is a generic constituent of a principal series $I(\mu_p',\nu_p')$.  
\end{enumerate}

Put $M=\prod_{l|N'}l^{c(\ome_{\sig_l})}$. 
Take a divisor $\frkM$ of $\frkN'$ such that $\frkr/\frkM\simeq\ZZ/M\ZZ$.  
Let $\vPh^\dagger\in\vPi$ be an ordinary $H(\RR)$-invariant essential vector. 
Define $\vTh^{\ome_\sig}_{\frkN'}\vPh^\dagger\in\vPi$ by  
\[\vTh^{\ome_\sig}_{\frkN'}\vPh^\dagger=\bfW_\infty^H\otimes \bigl(W^{\ord}_{\pi_p}\otimes W^{\ord}_{\sig_p}\bigl)\otimes\bigl(\otimes_{q\nmid pN'}\bfW_{\vPi_q}^{}\bigl)\otimes\bigl(\otimes_{l|N'}\vTh^{\ome_{\sig_l}}_{\frkN'}W_{\pi_l}\otimes W_{\sig_l}\bigl). \]
Recall the element $\vsi^{(p)}\in G(\widehat{\QQ})$ defined in (\ref{tag:12}). 
We combine Corollary \ref{cor:24} and (\ref{tag:25}), (\ref{tag:26}) to obtain the following formula. 

\begin{proposition}\label{prop:23}
Notations and assumptions being as above, if $\ell$ is sufficiently large, then  
\begin{multline*}
\frac{1}{\eta_{\vph_\pi}\eta_{\vph_\sig}}\left(\frac{\zet_p(1)}{\zet_p(2)}[\calk':\calk_0'(\frkM^2)]\frac{\scrp_{\calk'}(\pi(\bft_\ell\vsi^{(p)})\vTh^{\ome_\sig}_{\frkN'}\vPh^\dagger)}{(p^{-5}\alp_{\pi_p}\alp_{\sig_p})^\ell}\right)^2\\
=\frac{L\bigl(\frac{1}{2},\pi\times\sig\bigl)}{{\rm Pet}(\pi){\rm Pet}(\sig)}(\rho_p^{}\nu_p')(-1)\cale(\pi_p,\sig_p)\frac{2^{3+t_E}}{D_E^3}I_-\prod_{l|NN'}\frkf_{\pi_l,\sig_l}. 
\end{multline*}
\end{proposition}


\subsection{Statement of the main result}

\begin{definition}[$p$-modified period]\label{def:512}
Let $\bdsf\in e_{\ord}\bfS^{\U(n)}(\frkN,\chi,\calr)$ be a Hida family. 
Letting $\mu_1,\mu_2,\dots,\mu_n$ be chosen for $\bdpi_{\ulQ}$ as in Definition \ref{def:53}, we define the modified adjoint $p$-factor $\cale(\bdpi_{\ulQ,p},\Ad)=\cale(\bdpi_{\ulQ,p},\Ad,\addchar_p)$ as in Definition \ref{D:adjoint} for $\ulQ\in\frkX_{\bdsf}^{++}$. 
The subset $\frkY_{\bdsf}$ consists of $\ulQ\in\frkX^{++}_{\bdsf}$ such that $\ulQ(\bdsf)$ is new outside $p$. 
For $\ulQ\in\frkY_{\bdsf}$ we define the $p$-modified period by 
\begin{align*}\Ome_{\ulQ(\bdsf)}^\dagger&=[\calk:\calk_0(\frkN_{\bdpi_{\ulQ}})]{\rm Pet}(\bdpi_{\ulQ})\cale(\bdpi_{\ulQ,p},\Ad)\\
&=[\calk:\calk_0(\frkN_{\bdpi_{\ulQ}})]2^{\vka_{\bdpi_{\ulQ}}}\cale(\bdpi_{\ulQ,p},\Ad)L(1,\bdpi_{\ulQ},\Ad)\prod_{l|pN}\calb_{\bdpi_{\ulQ,l}},
\end{align*}
and define the Gross period by 
\[\Ome_{\ulQ(\bdsf)}=\frac{\Ome_{\ulQ(\bdsf)}^\dagger}{\ulQ(\eta_{\bdsf})}. \]
\end{definition}

Let 
\[\bdsf\in e_{\ord}\bfS^G(\frkN,\chi,\calr);\quad \bdsg\in e_{\ord}'\bfS^H(\frkN',\chi',\calr')\] 
be Hida families. 
For $\ulQ\in\frkX^{++}_{\bdsf}$ and $\ulQ'\in\frkX^{++}_{\bdsg}$, let $\bdpi_{\ulQ}$ and $\bdsig_{\ulQ'}$ be the automorphic representations of $G(\AA)$ and $H(\AA)$ associated with $\ulQ(\bdsf)$ and $\ulQ'(\bdsg)$ respectively. Choose $\nu_p^{}$, $\rho_p^{}$, $\mu_p^{}$; $\mu_p'$, $\nu_p'$ for $\bdpi_{\ulQ,p}$ and $\bdsig_{\ulQ',p}$ as in \S \ref{ssec:212}. 
We define the modified $p$-factor by  
\[\cale(\mathrm{Fil}^+\bfV_\calq)=\cale(\bdpi_{\ulQ,p},\bdsig_{\ulQ,p}). \]

\begin{theorem}\label{thm:22}
Assume that $q$ is odd whenever $H(\QQ_q)$ is compact. 
Let $M$ be the conductor of $\chi'$. 
Suppose that  
\beq
M^2\text{ is divisible by }N'. \tag{$H_3$}
\eeq
Then there exists a unique element $\scrl_{\bdsf,\bdsg}\in\mathrm{Frac}(\calr\widehat{\otimes}_\calo\calr')$ such that 
\[\calq(\scrl_{\bdsf,\bdsg})^2=\frac{\Gam(0,\bfV_\calq^{})L(0,\bfV_\calq^{})}{\Ome_{\ulQ(\bdsf)}\Ome_{\ulQ'(\bdsg)}}\cale(\mathrm{Fil}^+\bfV_\calq) \]
for $\calq=(\ulQ,\ulQ')\in\frkX^{\rm crit}\cap(\frkY_{\bdsf}\times\frkY_{\bdsg})$. 
\end{theorem}

\begin{remark}
The denominator of $\scrl_{\bdsf,\bdsg}$ is simply a product of explicit local $L$-factors. For each prime factor $l$ of $N'$ let $P_{\bdpi_l,\bdsig_l}\in\calr\widehat\otimes\calr'$ be an element such that $\calq(P_{\bdpi_l,\bdsig_l})=L^\GL\bigl(\frac{1}{2},\bdpi_{\ulQ,l}\times\bdsig_{\ulQ',l}\bigl)^{-1}$ for $\calq=(\ulQ,\ulQ')\in\frkX^+_\calr\times\frkX^+_{\calr'}$. 
It follows from the definition (\ref{tag:63}) that 
\[\scrl_{\bdsf,\bdsg}\cdot \prod_{l\mid N'}P_{\bdpi_l,\bdsig_l}\in\calr\widehat\otimes\calr'. \]
\end{remark}


\subsection{Proof of Theorem \ref{thm:12}}\label{ssec:65}

This subsection supposes that $NN'$ is odd. 
Let $M$ be an {\it odd} integer divisibly only by split primes.   
Fix a primitive character $\vrh$ of $(\frkr/\frkM)^\times$ such that $\chi'\vrh^{-2}$ has conductor $M$. 
We extend $\vrh$ to an automorphic character $\vrh_\AA=\prod_v\vrh_v$ of $\U(1)(\AA)$. 

Since $\ulQ'(\bdsg)$ has tame level $\frkN'$ for $\ulQ'\in \frkX^+_{\calr'}$, 
there exists a positive integer $A$ such that if $m\geq A$ and $M$ is divisible by $N^{\prime m}$, then $\bdsig_{\ulQ'}\otimes\vrh_\AA$ has tame conductor $M^2$ (cf. Remark \ref{rem:41} and (\ref{tag:41})) and $(\bdsig_{\ulQ',l}\otimes\vrh_l)_u$ is trivial for all prime factors $l$ of $N'$ and $\ulQ'\in\frkX_{\calr'}^+$. 
We enlarge $M$ so that $\bdpi_{\ulQ}^{}\otimes\vrh_\AA^{-1}$ has tame conductor $M^3$ and $(\bdpi_{\ulQ,l}^{}\otimes\vrh_l^{-1})_u$ is trivial for all prime factors $l$ of $N$ and $\ulQ\in\frkX_\calr^+$. 

Fix $(\ulQ_0^{},\ulQ_0')\in\frkX^{\rm crit}$. 
Let 
\begin{align*}
f_\vrh&\in e_\ord\bfS^G_{k_{\ulQ_0}}(p^{\ell_{\ulQ}}\frkM^3,\calr,\chi\vrh^3), &
g_\vrh&\in e_\ord\bfS^H_{k_{\ulQ'_0}}(p^{\ell_{\ulQ'}}\frkM^2,\calr',\chi'\vrh^{-2})
\end{align*}
be $p$-ordinary newforms associated to $\bdpi_{\ulQ_0}\otimes\vrh_\AA^{-1}$ and $\bdsig_{\ulQ'_0}\otimes\vrh_\AA^{}$. 
Theorem \ref{thm:51} allows us to lift $f_\vrh$ and $g_\vrh$ to Hida families 
\begin{align*}
\bdsf_\vrh&\in e_\ord\bfS^G(\frkM^3,\calr,\chi\vrh^3), &
\bdsg_\vrh&\in e_\ord\bfS^H(\frkM^2,\calr',\chi'\vrh^{-2}). 
\end{align*}
For our choice of $\vrh$ we see that $\bdsg_\vrh$ satisfies ($H_3$), and 
\begin{align*}
\frkY_{\bdsf_\vrh}^{}&=\frkX_{\bdsf_\vrh}^{++}, & 
\Ome^\dagger_{\ulQ(\bdsf_\vrh)}&=[\calk_0(\frkN):\calk_0(\frkM^3)]\Ome^{(NN')}(V_{\ulQ(\bdsf)}), \\ 
\frkY_{\bdsg_\vrh}^{}&=\frkX_{\bdsg_\vrh}^{++}, & 
\Ome^\dagger_{\ulQ'(\bdsg_\vrh)}&=[\calk'_0(\frkN'):\calk'_0(\frkM^2)]\Ome^{(NN')}(V_{\ulQ'(\bdsg)}). 
\end{align*}
Theorem \ref{thm:22} applied to $\bdsf_\vrh$ and $\bdsg_\vrh$ shows that 
\[\calq(\scrl_{\bdsf_\vrh,\bdsg_\vrh})^2=\frac{\Gam(0,\bfV_\calq^{})L(0,\bfV_\calq^{})}{\Ome_{\ulQ(\bdsf_\vrh)}\Ome_{\ulQ'(\bdsg_\vrh)}}\cale(\mathrm{Fil}^+\bfV_\calq)\]
for $\calq=(\ulQ,\ulQ')\in\frkX^{\rm crit}\cap(\frkX_{\bdsf_\vrh}^{++}\times\frkX_{\bdsg_\vrh}^{++})$. 
Therefore 
\[L_p(\bfV)=[\calk_0(\frkN):\calk_0(\frkM^3)][\calk'_0(\frkN'):\calk'_0(\frkM^2)]\frac{\scrl_{\bdsf_\vrh,\bdsg_\vrh}^2}{\eta_{\bdsf_\vrh}\eta_{\bdsg_\vrh}}\in\mathrm{Frac}(\calr\widehat{\otimes}_\calo\calr'). \]
satisfies the interpolation formula in Theorem \ref{thm:12} for $\calq_0\in\frkX^{\rm crit}\cap(\frkX_{\bdsf_\vrh}^{++}\times\frkX_{\bdsg_\vrh}^{++})$. 
Since this interpolation formula determines $L_p(\bfV)$, it holds for all $\calq\in\frkX^{\rm crit}$. 


\subsection{Proof of Theorem \ref{thm:22}}\label{ssec:66}


\begin{proposition}\label{prop:62}
Notations and assumptions being as in Theorem \ref{thm:22}, we have  
\[\zet_p(1)^4\calq(\Tht_{\bdsF})^2=\frac{L\bigl(\frac{1}{2},\pi\times\sig\bigl)}{\Ome_{\ulQ(\bdsf)}\Ome_{\ulQ(\bdsg)}}\cale(\mathrm{Fil}^+\bfV_\calq)\chi'(-1)\frac{2^{3+t_E}D_T^2}{D_E^3}\prod_{l|NN'}\frkf_{\bdpi_{\ulQ,l},\bdsig_{\ulQ',l}} \]
for $\calq=(\ulQ,\ulQ')\in\frkX^{\rm crit}_{\bdsf,\bdsg}$, where $D_T=\prod_{q\in\vSi_T^-\bsl\vSi_E^r}q$. 
\end{proposition}

\begin{proof}
Let $\pi=\bdpi_{\ulQ}$ and $\sig=\bdsig_{\ulQ'}$. 
Hypotheses ($H_0)$ and ($H_1$) hold by (splt), ($H_2)$ holds by Remark \ref{rem:20}, and ($H_4$) holds by \cite[Lemma 5.4]{Geraghty}. 
Thanks to Remark \ref{rem:12}(\ref{rem:134}), we can apply Proposition \ref{prop:23} to $\ulQ(\bdsf)$ and $\ulQ(\bdsg)$. 
If $\ell$ is sufficiently large, then Proposition \ref{prop:61} yields
\begin{align*}
\zet_p(1)^4\calq(\Tht_{\bdsF})^2
&=\left(\frac{\zet_p(1)}{\zet_p(2)}[\calk':\calk_0'(\frkM^2)]\frac{\scrp_{\calk'}(\pi(\bft_\ell\vsi^{(p)})\vPh_F)}{(p^{-5}\alp_{\pi_p}\alp_{\sig_p})^\ell}\right)^2\prod_{q\in\vSi_T^-\setminus\vSi_E^r}(q+1) \\
&=\eta_{\vph_\pi}\eta_{\vph_\sig}\frac{L\bigl(\frac{1}{2},\pi\times\sig\bigl)}{{\rm Pet}(\pi){\rm Pet}(\sig)}(\rho_p^{}\nu'_p)(-1)\cale(\pi_p,\sig_p)\frac{2^{3+t_E}I_-D_T^2\prod_{l|NN'}\frkf_{\pi_l,\sig_l}}{D_E^3\prod_{q\in\vSi_T^-\setminus\vSi_E^r}L(1,\eps_{E_q/\QQ_q})^2}. 
\end{align*}
We can let the stabilization of $\vph_\pi$ be $f_{\bfv_{k_{\ulQ}}}$. 
Then Proposition \ref{prop:52}, Remark \ref{rem:52} and Corollary \ref{cor:42} below give 
\begin{align*}
\frac{\ulQ(\eta_{\bdsf})}{\eta_{\vph_\pi}}
&=\frac{[\calk:\calk_0(p^\ell\frkN)]\alpha_{\pi_p}^{-\ell}(\tau_{p^\ell\frkN}f_{\bfv_{k_{\ulQ}}},f_{\bfv_{k_{\ulQ}}})_\calk}{(\vph_\pi^{},\vph_{\pi}^\vee)_\calk}
=\frac{[\calk:\calk_0(p^\ell\frkN)]\alpha_{\pi_p}^{-\ell}\calb_{\pi_p^{\ord}}^{[\ell]}}{\calb_{\pi_p}}\\
&=[\calk:\calk_0(\frkN_\pi)]\cale(\pi_p,\Ad)\rho_p(-1) \end{align*}
Since $1=\ome_\sig(-1)=(-1)^{k_{Q_1'}-k_{Q_2'}}(\mu_p'\nu_p')(-1)\chi'(-1)$, we have 
\begin{align*}
\frac{\ulQ(\eta_{\bdsg})}{\eta_{\vph_\sig}}
&=(-1)^{k_{Q'_1}-k_{Q'_2}}[\calk':\calk'_0(\frkN_\sig^{})]\cale(\sig_p,\Ad)\mu'_p(-1)\\
&=[\calk':\calk'_0(\frkN_\sig^{})]\cale(\sig_p,\Ad)\nu'_p(-1)\chi'(-1). 
\end{align*}
The stated interpolation formula follows from (\ref{tag:22}).
\end{proof}



\begin{lemma}
For each prime factor $l$ of $NN'$ there is an element $\sqrt{\frkf_{\bdpi_l,\bdsig_l}}\in\calr$ such that for $\calq=(\ulQ,\ulQ')\in\frkX^{\rm crit}$
\[\calq(\sqrt{\frkf_{\bdpi_l,\bdsig_l}})^2=\frkf_{\bdpi_{\ulQ,l},\bdsig_{\ulQ',l}}. \]
\end{lemma}

\begin{proof}
We denote the $l$-primary part of $\chi'$ by $\chi'_l$. 
Recall that if $l|N'$, then 
\[\frkf_{\bdpi_{\ulQ,l},\bdsig_{\ulQ',l}}=\frac{l^{6c(\chi'_l)}\chi'_l(l)^{-4c(\chi'_l)}\vep^\GL\bigl(\frac{1}{2},\bdpi_{\ulQ,l}\times\bdsig_{\ulQ',l},\addchar_l\bigl)}{\vep\bigl(\frac{1}{2},\chi_l^{-1},\addchar_l^{}\bigl)^4\vep^\GL\bigl(\frac{1}{2},\bdpi_{\ulQ,l},\addchar_l\bigl)^2L^\GL\bigl(\frac{1}{2},\bdpi_{\ulQ,l}\times\bdsig_{\ulQ',l}\bigl)^2}. \]
If $l\nmid N'$, then $\frkf_{\bdpi_{\ulQ,l},\bdsig_{\ulQ',l}}=\vep^\GL\bigl(\frac{1}{2},\bdpi_{\ulQ,l}\times\bdsig_{\ulQ',l},\addchar_l\bigl)$. 
One can construct elements $\vep_{\bdpi_l},\vep_{\bdpi_l,\bdsig_l}\in\calr^\times$ and $P_{\bdpi_l,\bdsig_l}\in\calr$ such that 
\begin{align*}
\calq(\vep_{\bdpi_l})&=\vep^\GL\biggl(\frac{1}{2},\bdpi_{\ulQ,l},\addchar_l\biggl), &
\calq(\vep_{\bdpi_l,\bdsig_l})&=\vep^\GL\biggl(\frac{1}{2},\bdpi_{\ulQ,l}\times\bdsig_{\ulQ',l},\addchar_l\biggl), \\
& & 
\calq(P_{\bdpi_l,\bdsig_l})&=\frac{1}{L^\GL\bigl(\frac{1}{2},\bdpi_{\ulQ,l}\times\bdsig_{\ulQ',l}\bigl)}
\end{align*}
for $\calq=(\ulQ,\ulQ')\in\frkX^{\rm crit}$ as in the proof of Proposition 6.11 of \cite{MH}.  
\end{proof}

Define the fudge factor $\sqrt{\frkf_{\bdpi,\bdsig}}\in\calr$ by 
\[\sqrt{\frkf_{\bdpi,\bdsig}}=\prod_{l|NN'}\sqrt{\frkf_{\bdpi_l,\bdsig_l}}. \]
We can prove Theorem \ref{thm:12} by defining the theta element $\scrl_{\bdsf,\bdsg}\in\mathrm{Frac}(\calr\widehat{\otimes}\calr')$ attached to $\bdsf$ and $\bdsg$ by   
\beq
\scrl_{\bdsf,\bdsg}=\frac{\zet_p(1)^2D_E\sqrt{\chi'(-1)D_E}}{2^\frac{3+t_E}{2}D_T\sqrt{\frkf_{\bdpi,\bdsig}}}\Tht_{\bdsF} \label{tag:63}
\eeq
in view of Proposition \ref{prop:62}. 


\section{The calculation of the local zeta integral at the $p$-adic case}\label{sec:3}

\subsection{The JPSS integrals}\label{ssec:a1}

Let $F$ be a finite extension of $\QQ_p$ which contains the integer ring $\frko$ having a single prime ideal $\frkp$. 
We denote the order of the residue field $\frko/\frkp$ by $q$. 
The absolute value $\Abs_F=|\cdot|$ on $F$ is normalized via $|\vpi|=q^{-1}$ for any generator $\vpi$ of $\frkp$, where $q$ denotes the order of the residue field $\frko/\frkp$. 

Let $B_n=T_nN_n$ be the Borel subgroup of $\GL_n$, where $T_n$ is the group of diagonal matrices in $\GL_n(F)$ and $N_n$ is the group of upper triangular unipotent matrices in $\GL_n(F)$. 
Let $w_n$ be the longest Weyl element in $\GL_n(F)$. 

Fix an additive character $\addchar:F\to\CC^\times$ which is trivial on $\frko$ and non-trivial on $\frkp^{-1}$. 
We write $\cals(F)$ for the space of locally constant compactly supported functions on $F$. 
The Fourier transform of $\phi\in\cals(F)$ is defined by 
\[\widehat{\phi}(y)=\int_F\phi(x)\addchar(-xy)\,\d x. \]
The measure $\d x$ is chosen so that $\widehat{\widehat{\phi}}(x)=\phi(-x)$. 

Let $\pi$ be an irreducible admissible generic representation of $\GL_{m+1}(F)$. 
We write $\scrw_{\addchar}(\pi)$ for the Whittaker model of $\pi$ with respect to an arbitrarily fixed additive character $\addchar$ of $F$. 
One can define an invariant perfect pairing 
\[\La\;,\;\Ra:\scrw_{\addchar}(\pi)\otimes\scrw_{\addchar^{-1}}(\pi^\vee)\to\CC\] 
by (\ref{tag:32}). 
Given $W\in\scrw_{\addchar}(\pi)$, we define $\widetilde{W}\in\scrw_{\addchar^{-1}}(\pi^\vee)$ by 
\[\widetilde{W}(g)=W(w_{m+1}\trs g^{-1}). \] 

Let $n$ be a positive integer which is equal or less than $m$. 
Put $l=m-n$. 
Let $\sig$ be an irreducible admissible generic representation of $\GL_n(F)$ whose central character is $\ome_\sig$. 
We associate to Whittaker functions $W\in \scrw_{\addchar}(\pi)$ and $W'\in\scrw_{\addchar}(\sig)$ the local zeta integrals
\begin{align*}
Z(s,W,W')&=\int_{N_n\bsl\GL_n(F)}W\biggl(\begin{bmatrix} h & \\ & \ono_{l+1} \end{bmatrix}\biggl)W'(h)|\det h|^{s-\frac{l+1}{2}}\,\d h, \\
\widetilde{Z}(s,\widetilde{W},\widetilde{W'})&=\int_{N_n\bsl\GL_n(F)}\int_{\Mat_{l,n}(F)} \widetilde{W}\left(\begin{bmatrix} h & & \\ x & \ono_l & \\ & & 1 \end{bmatrix}\right)\widetilde{W'}(h)|\det h|^{s-\frac{l+1}{2}}\,\d x\d h, 
\end{align*}
which converge absolutely for $\Re s\gg 0$, where $\d h$ is the Haar measure on $\GL_n(F)$ giving $\GL_n(\frko)$ volume $1$. 

We write $L^\GL(s,\pi\times\sig)$, $\vep^\GL(s,\pi\times\sig,\addchar)$ and $\gam^\GL(s,\pi\times\sig,\addchar)$ for the $L$, epsilon and gamma factors associated to $\pi$ and $\sig$. 
These local factors are studied extensively in \cite{JPSS2}. 
The gamma factor is defined as the proportionality constant of the functional equation 
\beq
Z(1-s,\pi^\vee(w_{m+1,n})\widetilde{W},\widetilde{W'})=\ome_\sig(-1)^m\gam^\GL(s,\pi\times\sig,\addchar)\widetilde{Z}(s,W,W'), \label{tag:a2}
\eeq
where 
\[w_{m+1,n}=\begin{bmatrix} 
 \ono_n & \\ & w_{m-n+1} 
 \end{bmatrix}. \]
 
\begin{remark}\label{rem:a1}
When we view $\pi$ and $\sig$ are representations of unitary groups over the split quadratic algebra $F\oplus F$, 
\[L(s,\pi\times\sig)=L^\GL(s,\pi\times\sig)L^\GL(s,\pi^\vee\times\sig^\vee). \]
When $n=1$ and $\chi$ is a character of $F^\times$, we will write 
\begin{align*}
L(s,\pi\otimes\chi)&=L^\GL(s,\pi\times\chi), \\   
\vep(s,\pi\otimes\chi,\addchar)&=\vep^\GL(s,\pi\times\chi,\addchar), \\
\gam(s,\pi\otimes\chi,\addchar)&=\gam^\GL(s,\pi\times\chi,\addchar). 
\end{align*}
These local factors are studied extensively also in \cite{GJ}. 
\end{remark}


\subsection{The splitting lemma}\label{ssec:a2}
  
Let $\pi$ be an irreducible admissible tempered representation of $\GL_{m+1}(F)$ and $\sig$ that of $\GL_m(F)$. 
We consider the integral
\[J(W_1^{},W_2^{},W_1',W_2')=\int_{\GL_n(F)}\biggl\La\pi\biggl(\begin{bmatrix} h & \\ & 1 \end{bmatrix}\biggl)W_1^{},W_2^{}\biggl\Ra\La \sig(h)W_1',W_2'\Ra'\,\d h,  \]
which is convergent for 
\begin{align*}
W_1&\in\scrw_{\addchar}(\pi^{}), &
W_2&\in\scrw_{\addchar^{-1}}(\pi^\vee), &
W_1'&\in\scrw_{\addchar^{-1}}(\sig^{}), & 
W_2'&\in\scrw_{\addchar}(\sig^\vee). 
\end{align*}

The following result is called a splitting lemma and was proved by Wei Zhang in Proposition 4.10 of \cite{Z} by using the work of Lapid and Mao \cite{LM}. 
It is worth noting that Proposition 4.10 of \cite{Z} uses unnormalized local Haar measures (cf. \S 2.1 of \cite{Z}) while we here use normalized ones. 

\begin{lemma}\label{lem:a1}
Notation being as above, we have 
\[J(W_1^{},W_2^{},W_1',W_2')=Z\biggl(\frac{1}{2},W_1^{},W_1'\biggl)Z\biggl(\frac{1}{2},W_2^{},W_2'\biggl)\prod_{i=1}^{m-1}\zet_F(i). \]
\end{lemma}

\begin{Remark}\label{rem:a2}
\begin{enumerate}
\item\label{rem:a21} Assume that $\pi$ and $\sig$ are unramified and that $\addchar$ is trivial on $\frko$ but non-trivial on $\frkp^{-1}$. 
Let $W_\pi\in\scrw_{\addchar}(\pi)$ be the normalized essential vector defined by $W_\pi(k)=1$ for $k\in\GL_{m+1}(\frko)$. 
Define $W_{\pi^\vee}$, $W_\sig$ and $W_{\sig^\vee}$ similarly.
Then 
\begin{align}
Z(s,W_\pi,W_\sig)&=L^{\GL}\biggl(s+\frac{1}{2},\pi\times\sig\biggl), \notag\\
\La W_\pi,W_{\pi^\vee}\Ra&=\frac{L^\GL(1,\pi^{}\times\pi^\vee)}{\zet_F(m+1)}. \label{tag:a3}
\end{align} 
by \cite[(3.3)]{Z}. 
Lemma \ref{lem:a1} reproves the formula
\[\frac{J(W_\pi,W_{\pi^\vee},W_\sig,W_{\sig^\vee})}{\La W_\pi,W_{\pi^\vee}\Ra\La W_\sig,W_{\sig^\vee}\Ra'}=\frac{L\bigl(\frac{1}{2},\pi\times\sig\bigl)\prod_{i=1}^{m+1}\zet_F(i)}{L^\GL(1,\pi\times\pi^\vee)L^\GL(1,\sig\times\sig^\vee)} \]
(cf. Remark \ref{rem:a1}), which was proved in Theorem 2.12 of \cite{NHarris}. 
\item\label{rem:a22} Proposition 5.1 of \cite{MH} is a triple product analogue of this lemma (cf. Proposition 3.8 of \cite{CH}). 
\end{enumerate}
\end{Remark}


\subsection{Ordinary vectors of representations of $\GL_n(F)$}\label{ssec:33}

For $S\subset F$ we write $\II_S$ for the characteristic function of $S$. 
For a compact subgroup $\Gam$ of $\GL_n(F)$ and a representation $(\pi,V)$ of $\GL_n(F)$ we let 
\[V^\Gam=\{v\in V\;|\;\pi(\gam)v=v\text{ for }\gam\in\Gam\}\] 
be the space of $\Gam$-invariant vectors in $V$. 
Put $\caln_n=N_n\cap\GL_n(\frko)$.

Let $\mu_1,\mu_2,\dots,\mu_n$ be characters of $F^\times$. 
The space $V$ of $\pi=I(\mu_1,\mu_2,\dots,\mu_n)$ consists of functions $h:\GL_n(F)\to\CC$ which satisfy
\begin{align*}
h(tug)&=h(g)\wp_n(t)^{1/2}\prod_{i=1}^n\mu_i(t_i), & 
\wp_n(t)&=\prod_{i=1}^n|t_i|^{n+1-2i}
\end{align*}
for $t=\diag[t_1,t_2,\dots,t_n]\in T_n$, $u\in N_n$ and $g\in\GL_n(F)$. 

We define an ordinary vector $h^\mathrm{ord}\in V$ by 
\[h^{\ord}_\pi(tuw_nv)=\II_{\caln_n}(v)\wp_n(t)^{1/2}\prod_{i=1}^n\mu_i(t_i)\]
for $t=\diag[t_1,t_2,\dots,t_n]\in T_n$ and $u,v\in N_n$, where $\II_{\caln_n}$ denotes the characteristic function of $\caln_n$. 
Define the operator $U_\frkp$ on $V^{\caln_n}$ by  
\[U_\frkp h=\sum_{\begin{smallmatrix} u=(u_{i,j})\in\caln_n\\
u_{i,j}\in\frko/\frkp^{j-i}\text{ for }i<j \end{smallmatrix}}\pi(uD_{n,\vpi})h, \]
where 
\[D_{n,\vpi}=\begin{bmatrix} \vpi^{n-1} & 0 & \dots & 0 & 0 \\ 0 & \vpi^{n-2} & \dots & 0 & 0 \\ \vdots & \vdots & \ddots & \vdots & \vdots \\ 0 & 0 & \dots & \vpi & 0 \\ 0 & 0 & \dots & 0 & 1 \end{bmatrix}. \]

\begin{proposition}\label{prop:31}
\[U_\frkp h^{\ord}_\pi=h^{\ord}_\pi\prod_{i=1}^n\Big(\mu_i(\vpi)q^{i-\frac{n+1}{2}}\Big)^{i-1}. \]
\end{proposition}

\begin{proof}
Let $g\in\GL_n(F)$ be such that $[U_\frkp h^{\ord}_\pi](g)\neq 0$. 
There exists $u\in\caln_n$ such that $h^{\ord}_\pi(guD_{n,\vpi})\neq 0$. 
We have $guD_{n,\vpi}\in B_nw_n\caln_n$. 
Since $D_{n,\vpi}^{}\caln_n^{}D_{n,\vpi}^{-1}\subset\caln_n^{}$, we get $g\in B_nw_n\caln_n$. 
By the characterization of $h^{\ord}_\pi$ we see that $h^{\ord}_\pi$ is an eigenvector of $U_\frkp$ with eigenvalue $[U_\frkp h^{\ord}_\pi](w_n)$. 
By definition we know that $[U_\frkp h^{\ord}_\pi](w_n)$ equals 
\[\sum_{u_{i,j}\in\frko/\frkp^{j-i}}h^{\ord}_\pi(w_nUD_{n,\vpi})=h^{\ord}_\pi(w_nD_{n,\vpi}), \]
from which we obtain the formula for the eigenvalue. 
\end{proof}


\subsection{An inductive property}\label{ssec:34} 

Fix an additive character $\addchar$ of $F$ which is trivial on $\frko$ but non-trivial on $\frkp^{-1}$. 
Define the additive character of $N_n$ by 
\[\addchar(u)=\addchar(u_{1,2}+u_{2,3}+\cdots+u_{n-1,n})\]
for $u=(u_{i,j})\in N_n$.  
For an irreducible admissible generic representation $\sig$ of $\GL_n(F)$ we write $\scrw_{\addchar}(\sig)$ for the Whittaker model of $\sig$ with respect to $\addchar$. 

Let $\pi$ be the irreducible generic constituent of $I(\mu_1,\mu_2,\dots,\mu_n)$. 
For $h\in V$ we define $W_{\addchar}(h)\in\scrw_{\addchar}(\pi)$ by 
\beq
W_{\addchar}(g,h)=\int_{N_n}h(w_nug)\overline{\addchar(u)}\,\d u\label{tag:31}
\eeq 
for $g\in\GL_n(F)$. 
Put 
\beq\label{E:Word}\begin{aligned}
W^{\ord}_\pi&=W_{\addchar}(h^{\ord}_\pi), & 
e_{n-1}&=(0,0,\dots,0,1)\in F^{n-1}. 
\end{aligned}\eeq
Define $\JJ_n(y)=\II_{\frko^n}(\trs y)$ for low vectors $y$. 
Let $\pi'$ the irreducible generic constituent of $I(\mu_2,\mu_3,\dots,\mu_n)$ and $\pi''$ that of $I(\mu_1,\mu_2,\dots,\mu_{n-1})$. 

\begin{lemma}\label{lem:31}
If $n\geq 2$, then for $g\in\GL_{n-1}(F)$ 
\begin{align*}
W^{\ord}_\pi\biggl(\begin{bmatrix} g & \\ & 1 \end{bmatrix}\biggl)&=|\det g|^\frac{1}{2}W^{\ord}_{\pi'}(g)\JJ_{n-1}(e_{n-1}g), \\
W^{\ord}_\pi\biggl(\begin{bmatrix} 1 & \\ & g \end{bmatrix}\biggl)&=|\det g|^{-\frac{1}{2}}W^{\ord}_{\pi''}(g)\JJ_{n-1}(e_{n-1}w_{n-1}\trs g^{-1}).   
\end{align*}
\end{lemma}

\begin{proof}Put 
 \begin{align*}
 w_{(n-1,1)}&=\begin{bmatrix} 
  & 1 \\
  \ono_{n-1} &  \end{bmatrix}, &  
 \bfu(x)&=\begin{bmatrix} 
 \ono_{n-1} & x \\
   & 1
 \end{bmatrix} 
 \end{align*}
 for $x\in F^{n-1}$. 
 The section $h_\pi^{\ord}$ satisfies 
 \[h^{\ord}_\pi\left(\begin{bmatrix} 1 & \\ & g \end{bmatrix}w_{(n-1,1)}\bfu(y)\right)
=|\det g|^{-\frac{1}{2}}h^{\ord}_{\pi'}(g)\II_{\frko^{n-1}}(y) \]
for $g\in\GL_{n-1}(F)$ and $y\in F^{n-1}$. 
Since  
\[w_n\begin{bmatrix} u & \\ & 1 \end{bmatrix}\bfu(y)\begin{bmatrix} g & \\ & 1 \end{bmatrix}=\begin{bmatrix} 1 & \\ & w_{n-1}ug \end{bmatrix}w_{(n-1,1)}\bfu(g^{-1}y) \]
for $u\in N_{n-1}$ and $y\in F^{n-1}$, we get 
\begin{align*}
&W_{\addchar}\biggl(\begin{bmatrix} g & \\ & 1 \end{bmatrix},h^{\ord}_\pi\biggl)\\
=&\int_{F^{n-1}}\int_{N_{n-1}}h^{\ord}_\pi\biggl(w_n\begin{bmatrix} u & \\ & 1 \end{bmatrix}\bfu(y)\begin{bmatrix} g & \\ & 1 \end{bmatrix}\biggl)\overline{\addchar(u)\addchar(y_{n-1})}\,\d y\d u\\
=&\int_{F^{n-1}}\int_{N_{n-1}}\frac{h^{\ord}_{\pi'}(w_{n-1}ug)}{|\det g|^\frac{1}{2}}\II_{\frko^{n-1}}(g^{-1}y)\overline{\addchar(u)\addchar(e_{n-1}y)}\,\d u\d y\\
=&|\det g|^\frac{1}{2}\int_{N_{n-1}}h^{\ord}_{\pi'}(w_{n-1}ug)\overline{\addchar(u)}\,\d u\,\int_{F^{n-1}}\II_{\frko^{n-1}}(y)\overline{\addchar(e_{n-1}g y)}\,\d y\\
=&|\det g|^\frac{1}{2}W_{\addchar}\left(g,h^{\ord}_{\pi'}\right)\II_{\frko^{n-1}}(\trs g\trs e_{n-1})
\end{align*}
as claimed. 
The second formula can be proved in the same way. 
\end{proof}


\subsection{The pairing}\label{ssec:35} 
Let $\sig$ be an irreducible admissible generic unitary representation of $\GL_n(F)$. 
One can define an invariant perfect pairing 
\[\La\;,\;\Ra:\scrw_{\addchar}(\sig)\otimes\scrw_{\addchar^{-1}}(\sig^\vee)\to\CC\] 
by 
\beq
\La W_1,W_2\Ra=\int_{N_{n-1}\bsl\GL_{n-1}(F)}W_1\biggl(\begin{bmatrix} g & \\ & 1 \end{bmatrix}\biggl)W_2\biggl(\begin{bmatrix} g & \\ & 1 \end{bmatrix}\biggl)\,\d g. \label{tag:32}
\eeq
where $\d g$ be the Haar measure on $\GL_{n-1}(F)$ giving $\GL_{n-1}(\frko)$ volume $1$. 
Given $W\in\scrw_{\addchar}(\sig)$, we define $\widetilde{W}\in\scrw_{\addchar^{-1}}(\sig^\vee)$ by $\widetilde{W}(g)=W(w_n\trs g^{-1})$. 
Let $\pi$ be an irreducible generic unitary constituent of $I(\mu_1,\mu_2,\dots,\mu_n)$. 
Put 
\begin{align*}
\cali_0^{(n)}(\frkp^\ell)&=\{(g_{ij})\in\GL_n(\frko)\;|\;g_{ij}\in \frkp^{\ell(i-j)}\text{ for }i>j\}, \\
\calb_{\pi^{\ord}}^{[\ell]}&=\frac{\zet_F(n)}{L^\GL(1,\pi^{}\times\pi^\vee)}\La W_\pi^{\ord},\pi^\vee(D_{n,\vpi}^{-\ell})\widetilde{W_\pi^\ord}\Ra. 
\end{align*}

The following formula generalizes Lemma 2.8 of \cite{MH}. 

\begin{proposition}\label{prop:32}
If $\pi$ is unitary and $\ell$ is sufficiently large, then 
\[\calb_{\pi^{\ord}}^{[\ell]}=\frac{\zet_F(1)^n}{L^\GL(1,\pi^{}\times\pi^\vee)}\cdot \frac{\alp_\pi^\ell}{[\GL_n(\frko):I^{(n)}_0(\frkp^\ell)]}\prod_{i<j}\frac{\mu_i(-1)}{\gam(1,\mu_i^{-1}\mu_j^{},\addchar)}. \]
\end{proposition}

The proof of Proposition \ref{prop:32} consists of several steps. 
Let $B_n^-=w_n^{}B_n^{}w_n^{-1}$ be the group of lower triangular matrices in $\GL_n(F)$. 
We denote the unipotent radical of $B_n^-$ by $N_n^-$. 
Since there is nothing to prove if $n=1$, we suppose that $n\geq 2$. 
Put 
\begin{align*}
m&=n-1, & 
\Del_m&=\vpi^{-\ell} D_{m,\vpi}^{-\ell}, &
\gam_m&=\prod_{i=1}^m\frac{\zet_F(i)}{\zet_F(1)}. 
\end{align*}

\begin{lemma}\label{lem:32}
Let $W\in\scrw_{\addchar}(\sig)$ and $b\in B_m^-$. 
If $\ell$ is sufficiently large, then $W\biggl(\begin{bmatrix} b\Del_m & \\ & 1 \end{bmatrix}\biggl)=0$ unless $\!\trs(e_mb)\in\frko^m$. 
\end{lemma}

\begin{proof}
We write 
\[b=\begin{bmatrix} g & \\ \trs y & t \end{bmatrix}=\begin{bmatrix} g & \\ & t \end{bmatrix}\begin{bmatrix} \ono_{m-1} & \\ t^{-1}\trs y & 1 \end{bmatrix}. \]
Notice that $e_mb=(\trs y, t)$. 
Since  
\[\left[\begin{array}{cc|c} \ono_{m-1} & & \\ t^{-1}\trs y & 1 & \\ \hline & & 1 \end{array}\right]
\left[\begin{array}{c|c} \ono_m & \begin{matrix} x \\ z \end{matrix} \\ \hline & 1 \end{array}\right]
=\left[\begin{array}{c|c} \ono_m & \begin{matrix} x \\ z+\frac{\trs yx}{t} \end{matrix} \\ \hline & 1 \end{array}\right]
\left[\begin{array}{cc|c} \ono_{m-1} & & \\ t^{-1}\trs y & 1 & \\ \hline & & 1 \end{array}\right] \]
for $x,y\in F^{m-1}$ and $z\in F$, we have 
\[W\left(\begin{bmatrix} b\Del_m & \\ & 1 \end{bmatrix}\left[\begin{array}{c|c} \ono_m & \Del_m^{-1}\begin{pmatrix} x \\ z \end{pmatrix} \\ \hline & 1 \end{array}\right]\right)=\addchar(tz+\trs yx)W\biggl(\begin{bmatrix} b\Del_m & \\ & 1 \end{bmatrix}\biggl). \]
If $\ell$ is sufficiently large, then the left hand side is $W\biggl(\begin{bmatrix} b\Del_m & \\ & 1 \end{bmatrix}\biggl)$ for all $x\in\frko^{m-1}$ and $z\in\frko$, which implies that $\trs y\in\frko^{m-1}$ and $t\in\frko$. 
\end{proof}

\begin{lemma}\label{lem:33}
If $\ell$ is sufficiently large, then $B_{\pi^{\ord}}^{[\ell]}=\gam_m\displaystyle\lim_{s\to 0}\scrf_\pi(s)$, where 
\[\scrf_\pi(s)=\int_{N_m^-}\int_{T_m^{}}\widetilde{W_\pi^\ord}\biggl(\begin{bmatrix} vt\Del_m & \\ & 1 \end{bmatrix}\biggl)\prod_{i=1}^m\mu_{n+1-i}(t_i)|t_i|^{s+1+\frac{m}{2}-i}\,\d t\d v. \] 
\end{lemma}

\begin{proof}
Put $W=\widetilde{W_\pi^\ord}$. 
We define the function $\scrb_\pi(s)$ by the integral  
\[\scrb_\pi(s)=\int_{N_m\bsl\GL_m(F)}|\det g|^sW^{\ord}_\pi\biggl(\begin{bmatrix} g & \\ & 1 \end{bmatrix}\biggl)W\biggl(\begin{bmatrix} g\Del_m & \\ & 1 \end{bmatrix}\biggl)\,\d g \]
for $\Re s\gg 0$. 
Then $B_{\pi^{\ord}}^{[\ell]}=\displaystyle\lim_{s\to 0}\scrb_\pi(s)$. 
We write 
\begin{align*}
\scrb_\pi(s)&=\int_{N_m\bsl\GL_m(F)}|\det g|^{s+\frac{1}{2}}W^{\ord}_{\pi'}(g)\JJ_m(e_mg)W\biggl(\begin{bmatrix} g\Del_m & \\ & 1 \end{bmatrix}\!\biggl)\,\d g\\
&=\int_{\GL_m(F)}|\det g|^{s+\frac{1}{2}}h^{\ord}_{\pi'}(w_mg)\JJ_m(e_mg)W\biggl(\begin{bmatrix} g\Del_m & \\ & 1 \end{bmatrix}\!\biggl)\,\d g, 
\end{align*}
using Lemma \ref{lem:31} and substituting the integral expression (\ref{tag:31}) of $W^{\ord}_{\pi'}$. 
 
Since $[\GL_m(\frko):\cali_m]=q^{m(m-1)/2}\gam_m^{-1}$, where $\cali_m$ is the Iwahori subgroup of $\GL_m(\frko)$, we see the following integral formula from (2) on p.~240 of \cite{W1}:
\beq
\int_{\GL_m(F)}\calf(g)\,\d g=\gam_m\int_{N_m^-}\int_{T_m^{}}\int_{N_m^{}}\calf(vtu)\wp_m(t)\,\d u\d t\d v \label{tag:33}
\eeq
for an integrable function $\calf$ on $\GL_m(F)$. 
It follows that  
\[\frac{\scrb_\pi(s)}{\gam_m}=\int_{N_m^-}\int_{T_m^{}}|\det t|^{s+\frac{1}{2}}h^{\ord}_{\pi'}(w_mt)\JJ_m(e_mvt)W\biggl(\begin{bmatrix} vt\Del_m & \\ & 1 \end{bmatrix}\!\biggl)\wp_m(t)\d t\d v. \]
The right hand side coincides with $\scrf(s)$ by Lemma \ref{lem:32}. 
\end{proof}

\begin{lemma}\label{lem:34}
If $n\geq 2$, then 
\[\scrf_\pi(s)=q^{-(n-1)\ell\bigl(s+\frac{n}{4}\bigl)}\mu_n(\vpi^\ell)^{n-1}\zet_F(s+1)\scrf_{\pi''}(s)\prod_{i=1}^{n-1}\frac{\mu_i(-1)}{\gam(s+1,\mu_i^{-1}\mu_n^{},\addchar)}. \]
\end{lemma}

\begin{proof}
By the definition of the JPSS integral we have  
\begin{align*}
\scrf_\pi(s)
&=\int_{N_{m-1}^-}\int_{T_{m-1}^{}}\widetilde{Z}(s+1,W_{v't'},\mu_n)\prod_{i=2}^m\mu_{n+1-i}(t_i)|t_i|^{s+1+\frac{m}{2}-i}\,\d t'\d v',  
\end{align*}
where $t'=\diag[t_2,\dots,t_n]$, and $W_{b'}\in\scrw_{\addchar^{-1}}(\pi^\vee)$ is defined by 
\[W_{b'}(g)=\widetilde{W_\pi^\ord}\left(g\left[\begin{array}{cc|c} 1 & & \\ & b' & \\ \hline & & 1 \end{array}\right]D_{n,\vpi}^{-\ell}\right) \]
for $b'\in B_{m-1}^-$. 
The functional equation (\ref{tag:a2}) gives 
\[\widetilde{Z}(s+1,W_{b'},\mu_n^{})=\mu_n(-1)^m\frac{Z(-s,\pi(w_{n,1})\widetilde{W_{b'}},\mu_n^{-1})}{\gam(s+1,\pi^\vee\otimes\mu_n^{},\addchar^{-1})}. \]
Since 
\begin{align*}
&\pi(w_{n,1})\widetilde{W_{b'}}\biggl(\begin{bmatrix} a & \\ & \ono_m \end{bmatrix}\biggl)\\
=&W_{b'}\biggl(w_n\begin{bmatrix} a^{-1} & \\ & w_m \end{bmatrix}\biggl)\\
=&\widetilde{W_\pi^\ord}\left(w_n\begin{bmatrix} a^{-1} & \\ & w_m \end{bmatrix}\left[\begin{array}{c|cc} 1 & & \\ \hline & b' & \\ & & 1 \end{array}\right]D_{n,\vpi}^{-\ell}\right)\\
=&W^{\ord}_\pi\left(\begin{bmatrix} a & \\ & w_m \end{bmatrix}\left[\begin{array}{c|cc} 1 & & \\ \hline & \trs b^{\prime-1} & \\ & & 1 \end{array}\right]D_{n,\vpi}^\ell\right)\\
=&W^{\ord}_\pi\left(\left[\begin{array}{c|c} a\vpi^{m\ell} & \\ \hline & w_m\trs\begin{bmatrix} b' & \\ & 1 \end{bmatrix}^{-1}D_{m,\vpi}^\ell \end{array}\right]\right)\\
=&\mu_n(a\vpi^{m\ell})\biggl|\frac{(a\vpi^{m\ell})^m\det b'}{\det D_{m,\vpi}^\ell}\biggl|^\frac{1}{2}\widetilde{W^{\ord}_{\pi''}}\biggl(\begin{bmatrix} b'\Del_{m-1} & \\ & 1 \end{bmatrix}\!\biggl)\II_\frko(a\vpi^{m\ell})
\end{align*}
by Lemma \ref{lem:31}, we have 
\begin{align*}
&Z(-s,\pi(w_{n,1})\widetilde{W_{b'}},\mu_n^{-1})\\
=&\int_{F^\times}\mu_n(\vpi^{m\ell})\biggl|\frac{(a\vpi^{m\ell})^m\det b'}{\det D_{m,\vpi}^\ell}\biggl|^\frac{1}{2}\widetilde{W^{\ord}_{\pi''}}\biggl(\begin{bmatrix} b'\Del_{m-1} & \\ & 1 \end{bmatrix}\biggl)\II_\frko(a\vpi^{m\ell})|a|^{-s-\frac{m}{2}}\,\d a \\
=&\mu_n(\vpi^{m\ell})q^{-m\ell\bigl(s+\frac{m+1}{4}\bigl)}|\det b'|^\frac{1}{2}\widetilde{W^{\ord}_{\pi''}}\biggl(\begin{bmatrix} b'\Del_{m-1} & \\ & 1 \end{bmatrix}\biggl)\zet_F(-s). 
\end{align*} 
We therefore get 
\[\widetilde{Z}(s+1,W_{b'},\mu_n^{})=\frac{\mu_n(-\ome^\ell)^m}{q^{m\ell\bigl(s+\frac{m+1}{4}\bigl)}}|\det b'|^\frac{1}{2}\frac{\widetilde{W^{\ord}_{\pi''}}\biggl(\begin{bmatrix} b'\Del_{m-1} & \\ & 1 \end{bmatrix}\biggl)}{\prod_{i=1}^m\gam(s+1,\mu_i^{-1}\mu_n^{},\addchar^{-1})}\zet_F(s+1) \]
by the multiplicativity of the gamma factor. 
Substituting this expression, we get the inductive formula for $\scrf_\pi(s)$. 
\end{proof}

We are now ready to prove Proposition \ref{prop:32}. 
We have 
\[\scrf_\pi(s)=\frac{\zet_F(s+1)^{n-1}}{q^{\frac{n(n-1)\ell}{2}\bigl(s+\frac{n+1}{3}\bigl)}}\alp_\pi^\ell\prod_{i<j}\frac{\mu_i(-1)}{\gam(s+1,\mu_i^{-1}\mu_j^{},\addchar)}, \]
applying Lemma \ref{lem:34} inductively. 
Since 
\[[\GL_n(\frko):I^{(n)}_0(\frkp^\ell)]=q^{\frac{n(n^2-1)}{6}\ell}\gam_n^{-1}, \]
we immediately deduce the declared formula from Lemma \ref{lem:33}. \qed 


\subsection{The JPSS zeta integral}\label{ssec:37}

Let $\pi$ be an irreducible unitary generic constituent of the principal series $I(\nu,\rho,\mu)$ and $\sig$ that of  $I(\mu',\nu')$. 
The matrix $\vsi$ is defined in (\ref{tag:23}). 

\begin{proposition}\label{prop:33}
Put 
\begin{align*}
&W^{\ord}_\pi=W_{\addchar}(h^{\ord}_\pi), & 
&W^{\ord}_\sig=W_{\addchar^{-1}}(h^{\ord}_\sig), &
\bft_\ell&=\begin{bmatrix}
0 & 0 & -\vpi^{-\ell} \\
0 & \vpi^\ell & 1 \\
\vpi^{2\ell} & \vpi^\ell & 0
\end{bmatrix}.
\end{align*}
If $\ell$ is sufficiently large, then  
\[Z\left(\frac{1}{2},\pi(\vsi\bft_\ell)W^{\ord}_\pi,W^{\ord}_\sig\right)
=\frac{\pm\zet_F(2)(q^{-5/2}\rho(\vpi)\mu(\vpi)^2\nu'(\vpi))^\ell}{\zet_F(1)\gam\bigl(\frac{1}{2},\mu^{}\nu',\addchar\bigl)\gam\bigl(\frac{1}{2},\rho^{}\nu',\addchar\bigl)\gam\bigl(\frac{1}{2},\mu\mu',\addchar\bigl)}. \]
\end{proposition}

\begin{lemma}
Put 
\[W_b=\pi\left(\left[\begin{array}{c|cc} 1 & & \\ \hline  & 1 & \\ & & b \end{array}\right]\bft_\ell\right)W^{\ord}_\pi. \]
Then 
\[Z(s,\pi(\vsi)W,W^{\ord}_\sig)
=\frac{\frac{\zet_F(2)}{\zet_F(1)}}{\gam(s,\pi^{}\otimes\nu',\addchar)}\int_{F^\times}Z(1-s,\widetilde{W}_b,\nu^{\prime-1})\mu'(b)|b|^{s-1}\,\d b. \]
\end{lemma}

\begin{proof}
Put $W=\pi(\bft_\ell)W^{\ord}_\pi$. 
We get 
\[Z(s,\pi(\vsi)W,W^{\ord}_\sig)=\int_{\GL_2(F)}W\biggl(\left[\begin{array}{c|c} g & \\ \hline  & 1 \end{array}\right]\vsi\biggl)h^{\ord}_\sig\biggl(\begin{bmatrix} 0 & 1 \\ 1 & 0 \end{bmatrix}g\biggl)|\det g|^{s-\frac{1}{2}}\,\d g, \]
substituting the integral expression (\ref{tag:31}) of $W^{\ord}_\sig$. 
Put $\bfu(y)=\begin{bmatrix} 1 & y \\ 0 & 1 \end{bmatrix}$ for $y\in F$. 
Then $\frac{\zet_F(1)}{\zet_F(2)}Z(s,\pi(\vsi)W,W^{\ord}_\sig)$ equals 
\[\int_{F^{\times 2}\times F^2}
W\left(\begin{bmatrix} \begin{bmatrix} a & \\ x & b \end{bmatrix}\bfu(y) & \\ & 1 \end{bmatrix}\vsi\right)\biggl|\frac{b}{a}\biggl|^{1/2}\mu'(b)\nu'(a)\II_\frko(y)|ab|^{s-\frac{1}{2}}\,\d y\frac{\d x\d a\d b}{|b|} \]
by the integration formula (\ref{tag:33}). 
Since $\vsi\begin{bmatrix} \bfu(y) & \\ & 1\end{bmatrix}\vsi=\begin{bmatrix} 1 & 0 & y \\ 0 & 1 & 0 \\ 0 & 0 & 1\end{bmatrix}$, we get 
\begin{align*}
&\frac{\zet_F(1)}{\zet_F(2)}Z(s,\pi(\vsi)W,W^{\ord}_\sig)\\
=&\int_{F^{\times 2}\times F}
W\left(\left[\begin{array}{c|cc}  a & & \\ \hline  x & b & \\ & & 1 \end{array}\right]\vsi\right)\mu'(b)\nu'(a)|ab|^{s-1}\,\d x\d a\d b\\
=&\int_{F^\times}\widetilde{Z}(s,\pi(\vsi)W_b,\nu')\mu'(b)|b|^{s-1}\d b, 
\end{align*}  
where 
\[W_b=\pi\left(\vsi\left[\begin{array}{c|cc} 1 & & \\ \hline  & b & \\ & & 1 \end{array}\right]\vsi\right)W
=\pi\left(\left[\begin{array}{c|cc} 1 & & \\ \hline  & 1 & \\ & & b \end{array}\right]\bft_\ell\right)W^{\ord}_\pi. \]

The stated formula follows from the functional equation (\ref{tag:a2})  
\[\widetilde{Z}(s,\pi(\vsi)W_b,\nu')
=\gam(s,\pi^{}\otimes\nu',\addchar)^{-1}Z(1-s,\widetilde{W}_b,\nu^{\prime-1}). \]
We here follow the convention in Remark \ref{rem:a1}. 
\end{proof}

Now we need the following formula: 

\begin{lemma}\label{lem:36}
Put $\Phi=\II_{1+\frkp^\ell}$. 
Then for $a,b\in F^\times$ 
\[\widetilde{W}_b\left(\left[\begin{array}{c|c} a & \\ \hline  & \ono_2 \end{array}\right]\right)
=\rho(\vpi^\ell)\mu(b\vpi^{2\ell})|b|\widehat{\Phi}(-b)\frac{|a\vpi^\ell|\widehat{\Phi}(a\vpi^\ell)}{\nu(-a\vpi^\ell)}. \]
\end{lemma}

\begin{proof}
Observe that 
\[\widetilde{W}_b(g)
=W_b(w_3\trs g^{-1})
=W^{\ord}_\pi\left(w_3\trs g^{-1}\left[\begin{array}{c|cc} 1 & & \\ \hline  & 1 & \\ & & b \end{array}\right]\bft_\ell\right). \]
It follows that 
\begin{align*}
\widetilde{W}_b\left(\left[\begin{array}{c|c} a & \\ \hline  & \ono_2 \end{array}\right]\right)
=&W^{\ord}_\pi\left(w_3\left[\begin{array}{c|cc} a^{-1} & & \\ \hline  & 1 & \\ & & b \end{array}\right]
\begin{bmatrix}
0 & 0 & -\vpi^{-\ell} \\
0 & \vpi^\ell & 1 \\
\vpi^{2\ell} & \vpi^\ell & 0
\end{bmatrix}\right)\\
=&W^{\ord}_\pi\left(\begin{bmatrix} b\vpi^{2\ell} & 0 & 0 \\ 0 & \vpi^\ell & 0 \\ 0 & 0 & -(a\vpi^\ell)^{-1} \end{bmatrix}
\begin{bmatrix} 1 & \vpi^{-\ell} & 0 \\
0 & 1 & \vpi^{-\ell}\\
0 & 0 & 1\end{bmatrix}\right). 
\end{align*}
Since 
\[\begin{bmatrix}
1 & y & z\\
0 & 1 & x\\
0 & 0 & 1
\end{bmatrix}
\begin{bmatrix} b\vpi^{2\ell} & 0 & 0 \\ 0 & \vpi^\ell & 0 \\ 0 & 0 & \frac{-1}{a\vpi^\ell} \end{bmatrix}
=\begin{bmatrix} b\vpi^{2\ell} & 0 & 0 \\ 0 & \vpi^\ell & 0 \\ 0 & 0 & \frac{-1}{a\vpi^\ell} \end{bmatrix}
\begin{bmatrix}
1 & \frac{y}{b\vpi^\ell} & -\frac{z}{ab\vpi^{3\ell}} \\
0 & 1 & -\frac{x}{a\vpi^{2\ell}}\\
0 & 0 & 1
\end{bmatrix}, \]
we have  
\begin{align*}
&\widetilde{W}_b\left(\left[\begin{array}{c|c} a & \\ \hline  & \ono_2 \end{array}\right]\right)\\
=&\int_{F^3}h^{\ord}_\pi\left(w_3\begin{bmatrix} b\vpi^{2\ell} & 0 & 0 \\ 0 & \vpi^\ell & 0 \\ 0 & 0 & \frac{-1}{a\vpi^\ell} \end{bmatrix}
\begin{bmatrix}
1 & \frac{y+b}{b\vpi^\ell} & \frac{a\vpi^\ell y-z}{ab\vpi^{3\ell}} \\
0 & 1 & \frac{a\vpi^\ell-x}{a\vpi^{2\ell}} \\
0 & 0 & 1
\end{bmatrix}\right)\overline{\addchar(x+y)}\,\d x\d y\d z\\ 
=&\frac{\rho(\vpi^\ell)\mu(b\vpi^{2\ell})}{\nu(-a\vpi^\ell)|ab\vpi^{3\ell}|}
\int_{F^3}h^{\ord}_\pi\left(w_3\begin{bmatrix} 1 &\frac{y+1}{\vpi^\ell} & z \\ 0 & 1 & \frac{1-x}{\vpi^\ell}\\ 0 & 0 & 1 \end{bmatrix}\right)\frac{|a^2b^2\vpi^{4\ell}|}{\addchar(a\vpi^\ell x+by)}\,\d x\d y\d z \\
=&\frac{\rho(\vpi^\ell)\mu(b\vpi^{2\ell})}{\nu(-a\vpi^\ell)}|ab\vpi^\ell|
\int_{F^2}\Phi(y)\Phi(x)\overline{\addchar(a\vpi^\ell x-by)}\,\d x\d y
\end{align*}
from which one can complete the proof of Lemma \ref{lem:36}. 
\end{proof}

We are now ready to prove Proposition \ref{prop:33}. 
For $\phi\in\cals(F)$ and a character $\chi$ of $F^\times$ we define Tate's local integral by 
\[Z(s,\phi,\chi)=\int_{F^\times}\phi(a)\chi(a)|a|^s\,\d a. \] 
Substituting this expression in Lemma \ref{lem:36}, we get
\begin{align*}
Z(1-s,\widetilde{W}_b,\nu^{\prime-1})
&=\rho(\vpi^\ell)\mu(b\vpi^{2\ell})|b|\widehat{\Phi}(-b)\int_{F^\times}\frac{|a\vpi^\ell|\widehat{\Phi}(a\vpi^\ell)}{\nu(-a\vpi^\ell)}\nu'(a)^{-1}\frac{\d a}{|a|^s}\\
&=\rho(\vpi^\ell)\mu(b\vpi^{2\ell})|b|\widehat{\Phi}(-b)\frac{\nu'(\vpi^\ell)}{\nu(-1)q^{\ell s}}Z(1-s,\widehat{\Phi},(\nu\nu')^{-1}). \end{align*}
We conclude that $\frac{\zet_F(1)}{\zet_F(2)}\gam(s,\pi^{}\otimes\nu',\addchar)Z(s,\pi(\vsi)W,W^{\ord}_\sig)$ equals 
\begin{align*}
&Z(1-s,\widehat{\Phi},(\nu\nu')^{-1})\rho(\vpi^\ell)\frac{\nu'(\vpi^\ell)}{\nu(-1)q^{\ell s}}\int_{F^\times}\mu(b\vpi^{2\ell})|b|\widehat{\Phi}(-b)\mu'(b)|b|^{s-1}\,\d b\\
=&Z(1-s,\widehat{\Phi},(\nu\nu')^{-1})\rho(\vpi^\ell)(\nu\mu\mu')(-1)\nu'(\vpi^\ell)q^{-\ell s}\mu(\vpi^{2\ell})Z(s,\widehat{\Phi},\mu\mu'). 
\end{align*}
If $\ell$ is sufficiently large, then by the functional equation 
\begin{align*}
Z(1-s,\widehat{\Phi},(\nu\nu')^{-1})&=\gam(1-s,(\nu\nu')^{-1},\addchar)^{-1}(\nu\nu')(-1)q^{-\ell}, \\ 
Z(s,\widehat{\Phi},(\mu\mu')^{-1})&=\gam(s,(\mu\mu')^{-1},\addchar)^{-1}(\mu\mu')(-1)q^{-\ell}. 
\end{align*}
We conclude that
\[Z\left(\frac{1}{2},\pi(\vsi\bft_\ell)W^{\ord}_\pi,W^{\ord}_\sig\right)
=\frac{\nu'(-1)\frac{\zet_F(2)}{q^{5\ell/2}\zet_F(1)}\rho(\vpi)^\ell\mu(\vpi)^{2\ell}\nu'(\vpi)^\ell}{\gam\bigl(\frac{1}{2},\pi^{}\otimes\nu',\addchar\bigl)\gam\bigl(\frac{1}{2},(\nu\nu')^{-1},\addchar\bigl)\gam\bigl(\frac{1}{2},\mu\mu',\addchar\bigl)}. \]
One can deduce the stated formula from the multiplicativity of the gamma factor and the functional equation $\gam\bigl(\frac{1}{2},\nu\nu',\addchar\bigl)\gam\bigl(\frac{1}{2},(\nu\nu')^{-1},\addchar\bigl)=(\nu\nu')(-1)$. \qed 


\section{Ramified computations: the split case}\label{sec:4}

\subsection{Essential vectors}\label{ssec:41}

We choose a non-trivial additive character $\addchar$ of $F$ so that the maximal fractional ideal on which it is trivial is $\frko$.
Let $\pi$ be an irreducible admissible generic representation of $\GL_{n+1}(F)$. 
Given an open compact subgroup $\Gam$ of $\GL_{m+1}(F)$ and its character $\calx:\Gam\to\CC^\times$, we put 
\[\scrw_{\addchar}(\pi,\Gam,\calx)=\{W\in \scrw_{\addchar}(\pi)\;|\;\pi(\gam)W=\calx(\gam)W\text{ for }\gam\in\Gam\}. \]

Assume that $n\geq 1$. 
For a positive integer $\ell$ the subgroup $\calk_0^{(m+1)}(\frkp^\ell)$ consists of matrices of the form 
\begin{align*}
&\begin{bmatrix} A & B \\ C & d\end{bmatrix} & 
(A&\in\GL_m(\frko),\;B\in\frko^m,\;\trs C\in(\frkp^\ell)^m,\; d\in\frko^\times). 
\end{align*}
When $\ell=0$, we set $\calk_0^{(m+1)}(\frkp^\ell)=\GL_{m+1}(\frko)$. 
Given a character $\ome$ of $\frko^\times$, we define the characters $\ome^\downarrow:\calk_0^{(m+1)}(\frkp^\ell)\to\CC^\times$ and $\ome^\uparrow:\calk_0^{(2)}(\frkp^\ell)\to\CC^\times$ by 
\begin{align*}
\ome^\downarrow\left(\begin{bmatrix} A & B \\ C & d\end{bmatrix}\right)&=\ome(d), &
\ome^\uparrow\left(\begin{bmatrix} a & b \\ c & d\end{bmatrix}\right)&=\ome(a). 
\end{align*}

We write $\ome_\pi$ for the central character of $\pi$. 
Let $c(\pi)$ denote the exponent of the conductor of $\pi$, i.e., the epsilon factor of $\pi$ is of the form  
\beq
\vep\left(s+\frac{1}{2},\pi,\addchar\right)=q^{-c(\pi)s}\vep\left(\frac{1}{2},\pi,\addchar\right). \label{tag:41}
\eeq
Th\'{e}orem\`{e} on p.~211 of \cite{JPSS} says that 
\[\dim \scrw_{\addchar}\bigl(\pi,\calk_0^{(m+1)}\bigl(\frkp^{c(\pi)}\bigl),\ome^\downarrow_\pi\bigl)=1. \] 
Theorem 3.1 of \cite{NM} enables us to normalize a basis vector of this one-dimensional space in the following way:  

\begin{definition}[essential vectors]\label{def:41}
There exists a unique vector 
\[W_\pi\in\scrw_{\addchar}\bigl(\pi,\calk_0^{(m+1)}\bigl(\frkp^{c(\pi)}\bigl),\ome^\downarrow_\pi\bigl)\] 
which satisfies $W_\pi(\ono_{m+1})=1$. 
This vector $W_\pi$ is called a normalized essential Whittaker vector of $\pi$ with respect to $\addchar$. 
\end{definition}


\subsection{The Atkin-Lehner operator}\label{ssec:42}

\begin{proposition}\label{prop:41}
Let $\pi$ be an irreducible admissible generic representation of $\GL_{m+1}(F)$. 
Put $\ell=c(\pi)$ and $\xi_{m,\ell}=\begin{bmatrix} \vpi^{-\ell}\ono_m & 0 \\ 0 & 1 \end{bmatrix}$. 
Let $W_{\pi^\vee}$ be the essential vector of $\pi$ with respect to $\addchar^{-1}$. 
Then $\pi^\vee(\xi_{m,\ell})\widetilde{W}_\pi=\vep\bigl(\frac{1}{2},\pi,\addchar\bigl)^mW_{\pi^\vee}$. 
\end{proposition}

\begin{proof}
One can immediately see that
\[\pi^\vee(\xi_{m,\ell})\widetilde{W}_\pi\in\scrw_{\addchar^{-1}}(\pi^\vee,\calk_0^{(m+1)}(\frkp^\ell),\ome_{\pi^\vee}^\downarrow)\]
from 
\[\trs\left(\xi_{m,\ell}^{-1}\begin{bmatrix} A & B \\ C & d \end{bmatrix}\xi_{m,\ell}^{}\right)=\begin{bmatrix} \trs A & \frac{\trs C}{\vpi^\ell} \\ \trs B\vpi^\ell & d \end{bmatrix}. \]
Thus $\pi^\vee(\xi_{m,\ell})\widetilde{W}_\pi=cW_{\pi^\vee}$ with $c\in\CC^\times$. 

To determine $c$, we take an irreducible unramified principal series $\sig=\Abs_F^{s_1}\times\cdots\times\Abs_F^{s_m}$ of $\GL_m(F)$. 
Recall the functional equation 
\begin{align*}
&\gam^\GL(s,\pi\times\sig,\addchar)\int_{N_m\bsl\GL_m(F)}W\left(\begin{bmatrix} g & 0 \\ 0 & 1 \end{bmatrix}\right)W_\sig(g)|\det g|^{s-\frac{1}{2}}\d g\\
=&\int_{N_m\bsl\GL_m(F)}\widetilde{W}\left(\begin{bmatrix} g & 0 \\ 0 & 1 \end{bmatrix}\right)\widetilde{W}_\sig(g)|\det g|^{\frac{1}{2}-s}\d g\\
=&\ome_\sig(\vpi)^\ell q^{\ell m(1-2s)/2}\int_{N_m\bsl\GL_m(F)}\pi^\vee(\xi_{m,\ell})\widetilde{W}\left(\begin{bmatrix} g & 0 \\ 0 & 1 \end{bmatrix}\right)\widetilde{W}_\sig(g)|\det g|^{\frac{1}{2}-s}\,\d g
\end{align*}
for every $W\in\scrw_{\addchar}(\pi)$. 
Letting $W=W_\pi$, we get  
\[\gam^\GL(s,\pi\times\sig,\addchar)L(s,\pi\times\sig)=\ome_\sig(\vpi)^\ell q^{\ell m(1-2s)/2}cL(1-s,\pi^\vee\times\sig^\vee)\]
by (\ref{tag:24}). 
Since 
\[\vep^\GL\biggl(\frac{1}{2},\pi\times\sig,\addchar\biggl)=\ome_\sig(\vpi)^\ell\vep\biggl(\frac{1}{2},\pi,\addchar\biggl)^m, \]
we obtain the relation by (\ref{tag:41}). 
\end{proof}

Let $m=3$. 
Thus $\pi$ is an irreducible admissible generic representation of $\GL_3(F)$ and $\sig$ that of $\GL_2(F)$. 
Put 
\begin{align*}
\ell&=c(\pi), & 
n&=c(\sig), & 
\tau_\ell^{}&=w_3^{}\xi_{2,\ell}^{-1}, & 
\tau'_n&=w_2^{}\xi_{1,n}^{-1}. 
\end{align*}
Let $W_\sig$ (resp. $W_{\sig^\vee}$) be the essential vector of $\sig^\vee$ (resp. $\sig$) with respect to $\addchar$ (resp. $\addchar^{-1}$) defined in Definition \ref{def:41}. 
We rewrite Proposition \ref{prop:41} in the following manner. 

\begin{corollary}\label{cor:41}
Notation being as above, we have 
\begin{align*}
W_\pi(\tau_\ell^{})&=\vep\biggl(\frac{1}{2},\pi,\addchar\biggl)^2, &
W_\sig(\tau'_n)&=\vep\biggl(\frac{1}{2},\sig,\addchar\biggl).  
\end{align*}
\end{corollary}


\subsection{Computation of the pairing}\label{ssec:43}

Let $\pi$ be an irreducible admissible generic unitary representation of $\GL_n(F)$. 
We write $\pi_u$ for the unramified component of the first nonzero spherical Bernstein-Zelevinsky derivative $\pi^{(n-r)}$ of $\pi$ (see Definition 1.3 of \cite{NM} for the precise definition). 

\begin{proposition}\label{prop:42}
Let $W_\pi$ be the essential vector of $\pi$ with respect to $\addchar$ and $W_{\pi^\vee}$ the essential vector of $\pi^\vee$ with respect to $\addchar^{-1}$. 
If $\pi$ is unitary, then 
\[\calb_\pi:=\frac{\zet_F(n)}{L^\GL(1,\pi^{}\times\pi^\vee)}\La W_\pi,W_{\pi^\vee}\Ra
=\begin{cases}
1 &\text{if $\pi$ is unramified, }\\ 
\frac{\zet_F(n) L^\GL(1,\pi_u^{}\times\pi_u^\vee) }{L^\GL(1,\pi^{}\times\pi^\vee)} &\text{if $\pi$ is ramified.}
\end{cases} \]
\end{proposition}

\begin{proof}
We may assume $\pi$ to be ramified in view of Remark \ref{rem:a1}. 
Put $m=n-1$. 
Let $t=\diag[t_1,t_2,\dots,t_m]\in T_m$. 
Put $t'=\diag[t_1,t_2,\dots,t_r]\in T_r$. 
Corollary 3.2 of \cite{NM} gives 
\[W_\pi\biggl(\begin{bmatrix} t & \\ & 1 \end{bmatrix}\biggl)=W_{\pi_u}(t')|\det t'|^\frac{m+1-r}{2}\II_\frko(t_r)\prod_{i=r+1}^m\II_{\frko^\times}(t_i)\]
if $r\geq 1$, and $\displaystyle W_\pi\biggl(\begin{bmatrix} t & \\ & 1 \end{bmatrix}\biggl)=\prod_{i=1}^m\II_{\frko^\times}(t_i)$ if $r=0$. 
If $r\geq 1$, then 
\begin{align*}
\La W_\pi,W_{\pi^\vee}\Ra
&=\int_{T_m}W_\pi\biggl(\begin{bmatrix} t & \\ & 1 \end{bmatrix}\biggl)W_{\pi^\vee}\biggl(\begin{bmatrix} t & \\ & 1 \end{bmatrix}\biggl)\, \prod_{i=1}^m|t_i|^{2i-m-1}\,\d t\\
&=\int_{T_r}W_{\pi_u^{}}(t')W_{\pi_u^\vee}(t')\II_\frko(t_r)\, \prod_{i=1}^r|t_i|^{2i-r}\,\d t
\end{align*}
by the Iwasawa decomposition. 
The last integral equals 
\[\int_{N_r\bsl\GL_r(F)}W_{\pi_u^{}}(g)W_{\pi_u^\vee}(g)\II_\frko(e_rg)|\det g|\,\d g=L^\GL(1,\pi_u^{}\times\pi^\vee_u) \]
by Proposition 2.3 of \cite{JS2}. 
The case $r=0$ is trivial. 
\end{proof}

\begin{definition}[Modified Euler factor for adjoint representations]\label{D:adjoint}
Suppose that $\pi$ is a constituent of $I(\mu_1,\mu_2,\dots,\mu_n)$. Put
\[\frac{1}{\cale(\pi,\Ad,\addchar)}=L^\GL(1,\pi_u^{}\times\pi_u^\vee)\prod_{i<j}\gam(1,\mu_i^{-1}\mu_j^{},\addchar)\times\begin{cases} 
\frac{1}{\zet_F(1)^n} &\text{if $c(\pi)=0$, } \\
\frac{q^{(n-1)c(\pi)}}{\zet_F(1)^{n-1}} &\text{if $c(\pi)>0$. }
\end{cases}\]
\end{definition}

\begin{corollary}\label{cor:42}
Notations and assumptions being as in Proposition \ref{prop:32}, if $\ell$ is sufficiently large, then 
\begin{align*}
\frac{\calb_{\pi^{\ord}}^{[\ell]}}{\calb_\pi}
&=\frac{\alp_\pi^\ell}{[\calk_0^{(n)}(\frkp^{c(\pi)}):I^{(n)}_0(\frkp^\ell)]}\cale(\pi,\Ad,\addchar)\prod_{i<j}\mu_i(-1).
\end{align*}
\end{corollary}

\begin{proof}
Since 
\[[\GL_n(\frko):\calk_0^{(n)}(\frkp^{c(\pi)})]=\begin{cases}
1 &\text{if $c(\pi)=0$, } \\
q^{(n-1)c(\pi)}\frac{\zet_F(1)}{\zet_F(n)} &\text{if $c(\pi)>0$, }
\end{cases}\]
the stated formula follows from Propositions \ref{prop:32} and \ref{prop:42}. 
\end{proof}


\subsection{A depletion \`{a} la Schmidt}\label{ssec:44}

We consider the embedding 
\begin{align*}
\iot'&:\GL_2(F)\hookrightarrow\GL_3(F), & 
\iot'(g)&=\begin{bmatrix} g & \\ & 1 \end{bmatrix}. 
\end{align*}

Let $\chi$ be a character of $F^\times$. 
When $\chi$ is ramified, the conductor $c(\chi)$ of $\chi$ is defined as the smallest positive integer $n$ such that $\chi$ is trivial on $1+\frkp^n$. 
When $\chi$ is unramified, we set $c(\chi)=0$. 
If $c(\chi)\geq 1$, then the Gauss sum is defined by 
\[\frkg(\chi,\addchar)=\sum_{a\in(\frko/\frkp^{c(\chi)})^\times}\chi(a)^{-1}\addchar\biggl(-\frac{a}{\vpi^{c(\chi)}}\biggl). \]
When $c(\chi)=0$, we formally set $\frkg(\chi,\addchar)=1$. 
The Gauss sum is related to the epsilon factor in the following way: 
\[\frkg(\chi,\addchar)=q^{c(\chi)/2}\chi(\vpi)^{-c(\chi)}\vep\left(\frac{1}{2},\chi,\addchar^{-1}\right). \]

\begin{proposition}[\cite{Schmidt}]\label{prop:43}
Let $\chi$ be a character of $\frko^\times$. 
Put $f=c(\chi)$. 
Assume that $f>0$. 
Given $W\in \scrw_{\addchar}(\pi,\calk_0^{(3)}(\frkp^\ell),\ome^\downarrow_\pi)$, we define $\vTh^\chi W\in\scrw_{\addchar}(\pi)$ by 
\[\vTh^\chi W=\sum_{i,j\in(\frko/\frkp^f)^\times}\sum_{y\in\frko/\frkp^{2f}}
\chi(ij)\pi\left(\begin{bmatrix} 1 & \frac{i}{\vpi^f} & \frac{y}{\vpi^{2f}} \\ 0 & 1 & \frac{j}{\vpi^f} \\ 0 & 0 & 1 \end{bmatrix}\right)W. \]
Then $\vTh^\chi W$ possesses the following properties: 
\begin{enumerate}
\item[(i)] $\pi(\iot'(\gam))\vTh^\chi W=\chi^\uparrow(\gam)^{-1}\vTh^\chi W$ for $\gam\in\calk_0^{(2)}(\frkp^{2f})$; 
\item[(ii)] $\vTh^\chi W(\iot'(h))=0$ unless $h\in N_2\calk_0^{(2)}(\frkp^{2f})$;
\item[(iii)] $\vTh^\chi W(\ono_3)=q^{3f}\chi(\vpi)^{2f}\vep\bigl(\frac{1}{2},\chi,\addchar\bigl)^{-2}W(\ono_3)$. 
\end{enumerate}
\end{proposition}

\begin{proof}
The proof is the same as that of Lemma 2.3 of \cite{Schmidt}. 
However, we reproduce the proof here for the reader's convenience.  
Let $\gam=\begin{bmatrix} a & b \\ c & d \end{bmatrix}\in\calk_0^{(2)}(\frkp^{2f})$. 
Observe that
\[\begin{bmatrix} a & b & 0 \\ c & d & 0 \\ 0 & 0 & 1 \end{bmatrix}
\begin{bmatrix} 1 & \frac{i}{\vpi^f} & \frac{y}{\vpi^{2f}} \\ 0 & 1 & \frac{j}{\vpi^f} \\ 0 & 0 & 1 \end{bmatrix}
=\begin{bmatrix} 1 & \frac{ai}{d\vpi^f} & \frac{y'}{\vpi^{2f}} \\ 0 & 1 & \frac{dj}{\vpi^f} \\ 0 & 0 & 1 \end{bmatrix}
\begin{bmatrix} a-\frac{aci}{d\vpi^f} & b-\frac{aci^2}{d\vpi^{2f}} & 0 \\ c & d+\frac{ci}{\vpi^f} & \frac{cy}{\vpi^{2f}} \\ 0 & 0 & 1 \end{bmatrix}, \]
where $y'=ya\Big(1-\frac{ci}{d\vpi^f}\Big)+bj\vpi^f$. 
Since $a\Big(1-\frac{ci}{d\vpi^f}\Big)\in\frko^\times$, $y'$ runs over a full system of representatives mod $\frkp^{2f}$ as $y$ does (for fixed $i,j$). 
We get 
\[\pi(\iot'(\gam))\Tht^\chi W
=\sum_{i,j\in(\frko/\frkp^f)^\times}\sum_{y\in\frko/\frkp^{2f}}
\chi(ij)\pi\left(\begin{bmatrix} 1 & \frac{ai}{d\vpi^f} & \frac{y}{\vpi^{2f}} \\ 0 & 1 & \frac{dj}{\vpi^f} \\ 0 & 0 & 1 \end{bmatrix}\right)W. \]
The right hand side is $\chi(a)^{-1}\Tht^\chi W$, which proves (i).  

To prove (ii), we have only to prove $\Tht^\chi W\left(\iot'\left(\begin{bmatrix} 0 & \vpi^n \\ \vpi^m & 0 \end{bmatrix}\right)\right)=0$ and 
\[\Tht^\chi W\left(\iot'\left(\begin{bmatrix} \vpi^n & 0 \\ \vpi^mc & \vpi^m \end{bmatrix}\right)\right)\neq 0\Rightarrow m=n=0,\;c\in\frkp^{2f}\]
by the Iwasawa decomposition.  
Since 
\[\begin{bmatrix} 0 & \vpi^n & 0 \\ \vpi^m & 0 & 0 \\ 0 & 0 & 1 \end{bmatrix}
\begin{bmatrix} 1 & \frac{i}{\vpi^f} & \frac{y}{\vpi^{2f}} \\ 0 & 1 & \frac{j}{\vpi^f} \\ 0 & 0 & 1 \end{bmatrix}\\
=\begin{bmatrix} 1 & \frac{\vpi^{f+n}}{i\vpi^m} & \frac{j\vpi^n}{\vpi^f} \\ 0 & 1 & \frac{y\vpi^m}{\vpi^{2f}} \\ 0 & 0 & 1 \end{bmatrix}
\begin{bmatrix} \vpi^{n+f} & 0 & 0 \\ 0 & \vpi^{m-f} & 0 \\ 0 & 0 & 1 \end{bmatrix}k, \]
where $k=\iot'\left(\begin{bmatrix} -i^{-1} & 0 \\ \vpi^f & i \end{bmatrix}\right)$, we get 
\begin{align*}
&\Tht^\chi W\left(\iot'\left(\begin{bmatrix} 0 & \vpi^n \\ \vpi^m & 0 \end{bmatrix}\right)\right)\\
=&\sum_{i,j\in(\frko/\frkp^f)^\times}\sum_{y\in\frko/\frkp^{2f}}
\chi(ij)\addchar\left(\frac{\vpi^{f+n}}{i\vpi^m}+\frac{y\vpi^m}{\vpi^{2f}}\right)W\left(\iot'\left(\begin{bmatrix} \vpi^{n+f} & 0 \\ 0 & \vpi^{m-f} \end{bmatrix}\right)\right). 
\end{align*}  
The right hand side is zero as $\displaystyle\sum_{j\in(\frko/\frkp^f)^\times}\chi(j)=0$. 

Finally we let $c\in \frko$ and assume that 
\beq
\Tht^\chi W\left(\iot'\left(\begin{bmatrix} \vpi^n & 0 \\ \vpi^mc & \vpi^m \end{bmatrix}\right)\right)\neq 0. \label{tag:44}
\eeq
Since we can write 
\[\begin{bmatrix} \vpi^n & 0 & 0 \\ \vpi^mc & \vpi^m & 0 \\ 0 & 0 & 1 \end{bmatrix}
\begin{bmatrix} 1 & \frac{i}{\vpi^f} & \frac{y}{\vpi^{2f}} \\ 0 & 1 & \frac{j}{\vpi^f} \\ 0 & 0 & 1 \end{bmatrix}\\
=\begin{bmatrix} 1 & \frac{i\vpi^n}{\vpi^{f+m}} & \frac{y\vpi^n}{\vpi^{2f}} \\ 0 & 1 & \frac{\vpi^mcy}{\vpi^{2f}}+\frac{\vpi^m j}{\vpi^f} \\ 0 & 0 & 1 \end{bmatrix}g_{n,m,i}, \]
where 
\[g_{n,m,i}=\iot'\left(\begin{bmatrix} \vpi^n & 0 \\ 0 & \vpi^m \end{bmatrix}\begin{bmatrix} 1-\frac{ci}{\vpi^f} & -\frac{ci^2}{\vpi^{2f}}\\ c & 1+\frac{ci}{\vpi^f}\end{bmatrix}\right), \]
we get 
\begin{align*}
&\Tht^\chi W\left(\iot'\left(\begin{bmatrix} \vpi^n & 0 \\ \vpi^mc & \vpi^m \end{bmatrix}\right)\right)\\
=&\sum_{i,j\in(\frko/\frkp^f)^\times}\sum_{y\in\frko/\frkp^{2f}}
\chi(ij)\addchar\left(\frac{i\vpi^n}{\vpi^{f+m}}+\frac{\vpi^mcy}{\vpi^{2f}}+j\frac{\vpi^m}{\vpi^f}\right)W(g_{n,m,i}). 
\end{align*}  
Replacing $j$ by $j+\vpi^f$, we find that $m\geq 0$ by (\ref{tag:44}). 
Since  
\beq
\sum_{j\in(\frko/\frkp^f)^\times}\chi(j)\addchar\left(j\frac{\vpi^m}{\vpi^f}\right)\neq 0\Rightarrow m=0, \label{tag:45}
\eeq
we see that $m=0$. 
Since $\displaystyle\sum_{y\in\frko/\frkp^{2f}}\addchar\left(\frac{cy}{\vpi^{2f}}\right)=0$ unless $c\in\frkp^{2f}$, we find that $c\in\frkp^{2f}$.  
We conclude that $n=0$ again by (\ref{tag:45}). 
We see that 
\[\vTh^\chi W(\ono_3)=q^{2f}\frkg(\chi^{-1},\addchar)^2W(\ono_3)=\frac{q^{3f}\chi(\vpi)^{2f}}{\vep\bigl(\frac{1}{2},\chi,\addchar\bigl)^2}W(\ono_3)\]
 by letting $n=m=c=0$ in the formula above. 
\end{proof}


\subsection{The zeta integral of $\Tht^{\ome_\sig}W_\pi$}\label{ssec:45}

Let $\pi$ be an irreducible admissible generic representation of $\GL_3(F)$ and $\sig$ an irreducible admissible \textit{ramified} generic representation of $\GL_2(F)$. 

\begin{proposition}\label{prop:44}
Notation being as above, if $c(\sig)\leq 2c(\ome_\sig)$, then 
\[Z(s,\pi(\iot'(\xi_{1,2}w_2))\Tht^{\ome_\sig} W_\pi,W_\sig)=\frac{q^{3c(\ome_\sig^{})}\ome_\sig(\vpi)^{2c(\ome_\sig^{})}\vep\bigl(\frac{1}{2},\sig,\addchar\bigl)}{\vep\bigl(\frac{1}{2},\ome_\sig,\addchar\bigl)^2[\GL_2(\frko):\calk_0^{(2)}(\frkp^{2c(\ome_\sig)})]}. \]
\end{proposition}

\begin{proof}
Put $f=c(\ome_\sig)$, $n=c(\sig)$ and $W'=\sig(\tau'_n)W_\sig$.
The left hand side is $Z(s,\Tht^{\ome_\sig} W_\pi^{},W')$ by the invariance of the JPSS integral.  
Since $\sig(u)W'=\ome_\sig^\uparrow(u)W'$ for $u\in\calk_0^{(2)}(\frkp^n)$. 
Proposition \ref{prop:43}(ii) gives 
\[Z(s,\Tht^{\ome_\sig} W_\pi^{},W')=\int_{\calk_0^{(2)}(\frkp^{2f})}\Tht^{\ome_\sig} W_\pi(\iot'(h))W'(h)|\det h|^{s-\frac{1}{2}}\,\d h. \]
The right hand side is $\Tht^{\ome_\sig} W_\pi(\ono_3)\cdot W_\sig(\tau'_n)[\GL_2(\frko):\calk_0^{(2)}(\frkp^{2f})]^{-1}$ by Proposition \ref{prop:43}(i). 
Since $W'(\ono_2)=W_\sig(\tau'_n)=\vep\bigl(\frac{1}{2},\sig,\addchar\bigl)$ by Corollary \ref{cor:41}, the proof is complete by Proposition \ref{prop:43}(iii). 
\end{proof}

\begin{remark}\label{rem:41}
If $f=c(\chi)$ is large enough, then 
\begin{align*}
c(\sig\otimes\chi)&=2f, & 
\vep(s,\sig\otimes\chi,\addchar)&=\vep(s,\chi,\addchar)\vep(s,\ome_\sig\chi,\addchar)
\end{align*}
by stability of the epsilon factor (see \cite[Proposition 2.2]{JS}). 
\end{remark}


\appendix


\section{Archimedean computations}\label{sec:b}

\subsection{Local factors}\label{ssec:b1}

For a positive integer $n$ let 
\[\U(n)=\{g\in\GL_n(\CC)\;|\;\trs g^\tht g=\ono_n\}\]
be the compact unitary group. 
Let $\CC^1$ denote the group of complex numbers of absolute value $1$. 
Define the character $\bvep:\CC^\times\to\CC^1$ by $\bvep(x)=\frac{x}{|x|}$. 
We view it as a character of any unitary group via composition with the determinant character. 
 
Fix tuples $\lam_1>\cdots>\lam_n$ and $\mu_1>\cdots>\mu_{n-1}$ of half integers such that $\lam_i-\frac{n+1}{2}\in\ZZ$ and $\mu_j-\frac{n}{2}\in\ZZ$. 
Let $\pi$ be an irreducible representation of $\U(n)$ with Harish-Chandra parameter $(\lam_1,\cdots,\lam_n)$ and $\sig$ an irreducible representation of $\U(n-1)$ with Harish-Chandra parameter $(\mu_1,\cdots,\mu_{n-1})$. 
The $L$-parameters of $\pi$ and $\sig$ are given by
\begin{align*}
\phi_\pi&=\bvep^{2\lam_1}\oplus\cdots\oplus\bvep^{2\lam_n}, &
\phi_\sig&=\bvep^{2\mu_1}\oplus\cdots\oplus\bvep^{2\mu_{n-1}}.
\end{align*}
The $L$-factors are defined by 
\begin{align*}
L(s,\pi\times\sig)
=&\prod_{i=1}^n\prod_{j=1}^{n-1}\Gam_\CC(s+|\lam_i+\mu_j|), \\
L(s,\pi,\Ad)
=&
\Gam_\RR(s+1)^n\prod_{i<j}\Gam_\CC(s+\lam_i-\lam_j), \\
L(s,\sig,\Ad)
=&
\Gam_\RR(s+1)^{n-1}\prod_{i<j}\Gam_\CC(s+\mu_i-\mu_j). 
\end{align*}
Note that $\sig^\vee$ appears as a subrepresentation of $\pi|_{\U(n-1)}$ if and only if 
\[\lam_1>-\mu_{n-1}>\lam_2>\cdots>\lam_{n-1}>-\mu_1>\lam_n. \]
In this case it is easy to check that 
\beq
\frac{L\left(\frac{1}{2},\pi\times\sig\right)}{L(1,\pi^{},\mathrm{Ad})L(1,\sig,\mathrm{Ad})}\in\pi^{\frac{n(n+1)}{2}}\QQ^\times. \label{tag:b0}
\eeq


\subsection{Representations of $\U(2)$}\label{ssec:b2}

For a commutative ring $A$ of characteristic $0$ and a non-negative integer $n$ we write $L_n(A)$ for the module of two variable homogeneous polynomials of degree $n$ over $A$. 
The group $\GL_2(A)$ acts on this module $L_n(A)$ by 
\[\rho_n(\alp)P(X,Y)=P((X,Y)\alp). \]
Define the pairing 
\[\ell_n:L_n(\QQ)\otimes L_n(\QQ)\to \QQ\] 
by 
\beq
\ell_n(X^iY^{n-i}\otimes X^{\prime j}Y^{\prime n-j})=\begin{cases}
(-1)^i\binom{n}{i}^{-1} &\text{if $i+j=n$, }\\
0 &\text{if $i+j\neq n$. }
\end{cases}
\label{tag:b1}
\eeq
It is well-known that for $\alp\in\Mat_2(A)$ and $P,Q\in L_n(A)$ 
\[\ell_n(\rho_n(\alp)P\otimes Q)=\ell_n(P\otimes\rho_n(J\trs\alp J^{-1})Q). \]
 
We view $\U(2)$ as a subgroup of $\GL_2(\CC)$ and regard $\rho_n$ as an irreducible representation of $\U(2)$ of dimension $n+1$. 
Note that $\rho_n$ is irreducible, has central character $\bvep^n$, highest weight $(n,0)$ and Harish-Chandra parameter $\bigl(n+\frac{1}{2},-\frac{1}{2}\bigl)$. 
For $\lam\in\ZZ$ we will write $\rho_{(n+\lam,\lam)}=\rho_n\otimes\bvep^\lam$. 
If $\sig$ is an irreducible representation of $\U(2)$ of dimension $n+1$, then there is an integer $\lam$ such that $\sig\simeq\rho_{(n+\lam,\lam)}$. 
We define the perfect pairing $\ell_\sig:\sig\otimes\sig^\vee\to\CC$ by 
\[\ell_\sig(\sig(h)P\otimes Q)=\bvep(\det h)^\lam\cdot\ell_n(\rho_n(h)P\otimes Q)\]
for $h\in\U(2)$ and $P,Q\in L_n(\CC)$. 

We identify the contragradient representation $\rho_{(a,b)}^\vee$ of $\rho_{(a,b)}^{}$ with $\rho_{(-b,-a)}^{}$.  
Put $J=\begin{bmatrix} 0 & 1 \\ -1 & 0 \end{bmatrix}$. 
Define the representation $\rho_{(a,b)}^\vth$ of $\GL_2(A)$ by $\rho_{(a,b)}^\vth(\alp)=\rho_{(a,b)}(\trs\alp^{-1})$ and an equivariant isomorphism $^\vth:\rho_{(a,b)}^\vth\simeq\rho_{(a,b)}^\vee$ by 
\[P^\vth(X,Y)=\rho_{(a,b)}(J)P(X,Y)=P(-Y,X). \] 

Define $\bfP_n\in L_n(A)\otimes L_n(A)$ by 
\[\bfP_n=(X_1Y_2-Y_1X_2)^n. \]
Since $\rho_n(\alp)\otimes\rho_n(\alp)\bfP_n=(\det\alp)^n\bfP_n$ for $\alp\in\GL_n(\CC)$, this vector $\bfP_n$ spans the line of $\sig\otimes\sig^\vee$ fixed by the diagonal action of $\U(2)$. 
Set $\ell_{\sig^\vee\otimes\sig}=\ell_{\sig^\vee}\otimes\ell_\sig$. 

\begin{lemma}\label{lem:b1}
\[\ell_{\sig^\vee\otimes\sig}(\bfP_n\otimes\bfP_n)=n+1. \]
\end{lemma}

\begin{proof}
Since $\bfP_n=\sum_{i=0}^n\binom{n}{i}(X_1Y_2)^i(-Y_1X_2)^{n-i}$, we have  
\[\bfP_n\otimes\bfP_n
=\sum_{i=0}^n\sum_{j=0}^n\binom{n}{i}\binom{n}{j}(-1)^{i+j}X_1^iY_1^{n-i}X_2^{n-i}Y_2^i\otimes X_1^{\prime j}Y_1^{\prime n-j}X_2^{\prime n-j}Y_2^{\prime j}\]
and hence 
\[\ell_{\sig^\vee\otimes\sig}(\bfP_n\otimes\bfP_n)=(-1)^n\sum_{i=0}^n\binom{n}{i}^2(-1)^i\binom{n}{i}^{-1}(-1)^{n-i}\binom{n}{n-i}^{-1}=n+1 \]
by (\ref{tag:b1}). 
\end{proof}

\subsection{Representations of $\U(3)$}\label{ssec:b3}


Fix a commutative integral domain $A$ of characteristic zero. 
Let $\calp(A)$ be the set of polynomials in $z=\begin{bmatrix} x_1 & x_2 & x_3 \\ y_1 & y_2 & y_3 \end{bmatrix}$ with coefficients in $A$. 
Fix non-negative integers $a,b$. 
The submodule $\calp_{a,b}(A)$ of $\calp(A)$ consists of homogeneous polynomials of degree $a$ in $x_1,x_2,x_3$ and of degree $b$ in $y_1,y_2,y_3$. 
The group $\GL_3(A)$ acts on the module $\calp_{a,b}(A)$ by  
\[\rho(g)P\left(\begin{matrix} x\\ y \end{matrix}\right)=P\left(\begin{matrix} x g\\ y\trs g^{-1}\end{matrix}\right), \]
where $x=(x_1,x_2,x_3)$ and $y=(y_1,y_2,y_3)$. 

We define a submodule $\calt_{a,b}(A)$ of $\calp_{a,b}(A)$ by
\[\calt_{a,b}(A)=(x_1y_1+x_2y_2+x_3y_3)\calp_{a-1,b-1}(A). \]
Since $\calt_{a,b}(A)$ is stable under the action of $\GL_3(A)$, the group $\GL_3(A)$ acts on the quotient module $\frkH_{a,b}(A)=\calp_{a,b}(A)/\calt_{a,b}(A)$. 
If $A$ is a field, then $\frkH_{a,b}(A)$ is an irreducible representation of $\GL_3(A)$ with highest weight $(a,0,-b)$ (cf. Chapter 7 of \cite{HIM}). 

\subsection{Contragradient representations}\label{ssec:b4}

We define a bilinear form 
\[l_{b,a}:\calp_{b,a}(A)\otimes\calp_{a,b}(A)\to A\] 
by 
\[l_{b,a}(Q\otimes P)=\frac{1}{a!b!}Q\begin{pmatrix} \frac{\partial}{\partial y_1} & \frac{\partial}{\partial y_2} & \frac{\partial}{\partial y_3} \\ \frac{\partial}{\partial x_1} & \frac{\partial}{\partial x_2} & \frac{\partial}{\partial x_3} \end{pmatrix}P\begin{pmatrix} x_1 & x_2 & x_3 \\ y_1 & y_2 & y_3 \end{pmatrix}. \]
Note that 
\[l_{b,a}\biggl(\prod_{i=1}^3x_i^{n_i^{}}y_i^{m_i^{}}\otimes {x_i'}^{n_i'}{y_i'}^{m_i'}\biggl)=\begin{cases}
\dfrac{\prod_{i=1}^3n_i!m_i!}{a!b!} &\text{if $n_i^{}=m_i'$ and $m_i^{}=n_i'$, }\\
0 &\text{otherwise. }
\end{cases}\]

\begin{lemma}\label{lem:b2}
For $P\in\calp_{a,b}(A)$, $Q\in\calp_{b,a}(A)$ and $g\in\GL_3(A)$ we have 
\[l_{b,a}(\rho(g)Q\otimes\rho(g) P)
=l_{b,a}(Q\otimes P). \]
\end{lemma}

\begin{proof}
Let $h=(h_{ij})$ be the inverse matrix of $g=(g_{ij})$. 
Put $s_j=\sum_{i=1}^3x_ig_{ij}$ and $t_j=\sum_{i=1}^3y_ih_{ji}$. 
Then $\frac{\partial}{\partial s_j}=\sum_{i=1}^jh_{ji}\frac{\partial}{\partial x_i}$ and $\frac{\partial}{\partial t_j}=\sum_{i=1}^jg_{ij}\frac{\partial}{\partial y_i}$. 
We therefore see that 
\begin{align*}
\rho(g)P\begin{pmatrix} x \\ y \end{pmatrix}&=P\begin{pmatrix} s \\ t \end{pmatrix}, & 
\rho(g)Q\begin{pmatrix} \frac{\partial}{\partial y} \\ \frac{\partial}{\partial x} \end{pmatrix}&=Q\begin{pmatrix} \frac{\partial}{\partial t} \\ \frac{\partial}{\partial s} \end{pmatrix}  
\end{align*} 
from which Lemma \ref{lem:b2} follows. 
\end{proof}

Put $\frkH_{b,a}^\vee(A)=\{P\in\calp_{a,b}^{}(A)\;|\;\Del P=0\}$, where 
\[\Del=\frac{\partial^2}{\partial x_1\partial y_1}+\frac{\partial^2}{\partial x_2\partial y_2}+\frac{\partial^2}{\partial x_3\partial y_3}. \]
If $A$ is a field, then since 
\[\frkH_{b,a}^\vee(A)=\{P\in\calp_{b,a}^{}(A)\;|\;l_{b,a}^{}(Q\otimes P)=0\text{ for }Q\in\calt_{b,a}^{}(A)\},  \]
the restriction of $\rho$ to $\frkH_{b,a}^\vee(A)$ is the contragredient representation of $\frkH_{b,a}^{}(A)$. 
In particular, $\frkH_{a,b}^\vee(A)$ is an irreducible representation of $\GL_3(A)$ with highest weight $(a,0,-b)$. 
The linear form $l_{b,a}$ induces a perfect pairing 
\[l_{b,a}^{}:\frkH_{b,a}^{}(A)\otimes\frkH_{b,a}^\vee(A)\to A. \]

\begin{remark}\label{rem:b1}
When we view $\frkH_{b,a}^\vee(\CC)$ as a representation of $\U(3)$, it is a theta lift of the discrete series of $\U(1,1)$ with Harish-Chandra parameter $\bigl(\frac{a+3}{2};\frac{b-1}{2}\bigl)$ (cf. \cite{KV}). 
Any irreducible representation of $\U(3)$ is of the form $\frkH_{b,a}^\vee\otimes\bvep^\lam$ with $0\leq a,b\in\ZZ$ and $\lam\in\ZZ$. 
\end{remark}

\begin{lemma}\label{lem:b3}
$\calp_{a,b}(\QQ)$ is a direct sum of $\frkH_{b,a}^\vee(\QQ)$ and $\calt_{a,b}(\QQ)$. 
\end{lemma}

\begin{proof}
It suffices to show that 
\[\frkH_{b,a}^\vee(\QQ)\cap\calt_{a,b}(\QQ)=\{0\}. \]
Given $f=\sum a^{ijk}_{lmn}\cdot {x_1'}^i{x_2'}^j{x_3'}^k{y_1'}^l{y_2'}^m{y_3'}^n\in\frkH_{b,a}^\vee(\QQ)\cap\calt_{a,b}(\QQ)$, we have  
\[0=l_{b,a}(f^\vth\otimes f)=(a!b!)^{-1}\sum (a^{ijk}_{lmn})^2i!j!k!l!m!n!, \]
which shows that $f=0$. 
\end{proof}

We endow the space $\frkH_{b,a}^\vth:=\frkH_{b,a}^{}$ with the action $\rho^\vth(g)=\rho(\trs g^{-1})$ of $\GL_3(A)$. 
We define the isomorphism $^\vth:\calp_{b,a}(A)\to\calp_{a,b}(A)$ by 
\[P^\vth\left(\begin{matrix} x\\ y \end{matrix}\right)=P\left(\begin{matrix} y \\ x \end{matrix}\right). \]
Since $(\rho^\vth(g)P)^\vth=\rho(g)P^\vth$, we can define the equivariant isomorphism 
\[^\vth:\frkH_{b,a}^\vth(\QQ)\simeq\frkH_{b,a}^\vee(\QQ), \] 
letting $P^\vth$ be the unique polynomial $Q\in\frkH_{b,a}^\vee(\QQ)$ such that $Q-P^\vth\in\calt_{a,b}(\QQ)$. 

\subsection{The setting}\label{ssec:b5}

Let $V$ be a three dimensional positive definite Hermitian space and $V'\subset V$ a two dimensional subspace. 
We fix a basis of $V$ so that the embedding $\iot:\GL(V')\hookrightarrow\GL(V)$ is given by 
\[\iot:\begin{bmatrix} a & b \\ c & d \end{bmatrix}\mapsto\begin{bmatrix} a & 0 & b \\ 0 & 1 & 0 \\ c & 0 & d \end{bmatrix}. \]

Fix a triplet of integers $\lam_1>\lam_2>\lam_3$ and a pair $\mu_1>\mu_2$ of half integers. 
Let $\pi$ be an irreducible representation of $\U(V)$ with Harish-Chandra parameter $(\lam_1,\lam_2,\lam_3)$ and $\sig$ an irreducible representation of $\U(V')$ with Harish-Chandra parameter $(\mu_1,\mu_2)$. 
The highest weight of $\pi$ is $(-k_1,-k_2,-k_3):=(\lam_1-1,\lam_2,\lam_3+1)$ while the highest weight of $\sig$ is $(-k_1',-k_2'):=\bigl(\mu_1-\frac{1}{2},\mu_2+\frac{1}{2}\bigl)$. 

Recall that $\sig^\vee$ appears as a subrepresentation of $\pi|_{\U(V')}$ if and only if 
\beq
\lam_1>-\mu_2>\lam_2>-\mu_1>\lam_3\Leftrightarrow k_1^{}\leq -k_2'\leq k_2^{}\leq -k_1'\leq k_3^{}. \tag{$\bigstar$}
\eeq 
We consider the additive character $\addchar^\CC_{-2}(x)=e^{2\pi(\bar x-x)}$ for $x\in\CC$. 
Since 
\[\vep\left(\frac{1}{2},\bvep^{2\kap},\addchar^\CC_{-2}\right)
=\begin{cases}
-1 &\text{if $\kap>0$, }\\
1 &\text{if $\kap<0$ }
\end{cases}\]
for $\kap\in\frac{1}{2}\ZZ\setminus\ZZ$, if $(\bigstar)$ holds, then 
\beq
\vep\left(\frac{1}{2},\phi_\pi\otimes\phi_\sig,\addchar^\CC_{-2}\right)
=\prod_{i,j}\vep\left(\frac{1}{2},\bvep^{2(\lam_i+\mu_j)},\addchar^\CC_{-2}\right)=-1. \label{tag:b2}
\eeq

Put \begin{align*}a&=k_3-k_2, & b&=k_2-k_1, & n&=k_3^{}+k_1', & l&=-k_1^{}-k_2'. \end{align*}
If $(\bigstar)$ holds, then $\sig\simeq\rho_{(a-n,l-b)}\otimes\bvep^{-k_2}$ and 
\[\frac{\scrl(\pi\times\sig)}{(2\pi)^2n!l!}=\frac{\Gam(a+b-l+2)\Gam(a-n+b+2)\Gam(a-n+1)\Gam(b-l+1)}{2^{-1}\Gam(a+b+3)\Gam(a-n+b-l+2)\Gam(a+2)\Gam(b+2)}. \]

\subsection{The $\U(V')$-invariant vector}\label{ssec:b6}

We hereafter assume ($\bigstar$). 
Then the $\GL_2(\CC)$-invariant subspace of $\pi\otimes\sig^\vee$ is one-dimensional. 
We will construct a basis vector of this $\U(V')$-invariant line. 
Replacing $\pi$ by $\pi\otimes\bvep^{k_2}$ and $\sig$ by $\sig\otimes\bvep^{-k_2}$, we may assume that $k_2=\lam_2=0$. 
Then 
\begin{align*}
\pi&\simeq\frkH_{b,a}, &
\pi^\vee&\simeq\frkH_{b,a}^\vee, &
\sig&\simeq\rho_{(a-n,l-b)}, &
\sig^\vee&\simeq\rho_{(b-l,n-a)}. 
\end{align*}
Put $\vPi=\pi\otimes\sig$ and $\vPi^\vth=\pi^\vth\otimes\sig^\vth$. 
The maps constructed in \S \ref{ssec:b2} and \S \ref{ssec:b4} give the equivariant isomorphism $^\vth:\vPi^\vth\simeq\vPi^\vee$. 

We define $\bfW_\vPi^H\in\vPi$ by 
\[\bfW_\vPi^H\equiv\det\begin{bmatrix} x_1 & x_3 \\ X_2 & Y_2 \end{bmatrix}^{b-l}\det\begin{bmatrix} X_2 & Y_2 \\ y_3 & -y_1 \end{bmatrix}^{a-n}x_2^ly_2^n \pmod{\calt_{b,a}(\QQ)}\]
and define a polynomial $\bfP_{\vPi^\vee}$ in $x_1,x_2,x_3,y_1,y_2,y_3,X_2,Y_2$ by
\[\bfP_{\vPi^\vee}^{}=\sum_{m=0}^{\min\{n,l\}}(-1)^ma_{\vPi^\vee}^{(m)}\bfP_{\vPi^\vee}^{(m)}\cdot x_2^{n-m}y_2^{l-m} \]
for $0\leq n\leq a$ and $0\leq l\leq b$, where 
\begin{align*}
\bfP_{\vPi^\vee}^{(m)}&=(x_1y_1+x_3y_3)^m(x_1Y_2-x_3X_2)^{a-n}(y_1X_2+y_3Y_2)^{b-l}, \\
a_{\vPi^\vee}^{(m)}&=\binom{n}{m}\binom{l}{m}\frac{\Gam(m+1)\Gam(a-n+b+2)\Gam(a+b-l+2)}{\Gam(a+b+2)\Gam(a-n+b-l+m+2)}. \end{align*}

\begin{lemma}\label{lem:b4}
\[\sum_{m=0}^{\min\{n,l\}}a_{\vPi^\vee}^{(m)}=1. \]
\end{lemma}

\begin{proof}
If $n\geq l$, then Vandermonde's convolution gives 
\[\sum_{m=0}^{\min\{n,l\}}a_{\vPi^\vee}^{(m)}=\binom{a+b+1}{l}^{-1}\sum_{m=0}^l\binom{n}{m}\binom{a-n+b+1}{l-m}=1 \]
as claimed. 
The case $n\leq l$ can be proved in the same way. 
\end{proof}

\begin{proposition}\label{prop:b1}
\begin{enumerate}
\item\label{prop:b11} $\bfP_{\vPi^\vee}\in\frkH_{a,b}\otimes L_{a-n+b-l}(\QQ)$. 
\item\label{prop:b12} $\vPi^\vee(\iot(h),h)\bfP_{\vPi^\vee}=\bfP_{\vPi^\vee}$ for $h\in\GL(V')$. 
\item\label{prop:b13} $\bfW_\vPi^{H\vth}=\bfP_{\vPi^\vee}^{}$. 
\item\label{prop:b14} $\vPi(\iot(h),h)\bfW_\vPi^H=\bfW_\vPi^H$ for $h\in\GL(V')$. 
\end{enumerate}
\end{proposition}

\begin{proof}
The polynomial $P_{n,l}$ is homogeneous of degree $a$ in $x_1,x_2,x_3$ and of degree $b$ in $y_1,y_2,y_3$ and of degree $a-n+b-l$ in $X_2,Y_2$. 
Since
\[\biggl(\frac{\partial^2}{\partial x_1\partial y_1}+\frac{\partial^2}{\partial x_3\partial y_3}\biggl)\bfP_{\vPi^\vee}^{(m)}=m(m+1+a-n+b-l)\bfP_{\vPi^\vee}^{(m-1)} \]
and since 
\[m(m+1+a-n+b-l)a_{\vPi^\vee}^{(m)}=(n-m+1)(l-m+1)a_{\vPi^\vee}^{(m-1)} \]
for $m=1,2,\dots,\min\{l,n\}$, we prove (\ref{prop:b11}). 
Recall that $\GL(V')$ acts on 
\[P=\sum_{i,j} P_{i,j}\cdot x_2^{n-i}y_2^{l-j}\in\frkH_{b,a}^\vee\otimes L_{c-d}^{}(\CC)\] 
by
\[\pi^\vee(\iot(h))\otimes\rho_{(c,d)}(h)P_{i,j}\left(\begin{bmatrix} x_1 & x_3 \\ y_3 & -y_1 \\ X_2 & Y_2 \end{bmatrix}\right)=(\det h)^{l-b-j+d}P_{i,j}\left(\begin{bmatrix} x_1 & x_3 \\ y_3 & -y_1 \\ X_2 & Y_2 \end{bmatrix}h\right) \]
for $h\in\GL(V')$.  
It follows that 
\[\pi^\vee(\iot(h))\otimes\sig^\vee(h)\bfP_{\vPi^\vee}^{(m)}
=\bfP_{\vPi^\vee}^{(m)}, \]
which proves (\ref{prop:b12}). 
Since 
\[\bfP_{\vPi^\vee}^{(m)}=(-x_2y_2)^m(x_1Y_2-x_3X_2)^{a-n}(y_1X_2+y_3Y_2)^{b-l}\pmod{\calt_{a,b}(\QQ)}, \]
we see that 
\begin{align*}
\bfP_{\vPi^\vee}
&\equiv\sum_ma_{\vPi^\vee}^{(m)}(x_1Y_2-x_3X_2)^{a-n}(y_1X_2+y_3Y_2)^{b-l}x_2^ny_2^l \\
&\equiv\Big(\sum_ma_{\vPi^\vee}^{(m)}\Big)\bfW_\vPi^{H\vth}\pmod{\calt_{a,b}(\QQ)}.
\end{align*}
Now (\ref{prop:b13}) is a consequence of Lemma \ref{lem:b4}. 
One sees (\ref{prop:b14}) from (\ref{prop:b13}). 
\end{proof}

\subsection{The restriction to $\U(V')$}\label{ssec:b7}

The branching law for the restriction of $\pi^\vee$ to $\U(V')$ is well-known:  
\[\pi^\vee|_{\U(V')}\simeq\oplus_{n=0}^a\oplus_{l=0}^b\rho_{(a-n,-b+l)}^{}. \]
Hara and Namikawa \cite{HN} explicitly give a $\GL(V')$-equivariant map 
\[\nabla_{n,l}:\frkH_{b,a}^\vee\to\rho_{(a-n,-b+l)}^{}\] 
by 
\[(\nabla_{n,l}P)(X_1,Y_1)=\frac{1}{n!l!}\cdot\frac{\partial^{n+l}P}{\partial x_2^n\partial y_2^l}\left(\begin{bmatrix} X_1 & 0 & Y_1 \\ -Y_1 & 0 & X_1 \end{bmatrix}\right). \]
We define the constants $c_{n,l}$ by 
\[l_{b,a}(Q\otimes P)=\sum_{n=0}^a\sum_{l=0}^bc_{n,l}\ell_{a-n+b-l}(\nabla_{n,l}Q\otimes\nabla_{l,n}P). \]

\subsection{The Ichino-Ikeda integral}\label{ssec:b8}

We will consider the integral
\[J(\bfP\otimes \bfQ)=\int_{\U(V')}\ell_\vPi((\vPi(\iot(h),h)\bfP)\otimes \bfQ)\,\d h \]
for $\bfP\in\vPi$ and $\bfQ\in\vPi^\vee$, where we set $\ell_\vPi=\ell_\pi\otimes\ell_\sig$. 

Since $\nabla_{n,l}\bfP_{\vPi^\vee}=a_{\vPi^\vee}^{(0)}\bfP_{a-n+b-l}$ and since $\nabla_{n',l'}\bfP_{\vPi^\vee}=0$ unless $n'=n$ and $l'=l$, Proposition \ref{prop:b1} and Lemma \ref{lem:b1} show that  
\begin{align*}
J(\bfW_\vPi^H\otimes \bfW_\vPi^{H\vth})
&=J(\bfP_{\vPi^\vee}^{\vth^{-1}}\otimes \bfP_{\vPi^\vee}^{})\\
&=\ell_\vPi(\bfP_{\vPi^\vee}^{\vth^{-1}}\otimes \bfP_{\vPi^\vee})\\
&=c_{n,l}^{}\ell_{\sig\otimes\sig^\vee}(a_{\vPi^\vee}^{(0)}\bfP_{a-n+b-l}\otimes a_{\vPi^\vee}^{(0)}\bfP_{a-n+b-l})\\
&=a_{\vPi^\vee}^{(0)}c_{n,l}^{}a_{\vPi^\vee}^{(0)}(a-n+b-l+1).  
\end{align*}
We denote the dimensions of $\pi$ and $\sig$ by $\d(\pi)$ and $\d(\sig)$. 
Since 
\begin{align*}
\d(\pi)&=\frac{1}{2}(a+b+2)(a+1)(b+1), & 
\d(\sig)&=a-n+b-l+1, 
\end{align*}
we have 
\begin{align*}
&\frac{J(\bfW_\vPi^H\otimes \bfW_\vPi^{H\vth})}{\scrl(\pi\times\sig^\vee)}\\
=&\frac{\d(\pi)\d(\sig)a_{\vPi^\vee}^{(0)}c_{n,l}^{}a_{\vPi^\vee}^{(0)}\Gam(a+b+2)\Gam(a-n+b-l+2)\Gam(a+1)\Gam(b+1)}{n!l!(2\pi)^2\Gam(a+b-l+2)\Gam(a-n+b+2)\Gam(a-n+1)\Gam(b-l+1)}\\
=&\frac{\d(\pi)\d(\sig)}{(2\pi)^2}\binom{a}{n}\binom{b}{l}a_{\vPi^\vee}^{(0)}c_{n,l}^{}a_{\vPi^\vee}^{(0)}\frac{\Gam(a+b+2)\Gam(a-n+b-l+2)}{\Gam(a+b-l+2)\Gam(a-n+b+2)}\\
=&\frac{\d(\pi)\d(\sig)}{(2\pi)^2}\binom{a}{n}\binom{b}{l}a_{\vPi^\vee}^{(0)}c_{n,l}^{}. 
\end{align*}

Let $W_\sig=X_1^{a-n+b-l}$ and $W_\pi=x_1^by_3^a$ be highest weight vectors. 
Since 
\begin{align*}
W_{\sig^\vee}&:=W_\sig^\vth=(-Y_1')^{a-n+b-l}, & 
W_{\pi^\vee}&:=W_\pi^\vth=y_1^{\prime b}x_3^{\prime a} 
\end{align*}
(cf. (\ref{tag:61}) and (\ref{tag:62})), we have 
\[\ell_\vPi((W_\pi\otimes W_\sig)\otimes(W_{\pi^\vee}\otimes W_{\sig^\vee}))=1. \]
Now the following formula is a consequence of Proposition \ref{prop:b3} below. 

\begin{proposition}\label{prop:b2}
\[\frac{\scrl(\pi\times\sig)^{-1}J(\bfW_\vPi^H\otimes \bfW_\vPi^{H\vth})}{\ell_\vPi((W_\pi\otimes W_\sig)\otimes(W_{\pi^\vee}\otimes W_{\sig^\vee}))}
=\frac{\d(\pi)\d(\sig)}{(2\pi)^2}. \]
\end{proposition}

\subsection{Computation of $c_{n,l}$}\label{ssec:b9}

We will prove the following formula:

\begin{proposition}\label{prop:b3}
\[\binom{a}{n}\binom{b}{l}a_{\vPi^\vee}^{(0)}c_{n,l}=1. \]
\end{proposition}

\begin{proof}
Assume that $n\geq l$. 
Let $\frkg\frkl(V)$ denote the Lie algebra of $\GL(V)$. 
Put 
\begin{align*}
E_1&=\begin{bmatrix} 0 & 0 & 0 \\ 1 & 0 & 0 \\ 0 & 0 & 0 \end{bmatrix}, &
E_2&=\begin{bmatrix} 0 & 0 & 0 \\ 0 & 0 & 0 \\ 0 & -1 & 0 \end{bmatrix}. 
\end{align*}
Then 
\begin{align*}
\pi^\vee(-\trs E_1)&=y_2'\frac{\partial}{\partial y_1'}-x_1'\frac{\partial}{\partial x_2'}, &
\pi^\vee(-\trs E_2)&=x_2'\frac{\partial}{\partial x_3'}-y_3'\frac{\partial}{\partial y_2'}. 
\end{align*}
Let $\bar w_{0,0}={y_1'}^a{x_3'}^b\in\frkH_{a,b}^\vee$ be the lowest weight vector. 
Put 
\[\bar w_{n,0}^{}=\frac{(a-n)!}{a!}\pi^\vee(-\trs E_1)^n\bar w^{}_{0,0}={y_1'}^{a-n}{y_2'}^n{x_3'}^b. \]
We define a vector $\bar w^{}_{n,l}\in \frkH_{a,b}^\vee$ by $\bar w_{n,l}=\pi^\vee(-\trs E_2)^l\bar w^{}_{n,0}$. 
Observe that 
\begin{align*}
\bar w_{n,l}&=\biggl(x_2'\frac{\partial}{\partial x_3'}-y_3'\frac{\partial}{\partial y_2'}\biggl)^l{y_1'}^{a-n}{y_2'}^n{x_3'}^b\\
&=\sum_{m=0}^{\min\{n,l\}}\binom{l}{m}\frac{n!b!}{(n-m)!(b-l+m)!}{y_1'}^{a-n}{y_2'}^{n-m}{x_2'}^{l-m}(-y_3')^m{x_3'}^{b-l+m}. 
\end{align*}
By the definition of $c_{n,l}$ we have 
\begin{align*}
&\ell_\vPi(\bfP_{\vPi^\vee}\otimes\bar w_{n,l}X_2^{\prime a-n+b-l})\\
=&c_{n,l}\ell_{\sig\otimes\sig^\vee}(\nabla_{n,l}\bfP_{\vPi^\vee}\otimes\bar\nabla_{n,l}\bar w_{n,l}X_2^{\prime a-n+b-l})\\
=&c_{n,l}\ell_{\sig\otimes\sig^\vee}\biggl(a_{\vPi^\vee}^{(0)}\bfP_{a-n+b-l}\otimes \frac{b!(-Y_1')^{a-n+b-l}X_2^{\prime a-n+b-l}}{(b-l)!}\biggl)
=\frac{c_{n,l}a_{\vPi^\vee}^{(0)}b!}{(b-l)!}. 
\end{align*}

On the other hand, we have  
\begin{align*}
&\ell_\vPi(\bfP_{\vPi^\vee}\otimes\bar w_{n,l}X_2^{\prime a-n+b-l})\\
=&\sum_{m=0}^{\min\{n,l\}}(-1)^ma_{\vPi^\vee}^{(m)}\ell_\vPi(\bfP_{\vPi^\vee}^{(m)}\cdot x_2^{n-m}y_2^{l-m}\otimes\bar w_{n,l}X_2^{\prime a-n+b-l})
\end{align*}
is the sum of the product of $(-1)^ma_{\vPi^\vee}^{(m)}\binom{l}{m}\frac{n!b!}{(n-m)!(b-l+m)!}$ and 
\begin{multline*}
\ell_\vPi(\bfP_{\vPi^\vee}^{(m)}\cdot x_2^{n-m}y_2^{l-m}\otimes{y_1'}^{a-n}{y_2'}^{n-m}{x_2'}^{l-m}(-y_3')^m{x_3'}^{b-l+m}X_2^{\prime a-n+b-l})\\
=(-1)^m(n-m)!(l-m)!\frac{(a-n)!}{a!}m!\frac{(b-l+m)!}{b!}
\end{multline*}
over $m=0,1,2,\dots,\min\{n,l\}$. 
It follows that  
\begin{align*}
c_{n,l}a_{\vPi^\vee}^{(0)}&=\frac{(b-l)!}{b!}\sum_{m=0}^{\min\{n,l\}}(-1)^ma_{\vPi^\vee}^{(m)}\binom{l}{m}n!(-1)^m(l-m)!\frac{(a-n)!}{a!}m!\\
&=\binom{a}{n}^{-1}\binom{b}{l}^{-1}\sum_{m=0}^{\min\{n,l\}}a_{\vPi^\vee}^{(m)}=1 
\end{align*}
by Lemma \ref{lem:b4}. 
\end{proof}


\section{Local integrals at non-split primes}\label{sec:d}

Let $E/F$ be a quadratic extension of non-archimedean local fields. 
We denote by $e$ the ramification index of $E/F$. 
Fix unramified characters $\chi,\rho$ of $E^\times$. 
Put 
\begin{align*}
\bet&=\chi(\vpi_E), & 
\gam&=\rho(\vpi_E), & 
\bfm'(\alp)&=\begin{bmatrix} \alp & 0\\ 0 & \bar\alp^{-1} \end{bmatrix}, & 
\bfm'(\alp,g_0)&=\begin{bmatrix} \alp & 0 & 0 \\ 0 & g_0 & 0 \\ 0 &  0 & \bar\alp^{-1} \end{bmatrix} 
\end{align*}
for $\alp\in E^\times$ and $g_0\in E^1$, where $\vpi_E$ is a generator of the maximal ideal of $\frkr$. 
The space $V'_\rho$ of the principal series $\sig=I'(\rho)$ consists of functions $f':H\to\CC$ which satisfy
\[f'(\bfm'(\alp)\bfu(b)h)=\rho(\alp)|\alp|_E^{1/2}f'(h)\]
for $\alp\in E^\times$, $b\in F$ and $h\in H$. 
The space $V_\chi$ of the principal series $\pi=I(\chi)$ consists of functions $f:G\to\CC$ which satisfy
\[f(\bfm'(\alp,g_0)ug)=\chi(\alp)|\alp|_Ef(g)\]
for $\alp\in E^\times$, $g_0\in G_0$, $u\in N$ and $g\in G$. 

Let $\calk'$ and $\calk$ be special maximal compact subgroups of $H$ and $G$, respectively. 
Define $f'_{\chi'}\in V'$ and $f_{\chi'}\in V$ by $f'_{\chi'}(k')=f_\chi(k)=1$ for $k'\in \calk'$ and $k\in \calk$. 
We consider the zonal spherical functions
\begin{align*}
L'(h)&=\La \sig(h)f'_\rho,f'_{\rho^{-1}}\Ra', &
L(g)&=\La\pi(g)f_\chi,f_{\chi^{-1}}\Ra 
\end{align*}
for $h\in H$ and $g\in G$. 
We here normalize the local perfect pairings $\La\;,\;\Ra'$ and $\La\;,\;\Ra$ so that $L'(k')=L(k)=1$ for $k'\in\calk'$ and $k\in\calk$. 
Put 
\begin{align*}
C'(\rho)&=\frac{1-q^{-1}\gam^{-e}}{1-\gam^{-e}}, & 
C(\chi)&=\frac{(1-q^{-2/e}\bet^{-1})(1+q^{e-2}\bet^{-1})}{1-\bet^{-2}}. 
\end{align*}

\begin{proposition}\label{prop:d1}
\begin{align*}
L'(\bfm'(\vpi_E^m))&=(1+q^{-1})^{-1}q^{-m/e}(C(\rho)\gam^m+C(\rho^{-1})\gam^{-m}), \\
L(\bfm'(\vpi_E^m,1))&=(1+q^{2e-5})^{-1}q^{-2m/e}(C(\chi)\bet^m+C(\chi^{-1})\bet^{-m}). 
\end{align*}
\end{proposition}

\begin{proof}
Let $\alp'$ be the positive root of $H$ and $\alp$ the positive root of $G$. 
In view of (15) of \cite{Cas} we have  
\begin{align*}
q_{\alp'}&=q_\alp^{}=q, &
q_{\alp'/2}&=1, & 
q_{\alp/2}&=q^{4-2e}, \\
a'_\alp&=\begin{bmatrix} \vpi & 0 \\ 0 & \vpi^{-1}\end{bmatrix}, &  
a_\alp&=\begin{bmatrix} \vpi_E & 0 & 0 \\ 0 & 1 & 0 \\ 0 & 0 & (\vpi_E^{-1})^\vth\end{bmatrix}. 
\end{align*}
Then the $c$-function for $V'_\rho$ (resp. $V_\chi$) is given by $C'(\rho)$ (resp. $C(\chi)$). 
The stated formulas are special cases of the Macdonald formula (see Theorem 4.2 of \cite{Cas}).  
\end{proof}

We will compute the following integral  
\[J(f_\chi^{},f'_\rho)=\int_HL(\iot(h))L'(h)\,\d h. \]

\begin{proposition}\label{prop:d2}
\[J(f_\chi^{},f'_\rho)=\scrl(\pi\times\sig). \]
\end{proposition}

\begin{proof}
When $e=1$, the formula is proved in \cite[Theorem 2.12]{NHarris}. 
Recall that 
\begin{align*}
L(s,\pi\times\sig)&=L(s,\chi\rho)L(s,\chi\rho^{-1})L(s,\chi^{-1}\rho)L(s,\chi^{-1}\rho^{-1})L(s,\rho)L(s,\rho^{-1}), \\
L(s,\sig,\mathrm{Ad})&=\zet_F(s)L(s,\eps_{E/F})L(s,\rho|_{F^\times})L(s,\rho^{-1}|_{F^\times}), \\
L(s,\pi,\mathrm{Ad})&=\zet_E(s)L(s,\eps_{E/F})L(s,\chi)L(s,\chi^{-1})L(s,\eps_{E/F}\chi_0)L(s,\eps_{E/F}\chi_0^{-1}) 
\end{align*}
(cf. Remark \ref{rem:21}), where $\chi_0$ stands for the restriction of $\chi$ to $F^\times$.  
By the Cartan decomposition we have 
\[J(f_\chi^{},f'_\rho)
=\sum_{m=0}^\infty L(\bfm'(\vpi_E^m,1))L'(\bfm'(\vpi_E^m))[\calk'\bfm'(\vpi_E^m)\calk':\calk']. \]
Since $[\calk'\bfm'(\vpi_E^m)\calk':\calk']=q^{2m/e}(1+q^{-1})$ if $m>0$, Proposition \ref{prop:d1} gives  
\[J(f_\chi^{},f'_\rho)
=1+\frac{1}{1+q^{2e-5}}\sum_{m=1}^\infty \frac{\bigl(C(\chi)\bet^m+\frac{C(\chi^{-1})}{\bet^m}\bigl)\bigl(C(\rho)\gam^m+\frac{C(\rho^{-1})}{\gam^m}\bigl)}{q^{m/e}}. \]

Put $x=q^{-1/2}$. 
If $e=2$, then 
\[\scrl(\pi\times\sig)=\frac{(1-x^2)(1+\gam x)(1+\gam^{-1}x)(1-\bet x^2)(1-\bet^{-1}x^2)}{(1+x^2)(1-\bet\gam x)(1-\bet\gam^{-1}x)(1-\bet^{-1}\gam x)(1-\bet^{-1}\gam^{-1}x)} \]
and $(1+x^2)(J(f_\chi^{},f'_\rho)-1)$ equals
\begin{align*}
&\frac{1-x^2\gam^{-2}}{1-\gam^{-2}}\cdot\frac{1-x^2\bet^{-1}}{1-\bet^{-1}}\cdot\frac{\bet\gam x}{1-\bet\gam x}
+\frac{1-x^2\gam^2}{1-\gam^2}\cdot\frac{1-x^2\bet^{-1}}{1-\bet^{-1}}\cdot\frac{\bet\gam^{-1}x}{1-\bet\gam^{-1} x}\\
+&\frac{1-x^2\gam^{-2}}{1-\gam^{-2}}\cdot\frac{1-x^2\bet}{1-\bet}\cdot\frac{\bet^{-1}\gam x}{1-\bet^{-1}\gam x}
+\frac{1-x^2\gam^2}{1-\gam^2}\cdot\frac{1-x^2\bet}{1-\bet}\cdot\frac{\bet^{-1}\gam^{-1}x}{1-\bet^{-1}\gam^{-1} x}. 
\end{align*}
We can prove the wanted identity by a brute force calculation. 
\end{proof}


\section{Ramified computations: the minus sign case}\label{sec:e}

\subsection{Maximal compact subgroups}\label{ssec:e1}
In this section $E$ is a quadratic extension of a local field $F$ of odd residual characteristic. 
We write $\frkr$ for the maximal compact subring of $E$ and denote by $\frkq$ the maximal ideal of $\frkr$. 
Given a Hermitian matrix $T$, we define the Hermitian form on $W=E^3$ by $(u,v)_T=\trs u^\tht Tv$, where 
\[T=\begin{bmatrix} 
t_1 & 0 & 0 \\ 
0 & 1 & 0 \\ 
0 & 0 & t_2 
\end{bmatrix}. \]
We denote its unitary group by $G$. 
Put 
\begin{align*}
e_1&=\trs(1,0,0), & 
e_2&=\trs(0,1,0), & 
e_3&=\trs(0,0,1), & 
W'&=Ee_1\oplus Ee_3. 
\end{align*} 
Let $H=\{h\in G\;|\;he_2=e_2\}$ be the unitary group of $W'$. 
 
Fix a generator $\vpi$ of $\frkp$.
We choose $t_1,t_2$ in the following way: 
\begin{enumerate}
\item[(i)] if $E/F$ is unramified, then $t_1=-\frac{1}{2\vpi}$ and $t_2=-1$;
\item[(ii)] if $E/F$ is ramified, then $t_1,t_2\in\frko^\times$ and $-t_1t_2\notin\Nr_{E/F}(E^\times)$. 
\end{enumerate}
Note that $T'=\begin{bmatrix} 
t_1 & 0 \\
0 & t_2 \end{bmatrix}$ is not split. 
The group $H$ is compact. 

We call an $\frkr$-lattice $L$ in $W$ integral if $(x,x)\in\frko$ for every $x\in L$ and call $L$ maximal if it is
maximal among the integral $\frkr$-lattices. 
Take the maximal $\frkr$-integral lattice $\call=\frkr e_1^{}\oplus\frkr e_2^{}\oplus\frkr e_3^{}$ in Case (ii). 
Clearly, 
\begin{align*}
(e_2,e_2)_T&=1, &
e_2&\in\call, & 
\{(x,e_2)_T\;|\;x\in\call\}&=\frkr. 
\end{align*}

Put 
\begin{align*}
\eta&=\begin{bmatrix} 
0 & 1 & 1 \\ 
1 & 0 & 0 \\ 
0 & -\vpi & \vpi 
\end{bmatrix}, & 
\call&=\left\{\eta^{-1}x\;\left|\;x=\begin{bmatrix} x_1 \\ x_2 \\ x_3\end{bmatrix},\;x_1,x_2,x_3\in\frkr\right.\right\} 
\end{align*}
in Case (i). 
Then 
\[\trs \eta^\tht\begin{bmatrix} 0 & 0 & 1 \\ 0 & 1 & 0 \\ 1 & 0 & 0 \end{bmatrix}\eta=-2\vpi T=:\scrt. \]
Thus $\call$ is a maximal $\frkr$-integral lattice with respect to $\scrt$. 
Note that 
\begin{align*}
(e_2,e_2)_\scrt&=-2\vpi, &
e_2&\in\call, & 
\{(x,e_2)_\scrt\;|\;x\in\call\}&=\frkr. 
\end{align*}
Let $\calk=\{g\in G\;|\;g\call=\call\}$ be a maximal compact subgroup of $G$.  


\subsection{The Ichino-Ikeda integral}\label{ssec:e4}

Let $\chi$ be an unramified character of $E^\times$. 
We retain the notation in Appendix \ref{sec:d}. 
We will compute the integral  
\[J(f_\chi^{},1_H)=\int_HL(\iot(h))\,\d h. \]

\begin{proposition}\label{prop:e2}
\[\scrl(\pi\times\sig^\vee)^{-1}J(f_\chi^{},1_H)=L(1,\eps_{E/F})^2. \]
\end{proposition}

\begin{remark}\label{rem:e3}
Let $\St_E$ denote the Steinberg representation of $\GL_2(E)$. 
We write $\1_H$ for the trivial representation of $H$. 
Fix a non-trivial additive character ${\boldsymbol\vPs}$ on $E$. 
If $\chi$ is unramified and unitary, then
\begin{align*}
L(s,I(\chi)\times\1_H)&=L(s,\St_E\otimes\chi)L(s,\St_E\otimes\chi^{-1})L(s,\St_E), \\
\vep(s,I(\chi)\times\1_H,{\boldsymbol\vPs})&=\vep(s,\St_E\otimes\chi,{\boldsymbol\vPs})\vep(s,\St_E\otimes\chi^{-1},{\boldsymbol\vPs})\vep(s,\St_E,{\boldsymbol\vPs}). 
\end{align*}
In particular, 
\[\vep(1/2,I(\chi)\times\1_H,{\boldsymbol\vPs})=-1. \]
\end{remark}

\begin{proof}
Since $L(s,\1_H,\mathrm{Ad})=L(s,\eps_{E/F})\zet_F(s+1)$, we have 
\[\scrl(\pi\times\1_H)=\frac{L(3,\eps_{E/F})}{L(1,\eps_{E/F})L(1,\eps_{E/F}\chi_0)L(1,\eps_{E/F}\chi_0^{-1})} \]
(see Remark \ref{rem:e3} and the proof of Proposition \ref{prop:d2}). 
In particular, we have $\scrl(\pi\times\1_H)=1$ in Case (ii). 
Since $H\subset\calk$ by Lemma 3.14 of \cite{Shimura} applied with $q=1$, $\vph_0=T$, $L=\call$, $\frkb=\frkr$ and $h=e_2$, there is nothing to compute in Case (ii). 

Finally, we consider Case (i). 
Lemma 3.14 of \cite{Shimura} applied with $q=-2\vpi$, $\vph_0=\scrt$, $L=\call$, $\frkb=\frkr$ and $h=e_2$ shows that $[H:H\cap\calk]=1+q$. 
Moreover, its proof shows that 
\[J(f_\chi^{},1_H)
=\frac{1}{q+1}\sum_{h\in H/H\cap\calk}L(\iot(h))
=\frac{1}{q+1}(1+qL(\bfm'(\vpi,1))). 
\]
Put $\bet=\chi(\vpi)$. 
Proposition \ref{prop:d1} gives 
\[L(\bfm'(\vpi,1))
=\frac{C(\chi)\bet+C(\chi^{-1})\bet^{-1}}{q^2(1+q^{-3})}
=\frac{L(3,\eps_{E/F})}{q^2}(\bet+\bet^{-1}+q^{-1}-q^{-2}),\]
which completes our proof. 
\end{proof}




\begin{thebibliography}{99}

\bibitem{BC04},
{Bella\"{\i}che, J. and Chenevier, G.},
{Formes non temp\'{e}r\'{e}es pour {$\rm U(3)$} et conjectures de
              {B}loch-{K}ato},
{Ann. Sci. \'{E}cole Norm. Sup. (4)}, 
{\bf 37},
{2004},
{4},
{611--662}. 

\bibitem{BD}
{M.~Bertolini} and {H.~Darmon}, 
{Heegner points on Mumford-Tate curves\/}, 
Invent. Math. {\bf 126} (1996), no. 3, 413--456.

\bibitem{BCZ}
{R. Beuzart-Plessis}, {P.-H. Chaudouard} and {M. Zydor}, 
{The global Gan-Gross-Prasad conjecture for unitary groups: the endoscopic case\/}, 
Publications math\'{e}matiques de l'IH\'{E}S {\bf 135} (2022) 183--336. 
 
\bibitem{BLZZ}
{R. Beuzart-Plessis}, {Y. Liu}, {W. Zhang} and {X. Zhu}, 
{Isolation of cuspidal spectrum, with application to the Gan-Gross-Prasad conjecture\/}, 
Ann. Math. {\bf 194} (2021) 519--584. 

\bibitem{BR}
{D.~Blasius} and {J.~D.~Rogawski}, 
{Tate class and arithmetic quotient of two-ball\/}, 
The zeta functions of Picard modular surfaces, 421--444, Univ. Montr\'{e}al, Montreal, QC, 1992.  


\bibitem{Cas}
{W.~Casselman}, 
{The unramified principal series of $p$-adic groups I: the spherical function\/},
Compos. Math. {\bf 40}, no. 3 (1980), 387--406.

\bibitem{CH}
{M.~Chida} and {M.-L.~Hsieh}, 
{Special values of anticyclotomic $L$-functions for modular forms\/}, 
J. Reine Angew. Math. 741 (2018), 87--131.

\bibitem{Coe}
{J.~Coates}, 
{On $p$-adic $L$-functions attached to motives over $\QQ$. II\/}, 
Bol. Soc. Brasil. Mat. (N.S.) {\bf 20}
(1989), no. 1, 101--112.

\bibitem{Coe2}
\bysame, 
{On $p$-adic $L$-functions\/}, 
S\'{e}minaire Bourbaki, Vol. 1988/89, Ast\'{e}risque 177--178 (1989),
Exp. No. {\bf 701}, 33--59.

\bibitem{CPR}
{J.~Coates} and {B.~Perrin-Riou}, 
{On $p$-adic $L$-functions attached to motives over $\QQ$\/}, 
Algebraic Number Theory, Adv. Stud. Pure Math., vol. {\bf 17}, 
Academic Press, Boston, MA, 1989, pp. 23--54.

\bibitem{Deligne73},
{Deligne, P.},
{Les constantes des \'{e}quations fonctionnelles des fonctions
              {$L$}},
{Modular functions of one variable, {II} ({P}roc. {I}nternat.
              {S}ummer {S}chool, {U}niv. {A}ntwerp, {A}ntwerp, 1972)},
{Lecture Notes in Math., Vol. 349},
{501--597},
{Springer, Berlin-New York},
{1973}. 



\bibitem{GHY}
{W.~T.~Gan}, {J.~P.~Hanke} and {J.-K.~Yu}, 
{On an exact mass formula of Shimura\/}, 
Duke Math. J. 
{\bf 107}
(2001)
103--133. 

\bibitem{Geraghty}
{D.~Geraghty}, 
{Modularity lifting theorems for ordinary Galois representations\/}, 
Math. Ann. {\bf 373} (2019), 1341--1427. 

\bibitem{GJ}
{R.~Godement} and {H.~Jacquet}, 
{Zeta Functions of Simple Algebras\/}, 
Lecture Notes in Math., vol. {\bf 260}. 
Springer, Berlin (1972)

\bibitem{Gre}
{R.~Greenberg}, 
{Iwasawa theory and $p$-adic deformations of motives\/}, 
Motives (Seattle, WA, 1991), Proc. Sympos. Pure Math., vol. {\bf 55}, 
Amer. Math. Soc., Providence, RI, 1994, pp. 193--223.

\bibitem{GS20},
{Greenberg, Matthew and Seveso, Marco Adamo},
{Triple product {$p$}-adic {$L$}-functions for balanced
              weights\/},
{Math. Ann.}
{\bf 376} (2020)
103--176. 

\bibitem{HN}{T.~Hara} and {K.~Namikawa}, {A cohomological interpretation of archimedean zeta integrals for $\GL_3\times\GL_2$\/}. 

\bibitem{NHarris}
{R.~N.~Harris}, 
{The refined Gross-Prasad conjecture for unitary groups\/}, 
Int. Math. Res. Not., 2014 (2014), 303--389.

\bibitem{He}
{H.~He}, 
{On the Gan-Gross-Prasad conjecture for $U(p,q)$\/}, 
Invent. Math. {\bf 209} (2017), no. 3, 837--884. 


\bibitem{Hid93}{H.~Hida}, 
{Elementary Theory of $L$-Functions and Eisenstein Series\/}, 
London Math. Soc. Stud. Texts, vol. {\bf 26}, Cambridge University Press, Cambridge, 1993.

\bibitem{Hid04}{H.~Hida}, 
{$p$-adic automorphic forms on Shimura varieties\/}, 
Springer Monographs in Mathematics, Springer-Verlag, New York, 2004.

\bibitem{HIM}
{M.~Hirano}, {T.~Ishii} and {T.~Miyazaki}, 
{Archimedean zeta integrals for $\GL(3)\times\GL(2)$\/}, 
Memoirs of the American Mathematical Society {\bf 1366} (2022), 1--136. 

\bibitem{MH3}
{M.-L.~Hsieh},
{Ordinary $p$-adic Eisenstein series and $p$-adic $L$-functions for unitary groups\/}, 
Ann. Inst. Fourier (Grenoble) {\bf 61} (2011), no. 3, 987--1059, DOI 10.5802/aif.2635.

\bibitem{MH2}
{M.-L.~Hsieh}, 
{Eisenstein congruence on unitary groups and Iwasawa main conjectures for CM fields\/}, 
J. Amer. Math. Soc., {\bf 27}(3), 753--862, 2014

\bibitem{MH}
{M.-L.~Hsieh}, 
{Hida families and $p$-adic triple product $L$-functions\/}, 
Amer. J. Math. {\bf 143} (2021), no. 2, 411--532. 


\bibitem{II}
{A.~Ichino} and {T.~Ikeda}, 
{On the periods of automorphic forms on special orthogonal groups and the Gross-Prasad conjecture\/}, 
Geom. Funct. Anal. {\bf 19} no. 5 (2010), 1378--1425.

\bibitem{Ikeda}
{T.~Ikeda}, 
{On the lifting of hermitian modular forms\/}, 
Compos. Math. {\bf 144} (2008) 1107--1154.

\bibitem{JPSS}
{H.~Jacquet}, {I.~I.~Piatetski-Shapiro} and {J.~Shalika}, 
{Conducteur des repr\'{e}sentations du groupe lin\'{e}aire\/}, 
Math. Ann. {\bf 256} (1981), no. 2, 199--214.

\bibitem{JPSS2}
\bysame, 
{Rankin-Selberg convolutions\/}, 
Amer. J. Math. {\bf 105} 
(1983), 367--464.

\bibitem{JS2}{H.~Jacquet} and {J.~Shalika}, {On Euler products and the classification of automorphic representations. I\/}, Amer. J. Math. {\bf 103} (1981), no. 3, 499--558. 

\bibitem{JS}
{H.~Jacquet} and {J.~Shalika}, 
{A lemma on highly ramified $\vep$-factors\/}, 
Math. Ann. {\bf 271}
(1985), 319--332. 

\bibitem{KV}
{M.~Kashiwara} and {M.~Vergne}, 
{On the Segal-Shale-Weil representations and harmonic polynomials\/}, 
Invent. Math. {\bf 44} (1978), 1--47.

\bibitem{KK}
{T.~Konno and K.~Konno}, 
{On doubling construction for real unitary dual pairs\/}, 
Kyushu J. Math. {\bf 61} (2007), no. 1, 35--82. 

\bibitem{KS}
{S.~Kudla} and {W.~J.~Sweet Jr}, 
{Degenerate principal series representations for $\U(n, n)$\/}, 
Israel J. Math. {\bf 98} (1997), 253--306.

\bibitem{LM}
{E.~Lapid} and {Z.~Mao}, 
{On a new functional equation for local integrals\/}, 
In Automorphic Forms and Related Geometry: Assessing the Legacy of I.~I.~Piatetski-Shapiro, 261--94. Contemp. Math. {\bf 614}. Providence, RI: American Mathematical Society, 2014. doi:10.1090/conm/614/12271.



\bibitem{YL}
{Y.~Liu}, 
{Anticyclotomic $p$-adic $L$-functions for Rankin--Selberg product\/}, preprint

\bibitem{NM}
{N.~Matringe}, 
{Essential Whittaker functions for $\GL(n)$\/}, 
Doc. Math. {\bf 18} (2013), 1191--1214. 






\bibitem{MVW}
{C.~Moeglin}, {M.-F.~Vignera} and {C.-L.~Waldspurger},
{Correspondence de Howe sur un corps p-adique\/}, 
Springer Lec. notes in Math. 
{\bf 1291}, 
1987. 



\bibitem{Schmidt}
{C.-G.~Schmidt}, 
{Relative modular symbols and $p$-adic Rankin-Selberg convolutions\/}, 
Invent. Math. {\bf 112}
(1993), 31--76. 

\bibitem{Shimura2}
{G.~Shimura}, 
{Euler Products and Eisenstein Series\/}, 
CBMS Reg. Conf. Ser. Math., vol. {\bf 93}, Amer. Math. Soc., (1997)

\bibitem{Shimura}
{G.~Shimura}, 
{Arithmetic of hermitian forms\/}, 
Doc. Math. {\bf 13} (2008), 739--774.

\bibitem{W1}
{J.-L.~Waldspurger}, 
{La formule de Plancherel pour les groupes $p$-adiques (d'apr\`{e}s Harish-Chandra)}, 
J. Inst. Math. Jussieu {\bf 2}(2), (2003), 235--333 


\bibitem{Z}
{W.~Zhang}, 
{Automorphic period and the central value of Rankin-Selberg $L$-function\/}, 
J. Amer. Math. Soc., {\bf 27} 
(2014), 541--612.

\end{thebibliography}
\end{document}